\documentclass[11pt]{article}
\usepackage{etex}

\usepackage[margin=0.9in]{geometry}
\usepackage{stackengine}

\usepackage{showlabels}
\usepackage{stmaryrd}

\usepackage{color}
\usepackage[latin1]{inputenc}
\usepackage[T1]{fontenc}
\usepackage[normalem]{ulem}
\usepackage[english]{babel}
\usepackage{verbatim}
\usepackage{graphicx}
\usepackage{enumerate}
\usepackage{amsmath,amssymb,amsfonts,amsthm,amscd,mathrsfs}
\usepackage{array}

\usepackage{dsfont}

\usepackage[T1]{fontenc}
\usepackage{babel}

\usepackage{url}

\usepackage{caption}
\usepackage{subcaption}
\usepackage[bookmarksopen, bookmarksnumbered]{hyperref}

\usepackage{tikz}
\usetikzlibrary{calc}
\usetikzlibrary{arrows}
\usetikzlibrary{arrows.meta}
\definecolor{wwhhii}{rgb}{1.,1.,1.}
\definecolor{rreedd}{rgb}{1.,0.,0.}
\definecolor{uuuuuu}{rgb}{0.26666666666666666,0.26666666666666666,0.26666666666666666}
\definecolor{darkgreen}{HTML}{0d8513}

\newtheorem{theorem}{Theorem}[section]

\newtheorem{lemma}[theorem]{Lemma}

\newtheorem{prop}[theorem]{Proposition}
\newtheorem{cor}[theorem]{Corollary}

\newtheorem{question}[theorem]{Question}
\theoremstyle{definition}

\newtheorem{ex}[theorem]{Example}

\input xy
\xyoption{all}

\DeclareMathOperator{\argmax}{argmax}
\DeclareMathOperator{\Exp}{Exp}

\DeclareMathOperator{\LAL}{Left}

\DeclareMathOperator{\EG}{EG}
\DeclareMathOperator{\SN}{SN}
\DeclareMathOperator{\co}{cor}
\newcommand{\ocor}{\overline{\co}}
\DeclareMathOperator{\la}{last}
\newcommand{\ola}{\overline{\la}}
\DeclareMathOperator{\de}{deg}
\DeclareMathOperator{\SYT}{SYT}

\newcommand{\PP}{\mathbb P}
\newcommand{\Z}{\mathbb Z}
\newcommand{\R}{\mathbb R}
\newcommand{\C}{\mathbb C}

\newcommand{\N}{\mathbb N}

\newcommand{\don}{\mathds{1}}

\newcommand{\hg}{\hat{g}}

\newcommand{\cD}{\mathcal D}
\newcommand{\cP}{\mathcal P}

\newcommand{\cE}{\mathcal E}
\newcommand{\cA}{\mathcal A}
\newcommand{\cL}{\Lambda}

\newcommand{\cH}{\mathcal H}

\newcommand{\cS}{\mathcal S}
\newcommand{\cR}{\mathcal R}

\newcommand{\bU}{\mathbf U}
\newcommand{\bT}{\mathbf T}
\newcommand{\bt}{\mathbf t}
\newcommand{\bR}{\mathbf R}
\newcommand{\bL}{\mathbf L}
\newcommand{\bell}{\boldsymbol{\ell}}
\newcommand{\br}{\mathbf r}

\newcommand{\bx}{\mathbf x}
\newcommand{\bX}{\mathbf X}

\newcommand{\oi}{\overline{i}}
\newcommand{\oj}{\overline{j}}

\newcommand{\ol}{\overline{\ell}}
\newcommand{\oV}{\overline{V}}
\newcommand{\oor}{\overline{r}}

\newcommand{\oeta}{\overline{\eta}}
\newcommand{\obT}{\overline{\bT}}
\newcommand{\cbT}{\hat{\bT}}
\newcommand{\obt}{\overline{\bt}}

\newcommand{\obell}{\overline{\bell}}
\newcommand{\obr}{\overline{\br}}

\newcommand{\opi}{\mu}
\newcommand{\hmu}{\hat{\mu}}

\newcommand{\hkappa}{\overline{\kappa}}

\makeatletter
\renewcommand\tableofcontents{%
  \null\hfill\textbf{\Large\contentsname}\hfill\null\par
  \@mkboth{\MakeUppercase\contentsname}{\MakeUppercase\contentsname}%
  \@starttoc{toc}%
}

\g@addto@macro\normalsize{%
  \setlength\abovedisplayskip{5pt}
  \setlength\belowdisplayskip{5pt}
  \setlength\abovedisplayshortskip{3pt}
  \setlength\belowdisplayshortskip{3pt}
}

\makeatother

\numberwithin{equation}{section}

\newtheorem{innercustomthm}{Theorem}
\newenvironment{customthm}[1]
  {\renewcommand\theinnercustomthm{#1}\innercustomthm}
  {\endinnercustomthm}

\begin{document}
\title{Shift-Invariance of the Colored TASEP and Finishing Times of the Oriented Swap Process}

\author{Lingfu Zhang
\thanks{Department of Mathematics, Princeton University, e-mail: lingfuz@math.princeton.edu}
\thanks{Department of Statistics, UC Berkeley, e-mail: lfzhang@berkeley.edu}
}
\date{}

\maketitle

\begin{abstract}
We prove a new shift-invariance property of the colored TASEP. 
From the shift-invariance of the colored stochastic six-vertex model (proved in Borodin-Gorin-Wheeler or Galashin), one can get a shift-invariance property of the colored TASEP at one time, and our result generalizes this to multiple times.
Our proof takes the single-time shift-invariance as an input, and uses analyticity of the probability functions and induction arguments.
We apply our shift-invariance to prove a distributional identity between the finishing times of the oriented swap process and the point-to-line passage times in exponential last-passage percolation, which is conjectured by Bisi-Cunden-Gibbons-Romik and Bufetov-Gorin-Romik, and is also equivalent to a purely combinatorial identity related to the Edelman-Greene correspondence. With known results from last-passage percolation, we also get new asymptotic results on the colored TASEP and the finishing times of the oriented swap process.
\end{abstract}

\section{Introduction} \label{sec:intro}
The Totally Asymmetric Simple Exclusion Process (TASEP) is a classical interacting particle system, where one considers a collection of particles in $\Z$, such that each site contains at most one particle. There is an independent Poisson clock on each edge $(x, x+1)$, such that when it rings, if there is a particle at site $x$, and the site $x+1$ is empty (in other words, there is a hole at site $x+1$), then the particle jumps to site $x+1$.
See e.g.\ the book of Liggett \cite{liggett2012interacting} and references therein.

The colored TASEP is a variant of this model: there is a particle at each site in $\Z$, and each particle has an additional property called \emph{color}, which is usually integer-valued. A particle with a smaller color is considered `stronger' than a particle with a larger color: when the Poisson clock on edge $(x, x+1)$ rings, if the particles on sites $x$ and $x+1$ have colors $i$ and $j$ respectively, then they will swap if and only if $i<j$.
The uncolored TASEP can be viewed as a colored TASEP with two colors: particles with the larger color correspond to holes, and particles with the smaller color correspond to particles. 
Another degeneration of the colored TASEP is TASEP with second-class particles.
In these models, there are normal particles, second-class particles, and holes on $\Z$.
The rule is that a normal or second-class particle can jump to a hole right next to it, and a normal particle can swap with a second-class particle right next to it.
These models involving second-class particles have been proven powerful in understanding the evolution of the TASEP and related exactly solvable models
\cite{ferrari1991microscopic, ferrari1992shock,
derrida1993exact, speer1994two, balazs2006cube, balazs2010order, MSZ}.

In this paper, we consider the following colored TASEP, studied in e.g.\ \cite{angel2009oriented, amir2011tasep, bufetov2020interacting, bufetov2020shock, borodin2021color}. The particles have mutually different colors, and initially, the particle at each site $x$ is colored $x$.
The configuration of this process at any time $t\ge 0$ can be viewed as a bijection $\zeta_t:\Z\to\Z$, where $\zeta_t(x)$ is the color of the particle at site $x$ at time $t$. Then $\zeta_0$ is the identity map, and $t\mapsto \zeta_t(x)$ is a cadlag function for each $x\in\Z$.
There are several distributional identities for this process, due to natural symmetries.
For example, take any $y\in\Z$, the function $x\mapsto \zeta_t(x-y)+y$ has the same distribution as $\zeta_t$.
There is also a reflection symmetry, which says that $x\mapsto -\zeta_t(-x)$ has the same distribution as $\zeta_t$.
There are also some less obvious equalities in distribution or symmetries that have been proved previously. For example, in \cite{amir2011tasep}, it is proved that $\zeta_t$ and its inverse (as a permutation of $\Z$) have the same distribution.
A version of this symmetry in a finite interval is given in \cite{angel2009oriented}.
A more general color-to-position symmetry is also proved in \cite{borodin2021color} (and also \cite{bufetov2020interacting}).

We present a new shift-invariance property, concerning the multi-time distribution of this colored TASEP.
It is stated for the distribution of `passage times', a notion coming from the connection between TASEP and the directed last-passage percolation (LPP) with i.i.d. exponential weights, and the distribution of related objects such as the geodesics.
A main question on the colored TASEP is about the scaling limit of the passage times.
Our shift-invariance implies convergence of the colored TASEP passage times on some `ordered sets' to the Airy sheet, a universal 2D random process constructed in \cite{dauvergne2018directed} and proven to be the scaling limit of LPP and its variants in \cite{dauvergne2021scaling}.

Another major application of our result is to prove a conjectured identity between LPP and the oriented swap process (OSP), a type of random sorting network that can be viewed as the colored TASEP in a finite interval.
Using this identity, we can use results from LPP to deduce asymptotic results of OSP.
For example, we show that the vector of the finishing times of OSP convergences to the Airy$_2$ process, under the KPZ scaling; and the site where the last swap happens has fluctuation in the order of $N^{2/3}$ (for OSP with $N$ numbers). On a smaller scale, the finishing times converge to random walks in total variation distance.

Our shift-invariance is related to various other hidden invariance in exactly solvable models, from \cite{borodin2019shift, dauvergne2020hidden, galashin2020symmetries}, as will be discussed in Section \ref{ssec:dis}; however, the approach we take is quite different from these previous works. In particular, we avoid working with formulas from algebraic combinatorics or representation theory.
Our starting point is a crucial input of a shift-invariance property of the colored stochastic six-vertex model, proved in \cite{galashin2020symmetries}.
Our central idea is to use the analyticity of the distribution functions, combined with the independence properties of the colored TASEP in different space-time areas.
See Section \ref{sec:stra} for a more detailed explanation of our strategy, with a simple example.

\subsection{Passage times and last-passage percolation}

To state our results we start with the following setup.
For the colored TASEP starting from the identity map, it is also equivalent to the coupling of a family of (uncolored) TASEPs, starting with step initial conditions.
More precisely, for each $A\in \Z$, we consider a TASEP such that initially, there is a particle on each site $\le A$, and each site $>A$ is empty.
For time $t\ge 0$, we denote the configuration as $\opi^A_t:\Z \to \{0, \infty\}$, where $0$ denotes a particle and $\infty$ denotes a hole.
We couple $\opi^A=(\opi^A_t)_{t\ge 0}$ for all $A\in \Z$, such that the jump of particles follow the same Poisson clock.
Such $\opi^A$ can be obtained from $\zeta=(\zeta_t)_{t\ge 0}$: we just let $\opi^A_t(x)=0$ if $\zeta_t(x)\le A$, and $\opi^A_t(x)=\infty$ if $\zeta_t(x)> A$.
We can also recover $\zeta$ from such $\opi^A$ for all $A\in \Z$,
by letting $\zeta_t(x)=\min\{A\in\Z: \opi^A_t(x)=0\}$.

The evolution of the colored TASEP $\zeta$ can also be described by the following family of random variables, which we call \emph{passage times}.
For each $A\in\Z$ and $B,C\in \N$, denote \begin{align*}
T_{B,C}^A&=\inf\{t\ge 0: |\{x\ge A+B+1-C: \zeta_t(x)\le A\}|\ge C\} \\ &= \inf\{t\ge 0: |\{x\ge A+B+1-C: \opi^A_t(x)=0\}|\ge C\},
\end{align*}
i.e., $T_{B,C}^A$ is the first time, when in $\opi^A$  there are at least $C$ particles on or to the right of site $A+B+1-C$. 
In other words, $T_{B,C}^A$ is the time when the $B$-th leftmost hole swaps with the $C$-th rightmost particle in $\opi^A$.
From $\{T_{B,C}^A\}_{B,C\in\N}$ we can recover the evolution of $\opi^A$.
We take such random variables and notations partially due to the connection between (uncolored) TASEP and LPP, dating back to Rost \cite{rost1981non}. 
We now state this connection, and we start by formally defining the model of 2D LPP (with i.i.d.\ exponential weights).

To each vertex $v \in \Z^2$ we associate an independent weight $\omega(v)$ with $\Exp(1)$ distribution.
For two vertices $u, v \in \Z^2$, we say $u\leq v$ if $u$ is coordinate-wise less than or equal to $v$.
For such $u, v$ and any up-right path $\gamma$ from $u$ to $v$, we define the \emph{passage time of the path} to be
\[
L(\gamma) := \sum_{w \in \gamma} \omega(w) .
\]
Then almost surely there is a unique up-right path from $u$ to $v$ that has the largest passage time.
We will always assume such uniqueness in this paper, and we call this path the \emph{geodesic} $\Gamma_{u,v}$, and call $L_{u,v}:=L(\Gamma_{u,v})$ the \emph{passage time from $u$ to $v$}.

For any $A\in\Z$, we can couple $\opi^A$ with a family of i.i.d. $\Exp(1)$ random variables $\{\omega^A(v)\}_{v\in \N^2}$, where for any $B, C\in \N$, $\omega^A((B,C))$ is the waiting time for the $B$-th leftmost hole to swap with the $C$-th rightmost particle.
We note that at any time $t$, the $B$-th leftmost hole is to the right of the $C$-th rightmost particle, if and only if there are at least $C$ particles on or to the right of site $A+B+1-C$.
Thus for the swap (between the $B$-th leftmost hole and the $C$-th rightmost particle), one starts the `waiting' when there are $C-1$ particles on or to the right of site $A+B+2-C$, and another particle arrives at site $A+B-C$.
Thus we can formally write \[\omega^A((B,C)) = T^A_{B,C}-T^A_{B-1,C}\vee T^A_{B,C-1},\] where we assume that $T^A_{B-1,C}=0$ if $B=1$, and $T^A_{B,C-1}=0$ if $C=1$.
Such $\omega^A((B,C))$ are i.i.d. $\Exp(1)$ for $B, C \in \N$, since in $\opi^A$, the jumps happen with rate $1$ independently. 

Let $L^A_{u,v}$ denote the passage time from $u$ to $v$ under the random field $\omega^A$.
Then we have $L^A_{(1,1),(B,C)} = T^A_{B,C}$ for any $B,C\in\N$, which explains the name of `passage times' for $T^A_{B,C}$.

From this connection between LPP and TASEP, we get a coupling of random fields $\omega^A$ for different $A$.
We now give a more direct but slightly more involved description of this coupling.
From $\omega^A$ one can get $\omega^{A+1}$ in the following way.
Take a family of i.i.d.\ $\Exp(1)$ random variables $\{E_{i,j}\}_{i,j\in \N}$, independent of $\omega^A$.
We then recursively define a function $\pi:\Z_{\ge 0}\to\N$ and a sequence of times $\{J_i\}_{i=0}^\infty$, as follows.
Let $\pi(0)=1$ and $J_0 = 0$.
Given any $\pi(i)$ and $J_i$, we let \[\pi(i+1)=\inf\{j\ge \pi(i):  L^A_{(1,1),(i+1,j)}-L^A_{(1,1),(i+2,j-1)}\vee J_i \ge E_{i+1,j} \},\]
and $J_{i+1} = L^A_{(1,1),(i+2,\pi(i+1)-1)}\vee J_i + E_{i+1,\pi(i+1)}$.
Then for $B,C\in\N$, we let
\[
L^{A+1}_{(1,1),(B,C)} =
\begin{cases}
L^{A}_{(1,1),(B,C-1)},\quad & C>\pi(B), \\
L^{A}_{(1,1),(B+1,C)},\quad & C<\pi(B), \\
J_B,\quad & C=\pi(B).
\end{cases}
\]
From this we can get $\omega^{A+1}$ via $\omega^{A+1}_{B,C}=L^{A+1}_{(1,1),(B,C)} - L^{A+1}_{(1,1),(B-1,C)} \vee L^{A+1}_{(1,1),(B,C-1)}$. 
We note that the above random variables are defined from the colored TASEP: each $E_{i,j}$ corresponds to the waiting time for the $i$-th right jump of the particle colored $A+1$, when there are $j-1$ particles with smaller colors to its right; each $J_i$ corresponds to the time when the particle colored $A+1$ makes the $i$-th jump to the right, and $\pi(i)-1$ corresponds to the number of particles with smaller colors to its right at time $J_i$.

In other words, the function $\pi$ splits $\omega^A$ into two parts; and by shifting the upper left part by $(0,1)$ and shifting the lower right part by $(-1,0)$, we get $\omega^{A+1}$, except for some vertices around the boundary.
By repeating this procedure, we can get $\omega^{A'}$ for any $A' >A$; and by the reflection symmetry, we can do a similar procedure to get $\omega^{A'}$ for any $A' <A$.
More precisely, for any $k\in \N$, we can find non-decreasing function $\pi_1,\ldots, \pi_k:\N\to\N$, with $\pi_1<\cdots <\pi_k$, such that they split $\omega^A$ into $k+1$ parts.
For the $i$-th part from the top, we shift it by $(1-i, k+1-i)$, and we get $\omega^{A+k}$, except for vertices around the boundaries.

\subsection{Colored TASEP identities}  \label{ssec:colortasep-iden}

We now state our shift-invariance of $\zeta$, using a graphical representation.
For each $A\in\Z$ and $B,C\in \N$, let $\cR^A_{B,C}$ be the rectangle of lattice points $[1+A,B+A]\times [1-A,C-A] \cap \Z^2$. From the connection with LPP, $T^A_{B,C}$ can be thought of as the passage time from the bottom-left corner to the up-right corner of $\cR^A_{B,C}$, where the random field is $\omega^A$ shifted by $(A,-A)$.\\

\noindent\textbf{Ordering of rectangles.} 
We make the collection of all such rectangles a partially ordered set, as follows.
For two rectangles $\cR, \cR' \subset \Z^2$, we say $\cR\le \cR'$ if the projection of $\cR$ onto the first coordinate contains the projection of $\cR'$ onto the first coordinate, and the projection of $\cR$ onto the second coordinate is contained in the projection of $\cR'$ onto the second coordinate.
In other words, suppose $\cR=\cR^A_{B,C}$ and $\cR'=\cR^{A'}_{B',C'}$, then $\cR\le \cR'$ if and only if $A\le A'$ and $A+B\ge A'+B'$, $A-C\ge A'-C'$.
We say that $\cR$ and $\cR'$ are \emph{ordered} if $\cR\le \cR'$ or $\cR'\le \cR$.
This ordering might look artificial, but it is crucial for our shift-invariance. See the discussions in Section \ref{ssec:shift-cons}.
\\

Our main result states that, if certain ordering relations are preserved, one can shift these rectangles while the joint distribution of the passage times is invariant.
\begin{theorem}  \label{thm:main-de}
Let $g\in\N$, $k_1,\ldots, k_g\in\N$, and take $A_{i,j}\in \Z$, $B_{i,j}, C_{i,j} \in\N$, for all $1\le i \le g$ and $1\le j \le k_i$.
Let $1\le \iota < g$, and for any $1\le i\le g$ and $1\le j \le k_i$ we let $A_{i,j}^+ = A_{i,j}+\don[i> \iota]$. Suppose that for any $1\le i<i' \le g$, and $1\le j \le k_i, 1\le j' \le k_{i'}$ we have $\cR^{A_{i,j}}_{B_{i,j},C_{i,j}} \le \cR^{A_{i',j'}}_{B_{i',j'},C_{i',j'}}$and $\cR^{A_{i,j}^+}_{B_{i,j},C_{i,j}} \le \cR^{A_{i',j'}^+}_{B_{i',j'},C_{i',j'}}$.
Then the vectors
$\{\max_{1\le j \le k_i} T_{B_{i,j},C_{i,j}}^{A_{i,j}} \}_{i=1}^g$ and $\{\max_{1\le j \le k_i} T_{B_{i,j},C_{i,j}}^{A_{i,j}^+} \}_{i=1}^g$ are equal in distribution.
\end{theorem}
In words, we can take a collection of passage times and divide them into groups, and shift the color parameters (i.e., $A_{i,j}$) for some groups.
As long as some ordering across groups always holds, the joint distribution of the maximums in different groups remains unchanged.

We next demonstrate what can be obtained from this theorem, via the following example.
\begin{ex}  \label{ex:parshift}
Consider the passage times $T^{-2}_{14,7}$, $T^{-1}_{15,6}$, $T^{0}_{10,11}$, $T^{3}_{5,14}$, $T^{3}_{4,15}$.
Each one is respectively the smallest time, such that
\begin{enumerate}
    \item $7$ particles with colors $\le -2$ on or to the right of site $6$, 
    \item $6$ particles with colors $\le -1$ on or to the right of site $9$, 
    \item $11$ particles with colors $\le 0$ on or to the right of site $0$, 
    \item $14$ particles with colors $\le 3$ on or to the right of site $-5$, 
    \item $15$ particles with colors $\le 3$ on or to the right of site $-7$.
\end{enumerate}
By Theorem \ref{thm:main-de}, $\max\{T^{-2}_{14,7}, T^{-1}_{15,6}\}$, $T^{0}_{10,11}$, $\max\{T^{3}_{5,14}, T^{3}_{4,15}\}$ have the same joint distribution as \[\max\{T^{-2}_{14,7}, T^{-1}_{15,6}\}, T^{1}_{10,11}, \max\{T^{4}_{5,14}, T^{4}_{4,15}\}.\] By applying Theorem \ref{thm:main-de} again, they also have the same joint distribution as
\[\max\{T^{-2}_{14,7}, T^{-1}_{15,6}\}, T^{1}_{10,11}, \max\{T^{2}_{5,14}, T^{2}_{4,15}\}.\]
See Figure \ref{fig:thm1par} for illustrations of these passage times, and Figure \ref{fig:thm1} for visualizations of the same using LPP.
\end{ex}

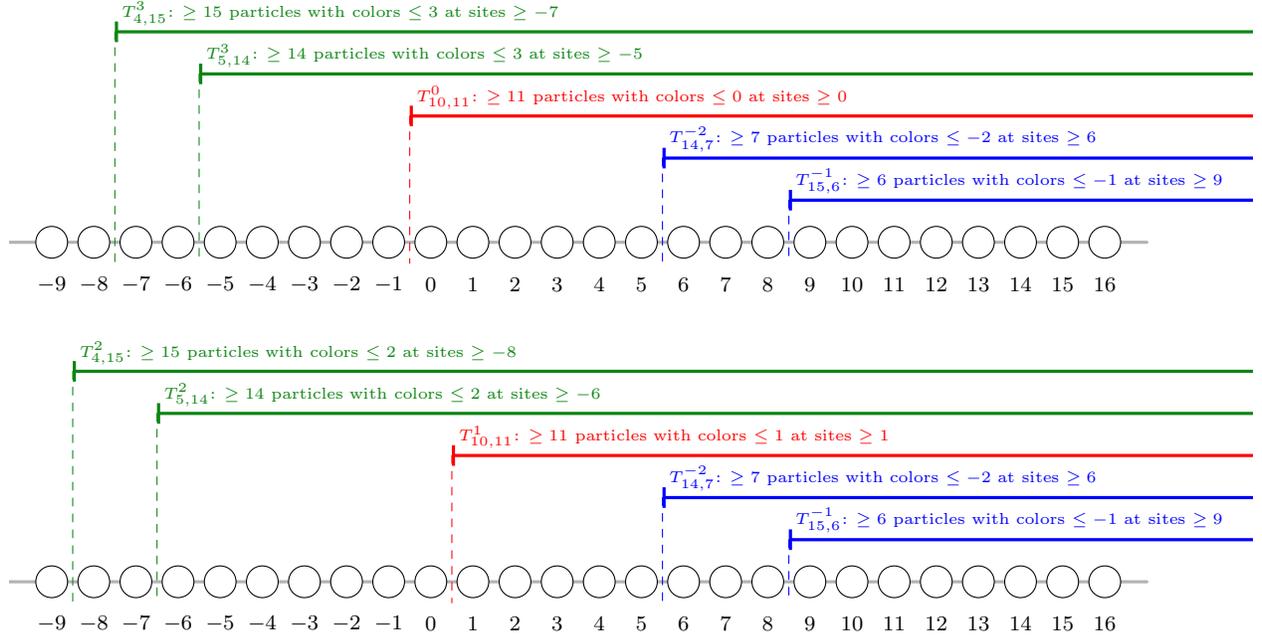
\begin{figure}[hbt!]
    \centering
\begin{subfigure}[b]{\textwidth}
         \centering
    \begin{tikzpicture}[line cap=round,line join=round,>=triangle 45,x=5.6cm,y=5.6cm]
\clip(-3,0.) rectangle (-0.05,0.8);

\draw [line width=1.2pt, opacity=0.3] (-3,0.2) -- (-0.3,0.2);
\draw [fill=white] (-2.9,0.2) circle (6.0pt);
\draw [fill=white] (-2.8,0.2) circle (6.0pt);
\draw [fill=white] (-2.7,0.2) circle (6.0pt);
\draw [fill=white] (-2.6,0.2) circle (6.0pt);
\draw [fill=white] (-2.5,0.2) circle (6.0pt);
\draw [fill=white] (-2.4,0.2) circle (6.0pt);
\draw [fill=white] (-2.3,0.2) circle (6.0pt);
\draw [fill=white] (-2.2,0.2) circle (6.0pt);
\draw [fill=white] (-2.1,0.2) circle (6.0pt);
\draw [fill=white] (-2.,0.2) circle (6.0pt);
\draw [fill=white] (-1.9,0.2) circle (6.0pt);
\draw [fill=white] (-1.8,0.2) circle (6.0pt);
\draw [fill=white] (-1.7,0.2) circle (6.0pt);
\draw [fill=white] (-1.6,0.2) circle (6.0pt);
\draw [fill=white] (-1.5,0.2) circle (6.0pt);
\draw [fill=white] (-1.4,0.2) circle (6.0pt);
\draw [fill=white] (-1.3,0.2) circle (6.0pt);
\draw [fill=white] (-1.2,0.2) circle (6.0pt);
\draw [fill=white] (-1.1,0.2) circle (6.0pt);
\draw [fill=white] (-1.,0.2) circle (6.0pt);
\draw [fill=white] (-0.9,0.2) circle (6.0pt);
\draw [fill=white] (-0.8,0.2) circle (6.0pt);
\draw [fill=white] (-0.7,0.2) circle (6.0pt);
\draw [fill=white] (-0.6,0.2) circle (6.0pt);
\draw [fill=white] (-0.5,0.2) circle (6.0pt);
\draw [fill=white] (-0.4,0.2) circle (6.0pt);

\draw [red] [very thick] [|-](-2.05,0.5) -- (10,0.5);

\draw [blue] [very thick] [|-](-1.45,0.4) -- (10,0.4);
\draw [blue] [very thick] [|-](-1.15,0.3) -- (10,0.3);

\draw [darkgreen] [very thick] [|-](-2.75,0.7) -- (10,0.7);
\draw [darkgreen] [very thick] [|-](-2.55,0.6) -- (10,0.6);

\draw [red] [dashed](-2.05,0.5) -- (-2.05,0.15);

\draw [blue] [dashed](-1.45,0.4) -- (-1.45,0.15);
\draw [blue] [dashed](-1.15,0.3) -- (-1.15,0.15);

\draw [darkgreen] [dashed](-2.75,0.7) -- (-2.75,0.15);
\draw [darkgreen] [dashed](-2.55,0.6) -- (-2.55,0.15);

\begin{scriptsize}
\draw (-2.9,0.1) node[anchor=center]{$-9$};
\draw (-2.8,0.1) node[anchor=center]{$-8$};
\draw (-2.7,0.1) node[anchor=center]{$-7$};
\draw (-2.6,0.1) node[anchor=center]{$-6$};
\draw (-2.5,0.1) node[anchor=center]{$-5$};
\draw (-2.4,0.1) node[anchor=center]{$-4$};
\draw (-2.3,0.1) node[anchor=center]{$-3$};
\draw (-2.2,0.1) node[anchor=center]{$-2$};
\draw (-2.1,0.1) node[anchor=center]{$-1$};
\draw (-2.0,0.1) node[anchor=center]{$0$};
\draw (-1.9,0.1) node[anchor=center]{$1$};
\draw (-1.8,0.1) node[anchor=center]{$2$};
\draw (-1.7,0.1) node[anchor=center]{$3$};
\draw (-1.6,0.1) node[anchor=center]{$4$};
\draw (-1.5,0.1) node[anchor=center]{$5$};
\draw (-1.4,0.1) node[anchor=center]{$6$};
\draw (-1.3,0.1) node[anchor=center]{$7$};
\draw (-1.2,0.1) node[anchor=center]{$8$};
\draw (-1.1,0.1) node[anchor=center]{$9$};
\draw (-1.0,0.1) node[anchor=center]{$10$};
\draw (-0.9,0.1) node[anchor=center]{$11$};
\draw (-0.8,0.1) node[anchor=center]{$12$};
\draw (-0.7,0.1) node[anchor=center]{$13$};
\draw (-0.6,0.1) node[anchor=center]{$14$};
\draw (-0.5,0.1) node[anchor=center]{$15$};
\draw (-0.4,0.1) node[anchor=center]{$16$};
\end{scriptsize}
\begin{tiny}
\draw [red] (-2.05,0.5) node[anchor=south west]{$T^0_{10,11}$: $\ge 11$ particles with colors $\le 0$ at sites $\ge 0$};
\draw [blue] (-1.45,0.4) node[anchor=south west]{$T^{-2}_{14,7}$: $\ge 7$ particles with colors $\le -2$ at sites $\ge 6$};
\draw [blue] (-1.15,0.3) node[anchor=south west]{$T^{-1}_{15,6}$: $\ge 6$ particles with colors $\le -1$ at sites $\ge 9$};

\draw [darkgreen] (-2.75,0.7) node[anchor=south west]{$T^{3}_{4,15}$: $\ge 15$ particles with colors $\le 3$ at sites $\ge -7$};
\draw [darkgreen] (-2.55,0.6) node[anchor=south west]{$T^{3}_{5,14}$: $\ge 14$ particles with colors $\le 3$ at sites $\ge -5$};
\end{tiny}

\draw (-4,0.6) node[anchor=west]{$\zeta_t$};
\draw (-4,0.4) node[anchor=west]{$\hmu_t^{1,8}\Longleftrightarrow \opi_t^{-6,1}, \opi_t^{-5,2}, \ldots, \opi_t^{1,8}$}; 
\draw (-4,0.2) node[anchor=west]{$\hmu_t^{0,5}\Longleftrightarrow \opi_t^{-4,1}, \opi_t^{-3,2}, \ldots, \opi_t^{0,5}$};

\end{tikzpicture}
\end{subfigure}
\begin{subfigure}[b]{\textwidth}
         \centering
    \begin{tikzpicture}[line cap=round,line join=round,>=triangle 45,x=5.6cm,y=5.6cm]
\clip(-3,0.) rectangle (-0.05,0.8);

\draw [line width=1.2pt, opacity=0.3] (-3,0.2) -- (-0.3,0.2);
\draw [fill=white] (-2.9,0.2) circle (6.0pt);
\draw [fill=white] (-2.8,0.2) circle (6.0pt);
\draw [fill=white] (-2.7,0.2) circle (6.0pt);
\draw [fill=white] (-2.6,0.2) circle (6.0pt);
\draw [fill=white] (-2.5,0.2) circle (6.0pt);
\draw [fill=white] (-2.4,0.2) circle (6.0pt);
\draw [fill=white] (-2.3,0.2) circle (6.0pt);
\draw [fill=white] (-2.2,0.2) circle (6.0pt);
\draw [fill=white] (-2.1,0.2) circle (6.0pt);
\draw [fill=white] (-2.,0.2) circle (6.0pt);
\draw [fill=white] (-1.9,0.2) circle (6.0pt);
\draw [fill=white] (-1.8,0.2) circle (6.0pt);
\draw [fill=white] (-1.7,0.2) circle (6.0pt);
\draw [fill=white] (-1.6,0.2) circle (6.0pt);
\draw [fill=white] (-1.5,0.2) circle (6.0pt);
\draw [fill=white] (-1.4,0.2) circle (6.0pt);
\draw [fill=white] (-1.3,0.2) circle (6.0pt);
\draw [fill=white] (-1.2,0.2) circle (6.0pt);
\draw [fill=white] (-1.1,0.2) circle (6.0pt);
\draw [fill=white] (-1.,0.2) circle (6.0pt);
\draw [fill=white] (-0.9,0.2) circle (6.0pt);
\draw [fill=white] (-0.8,0.2) circle (6.0pt);
\draw [fill=white] (-0.7,0.2) circle (6.0pt);
\draw [fill=white] (-0.6,0.2) circle (6.0pt);
\draw [fill=white] (-0.5,0.2) circle (6.0pt);
\draw [fill=white] (-0.4,0.2) circle (6.0pt);

\draw [red] [very thick] [|-](-1.95,0.5) -- (10,0.5);

\draw [blue] [very thick] [|-](-1.45,0.4) -- (10,0.4);
\draw [blue] [very thick] [|-](-1.15,0.3) -- (10,0.3);

\draw [darkgreen] [very thick] [|-](-2.85,0.7) -- (10,0.7);
\draw [darkgreen] [very thick] [|-](-2.65,0.6) -- (10,0.6);

\draw [red] [dashed](-1.95,0.5) -- (-1.95,0.15);

\draw [blue] [dashed](-1.45,0.4) -- (-1.45,0.15);
\draw [blue] [dashed](-1.15,0.3) -- (-1.15,0.15);

\draw [darkgreen] [dashed](-2.85,0.7) -- (-2.85,0.15);
\draw [darkgreen] [dashed](-2.65,0.6) -- (-2.65,0.15);

\begin{scriptsize}
\draw (-2.9,0.1) node[anchor=center]{$-9$};
\draw (-2.8,0.1) node[anchor=center]{$-8$};
\draw (-2.7,0.1) node[anchor=center]{$-7$};
\draw (-2.6,0.1) node[anchor=center]{$-6$};
\draw (-2.5,0.1) node[anchor=center]{$-5$};
\draw (-2.4,0.1) node[anchor=center]{$-4$};
\draw (-2.3,0.1) node[anchor=center]{$-3$};
\draw (-2.2,0.1) node[anchor=center]{$-2$};
\draw (-2.1,0.1) node[anchor=center]{$-1$};
\draw (-2.0,0.1) node[anchor=center]{$0$};
\draw (-1.9,0.1) node[anchor=center]{$1$};
\draw (-1.8,0.1) node[anchor=center]{$2$};
\draw (-1.7,0.1) node[anchor=center]{$3$};
\draw (-1.6,0.1) node[anchor=center]{$4$};
\draw (-1.5,0.1) node[anchor=center]{$5$};
\draw (-1.4,0.1) node[anchor=center]{$6$};
\draw (-1.3,0.1) node[anchor=center]{$7$};
\draw (-1.2,0.1) node[anchor=center]{$8$};
\draw (-1.1,0.1) node[anchor=center]{$9$};
\draw (-1.0,0.1) node[anchor=center]{$10$};
\draw (-0.9,0.1) node[anchor=center]{$11$};
\draw (-0.8,0.1) node[anchor=center]{$12$};
\draw (-0.7,0.1) node[anchor=center]{$13$};
\draw (-0.6,0.1) node[anchor=center]{$14$};
\draw (-0.5,0.1) node[anchor=center]{$15$};
\draw (-0.4,0.1) node[anchor=center]{$16$};
\end{scriptsize}
\begin{tiny}
\draw [red] (-1.95,0.5) node[anchor=south west]{$T^1_{10,11}$: $\ge 11$ particles with colors $\le 1$ at sites $\ge 1$};
\draw [blue] (-1.45,0.4) node[anchor=south west]{$T^{-2}_{14,7}$: $\ge 7$ particles with colors $\le -2$ at sites $\ge 6$};
\draw [blue] (-1.15,0.3) node[anchor=south west]{$T^{-1}_{15,6}$: $\ge 6$ particles with colors $\le -1$ at sites $\ge 9$};

\draw [darkgreen] (-2.85,0.7) node[anchor=south west]{$T^{2}_{4,15}$: $\ge 15$ particles with colors $\le 2$ at sites $\ge -8$};
\draw [darkgreen] (-2.65,0.6) node[anchor=south west]{$T^{2}_{5,14}$: $\ge 14$ particles with colors $\le 2$ at sites $\ge -6$};
\end{tiny}

\draw (-4,0.6) node[anchor=west]{$\zeta_t$};
\draw (-4,0.4) node[anchor=west]{$\hmu_t^{1,8}\Longleftrightarrow \opi_t^{-6,1}, \opi_t^{-5,2}, \ldots, \opi_t^{1,8}$}; 
\draw (-4,0.2) node[anchor=west]{$\hmu_t^{0,5}\Longleftrightarrow \opi_t^{-4,1}, \opi_t^{-3,2}, \ldots, \opi_t^{0,5}$};

\end{tikzpicture}
\end{subfigure}

\caption{
Consider (a) the first time when both the green events happen, (b) the first time when the red event happens, and (c) the first time when both the blue events happen. Their joint distributions are the same in the top and bottom panels.
This can be proved by applying Theorem \ref{thm:main-de} twice.
}  
\label{fig:thm1par}
\end{figure}

\begin{figure}[hbt!]
    \centering

\begin{subfigure}[b]{0.48\textwidth}
         \centering

\begin{tikzpicture}[line cap=round,line join=round,>=triangle 45,x=4.cm,y=4.cm]
\clip(-.5,-0.2) rectangle (1.19,1.59);

\fill[line width=0.pt,color=green,fill=green,fill opacity=0.15]
(0.05,1.25) -- (0.05,-0.15) -- (0.55,-0.15) -- (0.55,1.25) -- cycle;
\fill[line width=0.pt,color=green,fill=green,fill opacity=0.15]
(0.05,1.35) -- (0.05,-0.15) -- (0.45,-0.15) -- (0.45,1.35) -- cycle;

\fill[line width=0.pt,color=blue,fill=blue,fill opacity=0.15]
(-0.45,0.35) -- (-0.45,1.05) -- (0.95,1.05) -- (0.95,0.35) -- cycle;
\fill[line width=0.pt,color=blue,fill=blue,fill opacity=0.15]
(-0.35,0.25) -- (-0.35,0.85) -- (1.15,0.85) -- (1.15,0.25) -- cycle;

\fill[line width=0.pt,color=red,fill=red,fill opacity=0.15]
(-0.25,0.15) -- (-0.25,1.25) -- (0.75,1.25) -- (0.75,0.15) -- cycle;

\draw [line width=.1pt, opacity=0.3] (-0.5,-0.1) -- (2.6,-0.1);
\draw [line width=.1pt, opacity=0.3] (-0.5,0.) -- (2.6,0.);
\draw [line width=.1pt, opacity=0.3] (-0.5,0.1) -- (2.6,0.1);
\draw [line width=.1pt, opacity=0.3] (-0.5,0.2) -- (2.6,0.2);
\draw [line width=.1pt, opacity=0.3] (-0.5,0.3) -- (2.6,0.3);
\draw [line width=.1pt, opacity=0.3] (-0.5,0.4) -- (2.6,0.4);
\draw [line width=.1pt, opacity=0.3] (-0.5,0.5) -- (2.6,0.5);
\draw [line width=.1pt, opacity=0.3] (-0.5,0.6) -- (2.6,0.6);
\draw [line width=.1pt, opacity=0.3] (-0.5,0.7) -- (2.6,0.7);
\draw [line width=.1pt, opacity=0.3] (-0.5,0.8) -- (2.6,0.8);
\draw [line width=.1pt, opacity=0.3] (-0.5,0.9) -- (2.6,0.9);
\draw [line width=.1pt, opacity=0.3] (-0.5,1.) -- (2.6,1.);
\draw [line width=.1pt, opacity=0.3] (-0.5,1.1) -- (2.6,1.1);
\draw [line width=.1pt, opacity=0.3] (-0.5,1.2) -- (2.6,1.2);
\draw [line width=.1pt, opacity=0.3] (-0.5,1.3) -- (2.6,1.3);
\draw [line width=.1pt, opacity=0.3] (-0.5,1.4) -- (2.6,1.4);
\draw [line width=.1pt, opacity=0.3] (-0.5,1.5) -- (2.6,1.5);
\draw [line width=.1pt, opacity=0.3] (-0.5,1.6) -- (2.6,1.6);
\draw [line width=.1pt, opacity=0.3] (-0.5,1.7) -- (2.6,1.7);
\draw [line width=.1pt, opacity=0.3] (-0.5,1.8) -- (2.6,1.8);
\draw [line width=.1pt, opacity=0.3] (-0.5,1.9) -- (2.6,1.9);
\draw [line width=.1pt, opacity=0.3] (-0.5,2.) -- (2.6,2.);
\draw [line width=.1pt, opacity=0.3] (-0.5,2.1) -- (2.6,2.1);
\draw [line width=.1pt, opacity=0.3] (-0.4,-0.2) -- (-0.4,2.2);
\draw [line width=.1pt, opacity=0.3] (-0.3,-0.2) -- (-0.3,2.2);
\draw [line width=.1pt, opacity=0.3] (-0.2,-0.2) -- (-0.2,2.2);
\draw [line width=.1pt, opacity=0.3] (-0.1,-0.2) -- (-0.1,2.2);
\draw [line width=.1pt, opacity=0.3] (0.,-0.2) -- (0.,2.2);
\draw [line width=.1pt, opacity=0.3] (0.1,-0.2) -- (0.1,2.2);
\draw [line width=.1pt, opacity=0.3] (0.2,-0.2) -- (0.2,2.2);
\draw [line width=.1pt, opacity=0.3] (0.3,-0.2) -- (0.3,2.2);
\draw [line width=.1pt, opacity=0.3] (0.4,-0.2) -- (0.4,2.2);
\draw [line width=.1pt, opacity=0.3] (0.5,-0.2) -- (0.5,2.2);
\draw [line width=.1pt, opacity=0.3] (0.6,-0.2) -- (0.6,2.2);
\draw [line width=.1pt, opacity=0.3] (0.7,-0.2) -- (0.7,2.2);
\draw [line width=.1pt, opacity=0.3] (0.8,-0.2) -- (0.8,2.2);
\draw [line width=.1pt, opacity=0.3] (0.9,-0.2) -- (0.9,2.2);
\draw [line width=.1pt, opacity=0.3] (1.,-0.2) -- (1.,2.2);
\draw [line width=.1pt, opacity=0.3] (1.1,-0.2) -- (1.1,2.2);
\draw [line width=.1pt, opacity=0.3] (1.2,-0.2) -- (1.2,2.2);
\draw [line width=.1pt, opacity=0.3] (1.3,-0.2) -- (1.3,2.2);
\draw [line width=.1pt, opacity=0.3] (1.4,-0.2) -- (1.4,2.2);
\draw [line width=.1pt, opacity=0.3] (1.5,-0.2) -- (1.5,2.2);
\draw [line width=.1pt, opacity=0.3] (1.6,-0.2) -- (1.6,2.2);
\draw [line width=.1pt, opacity=0.3] (1.7,-0.2) -- (1.7,2.2);
\draw [line width=.1pt, opacity=0.3] (1.8,-0.2) -- (1.8,2.2);
\draw [line width=.1pt, opacity=0.3] (1.9,-0.2) -- (1.9,2.2);
\draw [line width=.1pt, opacity=0.3] (2.,-0.2) -- (2.,2.2);
\draw [line width=.1pt, opacity=0.3] (2.1,-0.2) -- (2.1,2.2);
\draw [line width=.1pt, opacity=0.3] (2.2,-0.2) -- (2.2,2.2);
\draw [line width=.1pt, opacity=0.3] (2.3,-0.2) -- (2.3,2.2);
\draw [line width=.1pt, opacity=0.3] (2.4,-0.2) -- (2.4,2.2);

\draw [uuuuuu] plot coordinates {(1.,-1.) (-1.,1.) };

\draw [darkgreen] plot coordinates {(0.1,-0.1) (0.4,1.3) };
\draw [darkgreen] plot coordinates {(0.1,-0.1) (0.5,1.2) };
\draw [fill=green] (0.1,-0.1) circle (1.5pt);
\draw [fill=green] (0.4,1.3) circle (1.5pt);
\draw [fill=green] (0.5,1.2) circle (1.5pt);

\draw [red] plot coordinates {(-0.2,0.2) (0.7,1.2) };
\draw [fill=red] (-0.2,0.2) circle (1.5pt);
\draw [fill=red] (0.7,1.2) circle (1.5pt);

\draw [blue] plot coordinates {(-0.3,0.3) (1.1,0.8) };
\draw [blue] plot coordinates {(-0.4,0.4) (0.9,1.0) };
\draw [fill=blue] (-0.3,0.3) circle (1.5pt);
\draw [fill=blue] (1.1,0.8) circle (1.5pt);
\draw [fill=blue] (-0.4,0.4) circle (1.5pt);
\draw [fill=blue] (0.9,1.0) circle (1.5pt);

\begin{scriptsize}
\draw (0.,0.) node[anchor=east, color=uuuuuu]{$x+y=2$};
\foreach \i in {-2,-1,...,14}
{
\draw node[anchor=west] at  (-0.5,\i/10+0.1) {\i};
}
\foreach \i in {-1,0,...,14}
{
\draw node[anchor=north] at  (\i/10-0.3,-0.1) {\i};
}
\end{scriptsize}

\end{tikzpicture}
\end{subfigure}
\hfill
\begin{subfigure}[b]{0.5\textwidth}
         \centering

\begin{tikzpicture}[line cap=round,line join=round,>=triangle 45,x=4.cm,y=4.cm]
\clip(-.5,-0.2) rectangle (1.19,1.59);

\fill[line width=0.pt,color=green,fill=green,fill opacity=0.15]
(-0.05,1.35) -- (-0.05,-0.05) -- (0.45,-0.05) -- (0.45,1.35) -- cycle;
\fill[line width=0.pt,color=green,fill=green,fill opacity=0.15]
(-0.05,1.45) -- (-0.05,-0.05) -- (0.35,-0.05) -- (0.35,1.45) -- cycle;

\fill[line width=0.pt,color=blue,fill=blue,fill opacity=0.15]
(-0.45,0.35) -- (-0.45,1.05) -- (0.95,1.05) -- (0.95,0.35) -- cycle;
\fill[line width=0.pt,color=blue,fill=blue,fill opacity=0.15]
(-0.35,0.25) -- (-0.35,0.85) -- (1.15,0.85) -- (1.15,0.25) -- cycle;

\fill[line width=0.pt,color=red,fill=red,fill opacity=0.15]
(-0.15,0.05) -- (-0.15,1.15) -- (0.85,1.15) -- (0.85,0.05) -- cycle;

\draw [line width=.1pt, opacity=0.3] (-0.5,-0.1) -- (2.6,-0.1);
\draw [line width=.1pt, opacity=0.3] (-0.5,0.) -- (2.6,0.);
\draw [line width=.1pt, opacity=0.3] (-0.5,0.1) -- (2.6,0.1);
\draw [line width=.1pt, opacity=0.3] (-0.5,0.2) -- (2.6,0.2);
\draw [line width=.1pt, opacity=0.3] (-0.5,0.3) -- (2.6,0.3);
\draw [line width=.1pt, opacity=0.3] (-0.5,0.4) -- (2.6,0.4);
\draw [line width=.1pt, opacity=0.3] (-0.5,0.5) -- (2.6,0.5);
\draw [line width=.1pt, opacity=0.3] (-0.5,0.6) -- (2.6,0.6);
\draw [line width=.1pt, opacity=0.3] (-0.5,0.7) -- (2.6,0.7);
\draw [line width=.1pt, opacity=0.3] (-0.5,0.8) -- (2.6,0.8);
\draw [line width=.1pt, opacity=0.3] (-0.5,0.9) -- (2.6,0.9);
\draw [line width=.1pt, opacity=0.3] (-0.5,1.) -- (2.6,1.);
\draw [line width=.1pt, opacity=0.3] (-0.5,1.1) -- (2.6,1.1);
\draw [line width=.1pt, opacity=0.3] (-0.5,1.2) -- (2.6,1.2);
\draw [line width=.1pt, opacity=0.3] (-0.5,1.3) -- (2.6,1.3);
\draw [line width=.1pt, opacity=0.3] (-0.5,1.4) -- (2.6,1.4);
\draw [line width=.1pt, opacity=0.3] (-0.5,1.5) -- (2.6,1.5);
\draw [line width=.1pt, opacity=0.3] (-0.5,1.6) -- (2.6,1.6);
\draw [line width=.1pt, opacity=0.3] (-0.5,1.7) -- (2.6,1.7);
\draw [line width=.1pt, opacity=0.3] (-0.5,1.8) -- (2.6,1.8);
\draw [line width=.1pt, opacity=0.3] (-0.5,1.9) -- (2.6,1.9);
\draw [line width=.1pt, opacity=0.3] (-0.5,2.) -- (2.6,2.);
\draw [line width=.1pt, opacity=0.3] (-0.5,2.1) -- (2.6,2.1);
\draw [line width=.1pt, opacity=0.3] (-0.4,-0.2) -- (-0.4,2.2);
\draw [line width=.1pt, opacity=0.3] (-0.3,-0.2) -- (-0.3,2.2);
\draw [line width=.1pt, opacity=0.3] (-0.2,-0.2) -- (-0.2,2.2);
\draw [line width=.1pt, opacity=0.3] (-0.1,-0.2) -- (-0.1,2.2);
\draw [line width=.1pt, opacity=0.3] (0.,-0.2) -- (0.,2.2);
\draw [line width=.1pt, opacity=0.3] (0.1,-0.2) -- (0.1,2.2);
\draw [line width=.1pt, opacity=0.3] (0.2,-0.2) -- (0.2,2.2);
\draw [line width=.1pt, opacity=0.3] (0.3,-0.2) -- (0.3,2.2);
\draw [line width=.1pt, opacity=0.3] (0.4,-0.2) -- (0.4,2.2);
\draw [line width=.1pt, opacity=0.3] (0.5,-0.2) -- (0.5,2.2);
\draw [line width=.1pt, opacity=0.3] (0.6,-0.2) -- (0.6,2.2);
\draw [line width=.1pt, opacity=0.3] (0.7,-0.2) -- (0.7,2.2);
\draw [line width=.1pt, opacity=0.3] (0.8,-0.2) -- (0.8,2.2);
\draw [line width=.1pt, opacity=0.3] (0.9,-0.2) -- (0.9,2.2);
\draw [line width=.1pt, opacity=0.3] (1.,-0.2) -- (1.,2.2);
\draw [line width=.1pt, opacity=0.3] (1.1,-0.2) -- (1.1,2.2);
\draw [line width=.1pt, opacity=0.3] (1.2,-0.2) -- (1.2,2.2);
\draw [line width=.1pt, opacity=0.3] (1.3,-0.2) -- (1.3,2.2);
\draw [line width=.1pt, opacity=0.3] (1.4,-0.2) -- (1.4,2.2);
\draw [line width=.1pt, opacity=0.3] (1.5,-0.2) -- (1.5,2.2);
\draw [line width=.1pt, opacity=0.3] (1.6,-0.2) -- (1.6,2.2);
\draw [line width=.1pt, opacity=0.3] (1.7,-0.2) -- (1.7,2.2);
\draw [line width=.1pt, opacity=0.3] (1.8,-0.2) -- (1.8,2.2);
\draw [line width=.1pt, opacity=0.3] (1.9,-0.2) -- (1.9,2.2);
\draw [line width=.1pt, opacity=0.3] (2.,-0.2) -- (2.,2.2);
\draw [line width=.1pt, opacity=0.3] (2.1,-0.2) -- (2.1,2.2);
\draw [line width=.1pt, opacity=0.3] (2.2,-0.2) -- (2.2,2.2);
\draw [line width=.1pt, opacity=0.3] (2.3,-0.2) -- (2.3,2.2);
\draw [line width=.1pt, opacity=0.3] (2.4,-0.2) -- (2.4,2.2);

\draw [uuuuuu] plot coordinates {(1.,-1.) (-1.,1.) };

\draw [darkgreen] plot coordinates {(0.,0.) (0.3,1.4) };
\draw [darkgreen] plot coordinates {(0.,0.) (0.4,1.3) };
\draw [fill=green] (0.,0.) circle (1.5pt);
\draw [fill=green] (0.3,1.4) circle (1.5pt);
\draw [fill=green] (0.4,1.3) circle (1.5pt);

\draw [red] plot coordinates {(-0.1,0.1) (0.8,1.1) };
\draw [fill=red] (-0.1,0.1) circle (1.5pt);
\draw [fill=red] (0.8,1.1) circle (1.5pt);

\draw [blue] plot coordinates {(-0.3,0.3) (1.1,0.8) };
\draw [blue] plot coordinates {(-0.4,0.4) (0.9,1.0) };
\draw [fill=blue] (-0.3,0.3) circle (1.5pt);
\draw [fill=blue] (1.1,0.8) circle (1.5pt);
\draw [fill=blue] (-0.4,0.4) circle (1.5pt);
\draw [fill=blue] (0.9,1.0) circle (1.5pt);

\begin{scriptsize}
\draw (0.,0.) node[anchor=east, color=uuuuuu]{$x+y=2$};
\foreach \i in {-2,-1,...,14}
{
\draw node[anchor=west] at  (-0.5,\i/10+0.1) {\i};
}
\foreach \i in {-1,0,...,14}
{
\draw node[anchor=north] at  (\i/10-0.3,-0.1) {\i};
}
\end{scriptsize}

\end{tikzpicture}

\end{subfigure}

\caption{
The passage times in Figure \ref{fig:thm1par} are visualized using LPP. Note that each $T^A_{B,C}$ is both (1) the first time for $C$ particles with colors $\le A$ to be on or to the right of site $A+B+1-C$, and
(2) the passage time from $(1+A, 1-A)$ to $(B+A, C-A)$ (which are corners of the rectangle $\cR^A_{B,C}$).
The left panel corresponds to the top panel in Figure \ref{fig:thm1par}, and the right panel corresponds to the bottom panel in Figure \ref{fig:thm1par}.
By Theorem \ref{thm:main-de}, the maximum of the blue passage times, the green passage times, and the red passage time have the same joint distributions, in both panels.
Note that this is not the shift-invariance of LPP, since for each starting point $(1+A,1-A)$, the last-passage time is for a different set of the random field $\omega^A$, associated with the TASEP $\opi^A$.
}  
\label{fig:thm1}
\end{figure}
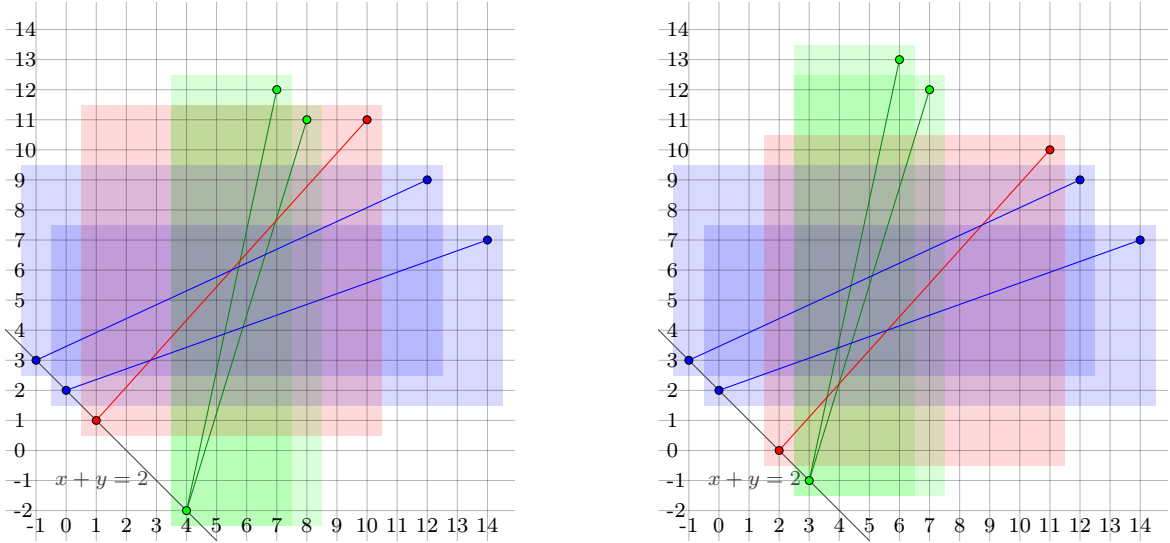

We also state the shift-invariance in another setting,
with some slightly different constraints on the parameters.
\begin{theorem}  \label{thm:main-sa}
Let $g\in\N$, and $A_i,A_i'\in\Z$, $B_i, C_i\in\N$ for each $1\le i \le g$.
Let $V_i$ be the rectangle consisting of all $(B,C)\in [1,B_i]\times[1,C_i]\cap\Z^2$, such that $\cR^{A_j}_{B_j,C_j}$ and $\cR^{A_i}_{B,C}$ are ordered, $\cR^{A_j'}_{B_j,C_j}$ and $\cR^{A_i'}_{B,C}$ are ordered, for each $1\le j \le g$.
Then $\{ \{T^{A_i}_{B,C}\}_{(B,C)\in V_i} \}_{i=1}^g$ has the same distribution as $\{ \{T^{A_i'}_{B,C}\}_{(B,C)\in V_i} \}_{i=1}^g$.
\end{theorem}
From these identities, we can get invariance for some related objects. For example, in LPP one can locally construct geodesics using passage times, thus we immediately get the following equality in distribution for geodesics.

For each $A\in\Z$ and $B,C\in\N$, we let $\Gamma^A_{B,C}$ be the LPP geodesic from $(1,1)$ to $(B,C)$, with the random field $\omega^A$ given by $\opi^A$.
\begin{cor}  \label{cor:main-geo}
In the setting of Theorem \ref{thm:main-sa}, let $W_i\subset V_i$ consist of all $(B,C) \in \Z^2$, such that $\{(B-1,C), (B,C-1)\}\cap \N^2 \subset V_i$.
Then $\{\Gamma^{A_i}_{B_i,C_i}\cap W_i\}_{i=1}^g$ has the same distribution as $\{\Gamma^{A_i'}_{B_i,C_i}\cap W_i\}_{i=1}^g$.
\end{cor}
See Figure \ref{fig:thm2} for an illustration of Theorem \ref{thm:main-sa} and Corollary \ref{cor:main-geo}.

\begin{figure}[hbt!]
    \centering
\begin{subfigure}[b]{0.48\textwidth}
         \centering

\begin{tikzpicture}[line cap=round,line join=round,>=triangle 45,x=4.cm,y=4.cm]
\clip(-.5,-0.2) rectangle (1.19,1.49);

\fill[line width=0.pt,color=green,fill=green,fill opacity=0.15]
(0.05,1.35) -- (0.05,-0.15) -- (0.45,-0.15) -- (0.45,1.35) -- cycle;

\fill[line width=0.pt,color=blue,fill=blue,fill opacity=0.15]
(-0.45,0.35) -- (-0.45,0.95) -- (1.15,0.95) -- (1.15,0.35) -- cycle;

\fill[line width=0.pt,color=red,fill=red,fill opacity=0.15]
(-0.25,0.15) -- (-0.25,1.15) -- (0.75,1.15) -- (0.75,0.15) -- cycle;

\draw [line width=.1pt, opacity=0.3] (-0.5,-0.1) -- (2.6,-0.1);
\draw [line width=.1pt, opacity=0.3] (-0.5,0.) -- (2.6,0.);
\draw [line width=.1pt, opacity=0.3] (-0.5,0.1) -- (2.6,0.1);
\draw [line width=.1pt, opacity=0.3] (-0.5,0.2) -- (2.6,0.2);
\draw [line width=.1pt, opacity=0.3] (-0.5,0.3) -- (2.6,0.3);
\draw [line width=.1pt, opacity=0.3] (-0.5,0.4) -- (2.6,0.4);
\draw [line width=.1pt, opacity=0.3] (-0.5,0.5) -- (2.6,0.5);
\draw [line width=.1pt, opacity=0.3] (-0.5,0.6) -- (2.6,0.6);
\draw [line width=.1pt, opacity=0.3] (-0.5,0.7) -- (2.6,0.7);
\draw [line width=.1pt, opacity=0.3] (-0.5,0.8) -- (2.6,0.8);
\draw [line width=.1pt, opacity=0.3] (-0.5,0.9) -- (2.6,0.9);
\draw [line width=.1pt, opacity=0.3] (-0.5,1.) -- (2.6,1.);
\draw [line width=.1pt, opacity=0.3] (-0.5,1.1) -- (2.6,1.1);
\draw [line width=.1pt, opacity=0.3] (-0.5,1.2) -- (2.6,1.2);
\draw [line width=.1pt, opacity=0.3] (-0.5,1.3) -- (2.6,1.3);
\draw [line width=.1pt, opacity=0.3] (-0.5,1.4) -- (2.6,1.4);
\draw [line width=.1pt, opacity=0.3] (-0.5,1.5) -- (2.6,1.5);
\draw [line width=.1pt, opacity=0.3] (-0.5,1.6) -- (2.6,1.6);
\draw [line width=.1pt, opacity=0.3] (-0.5,1.7) -- (2.6,1.7);
\draw [line width=.1pt, opacity=0.3] (-0.5,1.8) -- (2.6,1.8);
\draw [line width=.1pt, opacity=0.3] (-0.5,1.9) -- (2.6,1.9);
\draw [line width=.1pt, opacity=0.3] (-0.5,2.) -- (2.6,2.);
\draw [line width=.1pt, opacity=0.3] (-0.5,2.1) -- (2.6,2.1);
\draw [line width=.1pt, opacity=0.3] (-0.4,-0.2) -- (-0.4,2.2);
\draw [line width=.1pt, opacity=0.3] (-0.3,-0.2) -- (-0.3,2.2);
\draw [line width=.1pt, opacity=0.3] (-0.2,-0.2) -- (-0.2,2.2);
\draw [line width=.1pt, opacity=0.3] (-0.1,-0.2) -- (-0.1,2.2);
\draw [line width=.1pt, opacity=0.3] (0.,-0.2) -- (0.,2.2);
\draw [line width=.1pt, opacity=0.3] (0.1,-0.2) -- (0.1,2.2);
\draw [line width=.1pt, opacity=0.3] (0.2,-0.2) -- (0.2,2.2);
\draw [line width=.1pt, opacity=0.3] (0.3,-0.2) -- (0.3,2.2);
\draw [line width=.1pt, opacity=0.3] (0.4,-0.2) -- (0.4,2.2);
\draw [line width=.1pt, opacity=0.3] (0.5,-0.2) -- (0.5,2.2);
\draw [line width=.1pt, opacity=0.3] (0.6,-0.2) -- (0.6,2.2);
\draw [line width=.1pt, opacity=0.3] (0.7,-0.2) -- (0.7,2.2);
\draw [line width=.1pt, opacity=0.3] (0.8,-0.2) -- (0.8,2.2);
\draw [line width=.1pt, opacity=0.3] (0.9,-0.2) -- (0.9,2.2);
\draw [line width=.1pt, opacity=0.3] (1.,-0.2) -- (1.,2.2);
\draw [line width=.1pt, opacity=0.3] (1.1,-0.2) -- (1.1,2.2);
\draw [line width=.1pt, opacity=0.3] (1.2,-0.2) -- (1.2,2.2);
\draw [line width=.1pt, opacity=0.3] (1.3,-0.2) -- (1.3,2.2);
\draw [line width=.1pt, opacity=0.3] (1.4,-0.2) -- (1.4,2.2);
\draw [line width=.1pt, opacity=0.3] (1.5,-0.2) -- (1.5,2.2);
\draw [line width=.1pt, opacity=0.3] (1.6,-0.2) -- (1.6,2.2);
\draw [line width=.1pt, opacity=0.3] (1.7,-0.2) -- (1.7,2.2);
\draw [line width=.1pt, opacity=0.3] (1.8,-0.2) -- (1.8,2.2);
\draw [line width=.1pt, opacity=0.3] (1.9,-0.2) -- (1.9,2.2);
\draw [line width=.1pt, opacity=0.3] (2.,-0.2) -- (2.,2.2);
\draw [line width=.1pt, opacity=0.3] (2.1,-0.2) -- (2.1,2.2);
\draw [line width=.1pt, opacity=0.3] (2.2,-0.2) -- (2.2,2.2);
\draw [line width=.1pt, opacity=0.3] (2.3,-0.2) -- (2.3,2.2);
\draw [line width=.1pt, opacity=0.3] (2.4,-0.2) -- (2.4,2.2);

\draw [uuuuuu] plot coordinates {(1.,-1.) (-1.,1.) };

\draw [darkgreen] plot coordinates {(0.1,-0.1) (0.4,1.3) };
\draw [darkgreen] plot coordinates {(0.1,-0.1) (0.4,1.2) };
\draw [darkgreen] plot coordinates {(0.1,-0.1) (0.4,1.1) };
\draw [darkgreen] plot coordinates {(0.1,-0.1) (0.3,1.3) };
\draw [darkgreen] plot coordinates {(0.1,-0.1) (0.3,1.2) };
\draw [darkgreen] plot coordinates {(0.1,-0.1) (0.3,1.1) };
\draw [darkgreen] plot coordinates {(0.1,-0.1) (0.2,1.3) };
\draw [darkgreen] plot coordinates {(0.1,-0.1) (0.2,1.2) };
\draw [darkgreen] plot coordinates {(0.1,-0.1) (0.2,1.1) };
\draw [darkgreen] plot coordinates {(0.1,-0.1) (0.1,1.3) };
\draw [darkgreen] plot coordinates {(0.1,-0.1) (0.1,1.2) };
\draw [darkgreen] plot coordinates {(0.1,-0.1) (0.1,1.1) };
\draw [fill=green] (0.1,-0.1) circle (1.5pt);
\draw [fill=green] (0.4,1.3) circle (1.5pt);
\draw [fill=green] (0.4,1.2) circle (1.5pt);
\draw [fill=green] (0.4,1.1) circle (1.5pt);
\draw [fill=green] (0.3,1.3) circle (1.5pt);
\draw [fill=green] (0.3,1.2) circle (1.5pt);
\draw [fill=green] (0.3,1.1) circle (1.5pt);
\draw [fill=green] (0.2,1.3) circle (1.5pt);
\draw [fill=green] (0.2,1.2) circle (1.5pt);
\draw [fill=green] (0.2,1.1) circle (1.5pt);
\draw [fill=green] (0.1,1.3) circle (1.5pt);
\draw [fill=green] (0.1,1.2) circle (1.5pt);
\draw [fill=green] (0.1,1.1) circle (1.5pt);

\draw [red] plot coordinates {(-0.2,0.2) (0.7,1.1) };
\draw [red] plot coordinates {(-0.2,0.2) (0.6,1.1) };
\draw [red] plot coordinates {(-0.2,0.2) (0.5,1.1) };
\draw [red] plot coordinates {(-0.2,0.2) (0.4,1.1) };
\draw [red] plot coordinates {(-0.2,0.2) (0.7,1.) };
\draw [red] plot coordinates {(-0.2,0.2) (0.6,1.) };
\draw [red] plot coordinates {(-0.2,0.2) (0.5,1.) };
\draw [red] plot coordinates {(-0.2,0.2) (0.4,1.) };

\draw [fill=red] (-0.2,0.2) circle (1.5pt);
\draw [fill=red] (0.7,1.1) circle (1.5pt);
\draw [fill=red] (0.7,1.0) circle (1.5pt);
\draw [fill=red] (0.6,1.1) circle (1.5pt);
\draw [fill=red] (0.6,1.0) circle (1.5pt);
\draw [fill=red] (0.5,1.1) circle (1.5pt);
\draw [fill=red] (0.5,1.0) circle (1.5pt);
\draw [fill=uuuuuu] (0.4,1.1) circle (1.5pt);
\draw [fill=red] (0.4,1.0) circle (1.5pt);

\draw [blue] plot coordinates {(-0.4,0.4) (1.1,0.9) };
\draw [blue] plot coordinates {(-0.4,0.4) (1.0,0.9) };
\draw [blue] plot coordinates {(-0.4,0.4) (0.9,0.9) };
\draw [blue] plot coordinates {(-0.4,0.4) (0.8,0.9) };
\draw [blue] plot coordinates {(-0.4,0.4) (1.1,0.8) };
\draw [blue] plot coordinates {(-0.4,0.4) (1.0,0.8) };
\draw [blue] plot coordinates {(-0.4,0.4) (0.9,0.8) };
\draw [blue] plot coordinates {(-0.4,0.4) (0.8,0.8) };
\draw [blue] plot coordinates {(-0.4,0.4) (1.1,0.7) };
\draw [blue] plot coordinates {(-0.4,0.4) (1.0,0.7) };
\draw [blue] plot coordinates {(-0.4,0.4) (0.9,0.7) };
\draw [blue] plot coordinates {(-0.4,0.4) (0.8,0.7) };
\draw [blue] plot coordinates {(-0.4,0.4) (1.1,0.6) };
\draw [blue] plot coordinates {(-0.4,0.4) (1.0,0.6) };
\draw [blue] plot coordinates {(-0.4,0.4) (0.9,0.6) };
\draw [blue] plot coordinates {(-0.4,0.4) (0.8,0.6) };
\draw [blue] plot coordinates {(-0.4,0.4) (1.1,0.5) };
\draw [blue] plot coordinates {(-0.4,0.4) (1.0,0.5) };
\draw [blue] plot coordinates {(-0.4,0.4) (0.9,0.5) };
\draw [blue] plot coordinates {(-0.4,0.4) (0.8,0.5) };
\draw [blue] plot coordinates {(-0.4,0.4) (1.1,0.4) };
\draw [blue] plot coordinates {(-0.4,0.4) (1.0,0.4) };
\draw [blue] plot coordinates {(-0.4,0.4) (0.9,0.4) };
\draw [blue] plot coordinates {(-0.4,0.4) (0.8,0.4) };

\draw [fill=blue] (-0.4,0.4) circle (1.5pt);
\draw [fill=blue] (1.1,0.9) circle (1.5pt);
\draw [fill=blue] (1.0,0.9) circle (1.5pt);
\draw [fill=blue] (0.9,0.9) circle (1.5pt);
\draw [fill=blue] (0.8,0.9) circle (1.5pt);
\draw [fill=blue] (1.1,0.8) circle (1.5pt);
\draw [fill=blue] (1.0,0.8) circle (1.5pt);
\draw [fill=blue] (0.9,0.8) circle (1.5pt);
\draw [fill=blue] (0.8,0.8) circle (1.5pt);
\draw [fill=blue] (1.1,0.7) circle (1.5pt);
\draw [fill=blue] (1.0,0.7) circle (1.5pt);
\draw [fill=blue] (0.9,0.7) circle (1.5pt);
\draw [fill=blue] (0.8,0.7) circle (1.5pt);
\draw [fill=blue] (1.1,0.6) circle (1.5pt);
\draw [fill=blue] (1.0,0.6) circle (1.5pt);
\draw [fill=blue] (0.9,0.6) circle (1.5pt);
\draw [fill=blue] (0.8,0.6) circle (1.5pt);
\draw [fill=blue] (1.1,0.5) circle (1.5pt);
\draw [fill=blue] (1.0,0.5) circle (1.5pt);
\draw [fill=blue] (0.9,0.5) circle (1.5pt);
\draw [fill=blue] (0.8,0.5) circle (1.5pt);
\draw [fill=blue] (1.1,0.4) circle (1.5pt);
\draw [fill=blue] (1.0,0.4) circle (1.5pt);
\draw [fill=blue] (0.9,0.4) circle (1.5pt);
\draw [fill=blue] (0.8,0.4) circle (1.5pt);

\draw (0.,0.) node[anchor=east, color=uuuuuu]{$x+y=2$};

\draw [line width=2pt, red] plot coordinates {(0.7,1.1) (0.5,1.1) (0.5,1.0)};

\draw [line width=2pt, blue] plot coordinates {(1.1,0.9) (1.1,0.8) (0.9,0.8) (0.9,0.7) (0.8,0.7)};

\draw [line width=2pt, green] plot coordinates {(0.4,1.3) (0.4,1.2) (0.2,1.2) (0.2,1.1)};

\end{tikzpicture}
\end{subfigure}
\hfill
\begin{subfigure}[b]{0.5\textwidth}
         \centering

\begin{tikzpicture}[line cap=round,line join=round,>=triangle 45,x=4.cm,y=4.cm]
\clip(-.5,-0.2) rectangle (1.19,1.49);

\fill[line width=0.pt,color=green,fill=green,fill opacity=0.15]
(-0.05,1.45) -- (-0.05,-0.05) -- (0.35,-0.05) -- (0.35,1.45) -- cycle;

\fill[line width=0.pt,color=blue,fill=blue,fill opacity=0.15]
(-0.45,0.35) -- (-0.45,0.95) -- (1.15,0.95) -- (1.15,0.35) -- cycle;

\fill[line width=0.pt,color=red,fill=red,fill opacity=0.15]
(-0.15,0.05) -- (-0.15,1.05) -- (0.85,1.05) -- (0.85,0.05) -- cycle;

\draw [line width=.1pt, opacity=0.3] (-0.5,-0.1) -- (2.6,-0.1);
\draw [line width=.1pt, opacity=0.3] (-0.5,0.) -- (2.6,0.);
\draw [line width=.1pt, opacity=0.3] (-0.5,0.1) -- (2.6,0.1);
\draw [line width=.1pt, opacity=0.3] (-0.5,0.2) -- (2.6,0.2);
\draw [line width=.1pt, opacity=0.3] (-0.5,0.3) -- (2.6,0.3);
\draw [line width=.1pt, opacity=0.3] (-0.5,0.4) -- (2.6,0.4);
\draw [line width=.1pt, opacity=0.3] (-0.5,0.5) -- (2.6,0.5);
\draw [line width=.1pt, opacity=0.3] (-0.5,0.6) -- (2.6,0.6);
\draw [line width=.1pt, opacity=0.3] (-0.5,0.7) -- (2.6,0.7);
\draw [line width=.1pt, opacity=0.3] (-0.5,0.8) -- (2.6,0.8);
\draw [line width=.1pt, opacity=0.3] (-0.5,0.9) -- (2.6,0.9);
\draw [line width=.1pt, opacity=0.3] (-0.5,1.) -- (2.6,1.);
\draw [line width=.1pt, opacity=0.3] (-0.5,1.1) -- (2.6,1.1);
\draw [line width=.1pt, opacity=0.3] (-0.5,1.2) -- (2.6,1.2);
\draw [line width=.1pt, opacity=0.3] (-0.5,1.3) -- (2.6,1.3);
\draw [line width=.1pt, opacity=0.3] (-0.5,1.4) -- (2.6,1.4);
\draw [line width=.1pt, opacity=0.3] (-0.5,1.5) -- (2.6,1.5);
\draw [line width=.1pt, opacity=0.3] (-0.5,1.6) -- (2.6,1.6);
\draw [line width=.1pt, opacity=0.3] (-0.5,1.7) -- (2.6,1.7);
\draw [line width=.1pt, opacity=0.3] (-0.5,1.8) -- (2.6,1.8);
\draw [line width=.1pt, opacity=0.3] (-0.5,1.9) -- (2.6,1.9);
\draw [line width=.1pt, opacity=0.3] (-0.5,2.) -- (2.6,2.);
\draw [line width=.1pt, opacity=0.3] (-0.5,2.1) -- (2.6,2.1);
\draw [line width=.1pt, opacity=0.3] (-0.4,-0.2) -- (-0.4,2.2);
\draw [line width=.1pt, opacity=0.3] (-0.3,-0.2) -- (-0.3,2.2);
\draw [line width=.1pt, opacity=0.3] (-0.2,-0.2) -- (-0.2,2.2);
\draw [line width=.1pt, opacity=0.3] (-0.1,-0.2) -- (-0.1,2.2);
\draw [line width=.1pt, opacity=0.3] (0.,-0.2) -- (0.,2.2);
\draw [line width=.1pt, opacity=0.3] (0.1,-0.2) -- (0.1,2.2);
\draw [line width=.1pt, opacity=0.3] (0.2,-0.2) -- (0.2,2.2);
\draw [line width=.1pt, opacity=0.3] (0.3,-0.2) -- (0.3,2.2);
\draw [line width=.1pt, opacity=0.3] (0.4,-0.2) -- (0.4,2.2);
\draw [line width=.1pt, opacity=0.3] (0.5,-0.2) -- (0.5,2.2);
\draw [line width=.1pt, opacity=0.3] (0.6,-0.2) -- (0.6,2.2);
\draw [line width=.1pt, opacity=0.3] (0.7,-0.2) -- (0.7,2.2);
\draw [line width=.1pt, opacity=0.3] (0.8,-0.2) -- (0.8,2.2);
\draw [line width=.1pt, opacity=0.3] (0.9,-0.2) -- (0.9,2.2);
\draw [line width=.1pt, opacity=0.3] (1.,-0.2) -- (1.,2.2);
\draw [line width=.1pt, opacity=0.3] (1.1,-0.2) -- (1.1,2.2);
\draw [line width=.1pt, opacity=0.3] (1.2,-0.2) -- (1.2,2.2);
\draw [line width=.1pt, opacity=0.3] (1.3,-0.2) -- (1.3,2.2);
\draw [line width=.1pt, opacity=0.3] (1.4,-0.2) -- (1.4,2.2);
\draw [line width=.1pt, opacity=0.3] (1.5,-0.2) -- (1.5,2.2);
\draw [line width=.1pt, opacity=0.3] (1.6,-0.2) -- (1.6,2.2);
\draw [line width=.1pt, opacity=0.3] (1.7,-0.2) -- (1.7,2.2);
\draw [line width=.1pt, opacity=0.3] (1.8,-0.2) -- (1.8,2.2);
\draw [line width=.1pt, opacity=0.3] (1.9,-0.2) -- (1.9,2.2);
\draw [line width=.1pt, opacity=0.3] (2.,-0.2) -- (2.,2.2);
\draw [line width=.1pt, opacity=0.3] (2.1,-0.2) -- (2.1,2.2);
\draw [line width=.1pt, opacity=0.3] (2.2,-0.2) -- (2.2,2.2);
\draw [line width=.1pt, opacity=0.3] (2.3,-0.2) -- (2.3,2.2);
\draw [line width=.1pt, opacity=0.3] (2.4,-0.2) -- (2.4,2.2);

\draw [uuuuuu] plot coordinates {(1.,-1.) (-1.,1.) };

\draw [darkgreen] plot coordinates {(0,0) (0.,1.3) };
\draw [darkgreen] plot coordinates {(0,0) (0.,1.2) };
\draw [darkgreen] plot coordinates {(0,0) (0.,1.4) };
\draw [darkgreen] plot coordinates {(0,0) (0.3,1.3) };
\draw [darkgreen] plot coordinates {(0,0) (0.3,1.2) };
\draw [darkgreen] plot coordinates {(0,0) (0.3,1.4) };
\draw [darkgreen] plot coordinates {(0,0) (0.2,1.3) };
\draw [darkgreen] plot coordinates {(0,0) (0.2,1.2) };
\draw [darkgreen] plot coordinates {(0,0) (0.2,1.4) };
\draw [darkgreen] plot coordinates {(0,0) (0.1,1.3) };
\draw [darkgreen] plot coordinates {(0,0) (0.1,1.2) };
\draw [darkgreen] plot coordinates {(0,0) (0.1,1.4) };
\draw [fill=green] (0,0) circle (1.5pt);
\draw [fill=green] (0.,1.3) circle (1.5pt);
\draw [fill=green] (0.,1.2) circle (1.5pt);
\draw [fill=green] (0.,1.4) circle (1.5pt);
\draw [fill=green] (0.3,1.3) circle (1.5pt);
\draw [fill=green] (0.3,1.2) circle (1.5pt);
\draw [fill=green] (0.3,1.4) circle (1.5pt);
\draw [fill=green] (0.2,1.3) circle (1.5pt);
\draw [fill=green] (0.2,1.2) circle (1.5pt);
\draw [fill=green] (0.2,1.4) circle (1.5pt);
\draw [fill=green] (0.1,1.3) circle (1.5pt);
\draw [fill=green] (0.1,1.2) circle (1.5pt);
\draw [fill=green] (0.1,1.4) circle (1.5pt);

\draw [red] plot coordinates {(-0.1,0.1) (0.7,.9) };
\draw [red] plot coordinates {(-0.1,0.1) (0.6,.9) };
\draw [red] plot coordinates {(-0.1,0.1) (0.5,.9) };
\draw [red] plot coordinates {(-0.1,0.1) (0.8,.9) };
\draw [red] plot coordinates {(-0.1,0.1) (0.7,1.) };
\draw [red] plot coordinates {(-0.1,0.1) (0.6,1.) };
\draw [red] plot coordinates {(-0.1,0.1) (0.5,1.) };
\draw [red] plot coordinates {(-0.1,0.1) (0.8,1.) };

\draw [fill=red] (-0.1,0.1) circle (1.5pt);
\draw [fill=red] (0.7,.9) circle (1.5pt);
\draw [fill=red] (0.7,1.0) circle (1.5pt);
\draw [fill=red] (0.6,.9) circle (1.5pt);
\draw [fill=red] (0.6,1.0) circle (1.5pt);
\draw [fill=red] (0.5,.9) circle (1.5pt);
\draw [fill=red] (0.5,1.0) circle (1.5pt);
\draw [fill=red] (0.8,.9) circle (1.5pt);
\draw [fill=red] (0.8,1.0) circle (1.5pt);

\draw [blue] plot coordinates {(-0.4,0.4) (1.1,0.9) };
\draw [blue] plot coordinates {(-0.4,0.4) (1.0,0.9) };
\draw [blue] plot coordinates {(-0.4,0.4) (0.9,0.9) };
\draw [blue] plot coordinates {(-0.4,0.4) (0.8,0.9) };
\draw [blue] plot coordinates {(-0.4,0.4) (1.1,0.8) };
\draw [blue] plot coordinates {(-0.4,0.4) (1.0,0.8) };
\draw [blue] plot coordinates {(-0.4,0.4) (0.9,0.8) };
\draw [blue] plot coordinates {(-0.4,0.4) (0.8,0.8) };
\draw [blue] plot coordinates {(-0.4,0.4) (1.1,0.7) };
\draw [blue] plot coordinates {(-0.4,0.4) (1.0,0.7) };
\draw [blue] plot coordinates {(-0.4,0.4) (0.9,0.7) };
\draw [blue] plot coordinates {(-0.4,0.4) (0.8,0.7) };
\draw [blue] plot coordinates {(-0.4,0.4) (1.1,0.6) };
\draw [blue] plot coordinates {(-0.4,0.4) (1.0,0.6) };
\draw [blue] plot coordinates {(-0.4,0.4) (0.9,0.6) };
\draw [blue] plot coordinates {(-0.4,0.4) (0.8,0.6) };
\draw [blue] plot coordinates {(-0.4,0.4) (1.1,0.5) };
\draw [blue] plot coordinates {(-0.4,0.4) (1.0,0.5) };
\draw [blue] plot coordinates {(-0.4,0.4) (0.9,0.5) };
\draw [blue] plot coordinates {(-0.4,0.4) (0.8,0.5) };
\draw [blue] plot coordinates {(-0.4,0.4) (1.1,0.4) };
\draw [blue] plot coordinates {(-0.4,0.4) (1.0,0.4) };
\draw [blue] plot coordinates {(-0.4,0.4) (0.9,0.4) };
\draw [blue] plot coordinates {(-0.4,0.4) (0.8,0.4) };

\draw [fill=blue] (-0.4,0.4) circle (1.5pt);
\draw [fill=blue] (1.1,0.9) circle (1.5pt);
\draw [fill=blue] (1.0,0.9) circle (1.5pt);
\draw [fill=blue] (0.9,0.9) circle (1.5pt);
\draw [fill=uuuuuu] (0.8,0.9) circle (1.5pt);
\draw [fill=blue] (1.1,0.8) circle (1.5pt);
\draw [fill=blue] (1.0,0.8) circle (1.5pt);
\draw [fill=blue] (0.9,0.8) circle (1.5pt);
\draw [fill=blue] (0.8,0.8) circle (1.5pt);
\draw [fill=blue] (1.1,0.7) circle (1.5pt);
\draw [fill=blue] (1.0,0.7) circle (1.5pt);
\draw [fill=blue] (0.9,0.7) circle (1.5pt);
\draw [fill=blue] (0.8,0.7) circle (1.5pt);
\draw [fill=blue] (1.1,0.6) circle (1.5pt);
\draw [fill=blue] (1.0,0.6) circle (1.5pt);
\draw [fill=blue] (0.9,0.6) circle (1.5pt);
\draw [fill=blue] (0.8,0.6) circle (1.5pt);
\draw [fill=blue] (1.1,0.5) circle (1.5pt);
\draw [fill=blue] (1.0,0.5) circle (1.5pt);
\draw [fill=blue] (0.9,0.5) circle (1.5pt);
\draw [fill=blue] (0.8,0.5) circle (1.5pt);
\draw [fill=blue] (1.1,0.4) circle (1.5pt);
\draw [fill=blue] (1.0,0.4) circle (1.5pt);
\draw [fill=blue] (0.9,0.4) circle (1.5pt);
\draw [fill=blue] (0.8,0.4) circle (1.5pt);

\draw (0.,0.) node[anchor=east, color=uuuuuu]{$x+y=2$};

\draw [line width=2pt, red] plot coordinates {(0.8,1.0) (0.6,1.0) (0.6,0.9)};

\draw [line width=2pt, blue] plot coordinates {(1.1,0.9) (1.1,0.8) (0.9,0.8) (0.9,0.7) (0.8,0.7)};

\draw [line width=2pt, green] plot coordinates {(0.3,1.4) (0.3,1.3) (0.1,1.3) (0.1,1.2)};

\end{tikzpicture}

\end{subfigure}

\caption{
An illustration of Theorem \ref{thm:main-sa} and Corollary \ref{cor:main-geo}: the green, red, and blue passage times, and the shape of part of the geodesics, are all equal in distribution, in the left and the right panels.
}  
\label{fig:thm2}
\end{figure}
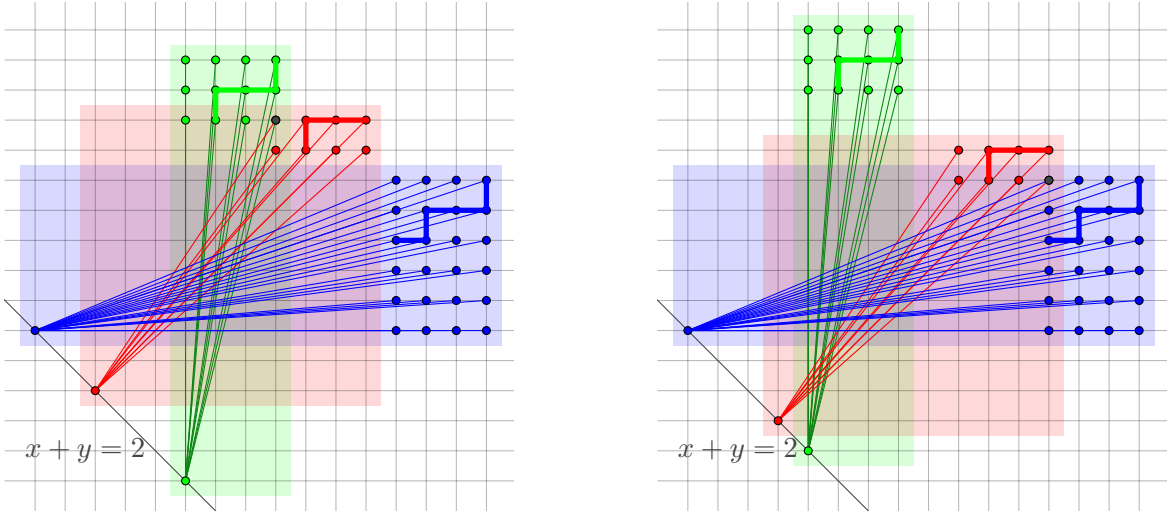

\subsubsection{Convergence to the Airy sheet}

A natural question to ask is about the scaling limit of the colored TASEP.
Let's first consider the following scaling limit of LPP, the Airy sheet $\cS:\R^2\to\R$, constructed in \cite{dauvergne2018directed}.
For each $x\in \R$, $y\mapsto \cS(x,x+y)+y^2$ is the so-called Airy$_2$ process on $\R$ from \cite{prahofer2002scale}, which is stationary and ergodic, and has the GUE Tracy-Widom one-point distributions.
The convergence from LPP to $\cS$ can be stated as follows. Take $\cS_n: \R^2\to \R$ where
\[
\cS_n(x,y) = 2^{-4/3}n^{-1/3} (L_{(1+\lfloor 2^{2/3}n^{2/3}x \rfloor, 1-\lfloor 2^{2/3}n^{2/3}x \rfloor), (n+\lfloor 2^{2/3}n^{2/3}y \rfloor, n-\lfloor 2^{2/3}n^{2/3}y \rfloor)}-4n).
\]
In words, $\cS_n$ is a rescaled version of all the LPP passage times, for two endpoints varying in lines $\{(1+i,1-i):i \in \Z\}$ and $\{(n+i,n-i):i \in \Z\}$.
In \cite{dauvergne2021scaling} it is proved that $\cS_n \to \cS$ weakly, in the topology of uniform convergence on compact sets.

We note that for each fixed $i\in \Z$, $\{L_{(1+i,1-i),v}\}_{(1+i,1-i)\le v}$ is equivalent to the evolution of an (uncolored) TASEP with step initial condition, via the coupling between LPP and TASEP.
Thus LPP with starting point varying in $\{(1+i,1-i):i \in \Z\}$ gives a coupling of a family of (uncolored) step initial TASEPs, indexed by $\Z$. Its law is different from that of $\{\mu^A\}_{A\in\Z}$, given by the colored TASEP.
For the scaling limit of the colored TASEP, it is also believed to be the Airy sheet, i.e., these two families of (uncolored) step initial TASEP have the same scaling limit.
We state it using the passage times of the colored TASEP.
\begin{question}
For each $n\in\N$, let $\cS_n^*:\R^2\to\R$ be the random function, such that
\[
\cS_n^*(x,y) = 2^{-4/3}n^{-1/3} (T_{n+\lfloor 2^{2/3}n^{2/3}y \rfloor-\lfloor 2^{2/3}n^{2/3}x \rfloor, n-\lfloor 2^{2/3}n^{2/3}y \rfloor+\lfloor 2^{2/3}n^{2/3}x \rfloor}^{\lfloor 2^{2/3}n^{2/3}x \rfloor}-4n),\quad \forall x, y\in \R.
\]
Find the limit of $\cS_n^*$, as $n\to\infty$. In particular, does the limit have the same distribution as $\cS$?
\end{question}
Combined with the LPP shift-invariance proved in \cite{dauvergne2020hidden}, and the convergence of $\cS_n$, our Theorem \ref{thm:main-de} or \ref{thm:main-sa} implies the following result.
\begin{theorem}  \label{thm:conv-colored}
Take any $\Theta\subset \R^2$, such that for any $u=(u_1, u_2), v=(v_1, v_2) \in \Theta$, there is either $u_1\le v_1$ and $u_2\ge v_2$, or $u_1\ge v_1$ and $u_2\le v_2$.
Then $\cS_n^* \to \cS$ in $\Theta$, weakly in the topology of uniform convergence on compact sets.
\end{theorem}

\subsubsection{Discussions on related hidden invariances}  \label{ssec:dis}

We now explain the relation between our shift-invariance, and various other hidden invariance in exactly solvable models, from \cite{borodin2019shift, dauvergne2020hidden, galashin2020symmetries}.

With the connection to LPP, our results take a similar form as those in \cite{dauvergne2020hidden}. In particular, for the ordered relation between two lattice rectangles $\cR, \cR'$, it is equivalent to the following notion widely used in \cite{dauvergne2020hidden}: for any up-right paths from the bottom-left corner to the up-right corner in $\cR$ and $\cR'$ respectively, they must intersect. 
However, as noted above, for $\opi^A$ with different $A\in\Z$, the corresponding LPP random fields are different and coupled in a non-trivial way, thus the above identities do not follow from the results in \cite{dauvergne2020hidden}. We also do not see how our results can be proved using the framework or similar arguments as those in \cite{dauvergne2020hidden}, via certain conditional independence obtained from the RSK correspondence.

In \cite{borodin2019shift, galashin2020symmetries}, certain shift- and flip-invariance statements are proved for the colored stochastic six-vertex model.
Our colored TASEP can be obtained from the colored stochastic six-vertex model, by passing the parameters there to a certain limit.
As pointed out in \cite[Remark 2.3]{bufetov2020absorbing}, when passing the shift- or flip-invariance statements from the colored stochastic six-vertex model to the colored TASEP, one gets equalities in distribution at a single time (see e.g.\ \cite[Theorem 2.2]{bufetov2020absorbing}).
More precisely, we define the height function as
\[
\cH^A(x,t) = \text{number of particles on or to the right of $x$ with color $\le A$ at time $t$,}
\]
for any $A, x\in \Z$ and $t\ge 0$. This is degenerated from the height function of the colored stochastic six-vertex model, as appeared in \cite{borodin2019shift, galashin2020symmetries} and also below. 
The single-time equality in distribution is the following result.
\begin{theorem}  \label{thm:eqgal}
For any $t\ge 0$, $\{A_i\}_{i=1}^\ell \in \Z^\ell$, $\{A_i'\}_{i=1}^n \in \Z^n$, 
$\{x_i\}_{i=1}^\ell, \{C_i\}_{i=1}^\ell \in \N^\ell$, $\{x_i'\}_{i=1}^n, \{C_i'\}_{i=1}^n \in \N^n$, such that
\[
A_1,\ldots, A_\ell < A_1',\ldots, A_n',\quad x_1,\ldots, x_\ell \ge x_1',\ldots, x_n',
\]
we have
\begin{align*}
& \PP\left[ \cH^{A_i}(x_i,t) \ge C_i, \cH^{A_j'}(x_j',t) \ge C_j', \forall 1\le i \le n, 1\le j \le \ell  \right] \\
=& \PP\left[ \cH^{A_i+1}(x_i+1,t) \ge C_i, \cH^{A_j'}(x_j',t) \ge C_j', \forall 1\le i \le n, 1\le j \le \ell  \right]    
\end{align*}
\end{theorem}
In \cite{bufetov2020absorbing}, a slightly weaker result is stated (\cite[Theorem 2.2]{bufetov2020absorbing}).
Our Theorem \ref{thm:main-de} generalizes Theorem \ref{thm:eqgal} to the multiple time setting. For comparison, we restate Theorem \ref{thm:main-de} using the height function, as follows.
\begin{customthm}{1.1 (Height Function Version)}
Let $g\in\N$, $k_1,\ldots, k_g\in\N$, and take $A_{i,j}\in \Z$, $x_{i,j}, C_{i,j} \in\N$, $t_i\ge 0$, for all $1\le i \le g$ and $1\le j \le k_i$.
Let $1\le \iota < g$, and for any $1\le i\le g$ and $1\le j \le k_i$ we let $A_{i,j}^+ = A_{i,j}+\don[i> \iota]$ and $x_{i,j}^+ = x_{i,j}+\don[i> \iota]$. Suppose that for any $1\le i<i' \le g$, and $1\le j \le k_i, 1\le j' \le k_{i'}$ we have $A_{i,j}^+ \le A_{i',j'}$, $A_{i,j}-C_{i,j} \ge A_{i',j'}^+ - C_{i,j}$, and $x_{i,j} + C_{i,j} \ge x_{i',j'}^+ + C_{i',j'}$.
Then we have
\begin{align*}
& \PP\left[ \cH^{A_{i,j}}(x_{i,j},t_i) \ge C_{i,j}, \forall 1\le i \le g, 1\le j \le k_i  \right] \\
=& \PP\left[ \cH^{A_{i,j}^+}(x_{i,j}^+,t_i) \ge C_{i,j}, \forall 1\le i \le g, 1\le j \le k_i  \right].
\end{align*}
\end{customthm}
We note that to allow for multiple times, we have to impose more restrictions on the spatial parameters.
It seems that there are no corresponding arguments (that can prove this multiple-time version) as those in the proofs in \cite{borodin2019shift, galashin2020symmetries} in the degenerated setting of the colored TASEP model.
This is because the full power of the colored stochastic six-vertex model is necessary for their proofs.

We also note that our Theorem \ref{thm:main-de} is in a spirit close to the main results in \cite{borodin2019shift}, for we require that all pairs of rectangles (of different $i$'s) are ordered.
It is asked as \cite[Conjecture 1.5]{borodin2019shift} whether such shift-invariance holds in more generality, with ordering imposed only for endpoints where relative shift happens.
This has been answered in \cite{galashin2020symmetries} and \cite{dauvergne2020hidden} in different settings.
It is natural to ask if a similar extension holds for our results, i.e., are some of the ordering inequalities in Theorem \ref{thm:main-de} not necessary for the shift-invariance to hold?
Our Theorem \ref{thm:main-sa} can be viewed as one step toward this, for there we allow rectangles with the same $A_i$ to be not ordered.
See Section \ref{ssec:shift-cons} for some further discussions on this.

\subsection{The oriented swap process and the conjectured identity}

Consider the Cayley graph of the symmetric group $S_N$, generated by swaps of numbers at adjacent sites $k$ and $k+1$, for each $1\le k \le N-1$.
A \emph{sorting network}, first studied by Stanley \cite{stanley1984number}, is one of the shortest paths between the identity permutation $(1, \ldots, N)$ and the reverse permutation $(N, \ldots, 1)$. Equivalently, it is a sequence of $N(N-1)/2$ adjacent swaps, from the identity permutation to the reverse permutation.
We note that in such a sequence, each swap must increase the total number of inversions (i.e., switch adjacent $i, j$ to $j, i$ for some $i<j$); otherwise more than $N(N-1)/2$ swaps are needed to obtain the reverse permutation from the identity permutation.

One natural way of defining a {random sorting network} is to assign equal probability to each sorting network.
This model has been extensively studied in \cite{angel2007random, angel2012pattern, rozinov2016statistics, gorin2019random,  dauvergne2020circular, dauvergne2018archimedean}.
Another natural way of defining a random sorting network is the so-called oriented swap process (OSP), a Markovian construction from \cite{angel2009oriented}.
One lets the permutation evolve in a continuous way, such that at any time, for every two adjacent sites, if the number at the left site is smaller than the number at the right site, with rate $1$ these two numbers swap.
More formally, one can think of OSP as a finite interval version of the colored TASEP.
Consider $N$ particles on the finite lattice $\{1,\ldots,N\}$, such that initially (at time $0$) the particle at site each $k$ is colored $k$.
For each edge $(k, k+1)$ (for $1\le k \le N-1$), there is an independent Poisson clock that rings at rate $1$; and whenever it rings, the particles occupying the sites $k$ and $k+1$ will attempt to swap.
Suppose that these two particles have colors $i$ and $j$ respectively, then the swap succeeds only if $i < j$, i.e., only if the swap increases the number of inversions in the sequence of particle colors.

We consider the vector of \emph{finishing times} $\bU_N=(U_N(1),\ldots, U_N(N-1))$, where $U_N(k)$ is the last time such that a swap happens between the sites $k$ and $k+1$.
As has been observed in \cite{angel2009oriented}, for each individual $k$, the time $U_N(k)$ has the same distribution as $L_{(1,1),(k,N-k)}$. Using this, it is proved that if we let $N\to \infty$ and assume $k/N\to y$ for some $0<y<1$, then $U_N(k)$ grows linearly in $N$ with a fluctuation in the order of $N^{1/3}$, and (after appropriate rescaling) the fluctuation converges to GUE Tracy-Widom distribution
(\cite[Theorem 1.6]{angel2009oriented}).

It is asked in \cite{angel2009oriented} about the asymptotic behavior of the \emph{absorbing time} of OSP, defined as $\max_{1\le k \le N-1}U_N(k)$.
It is conjectured by Bisi, Cunden, Bibbson, and Romik that the vector $\bU_N$ and $\{L_{(1,1),(k,N-k)}\}_{k=1}^{N-1}$ are equal in distribution (\cite[Conjecture 1.2]{bisi2020oriented} and \cite[Conjecture 1.2]{bisi2020sorting}); and assuming such equality in distribution they show that the absorbing time of OSP (after appropriate rescaling) converges to the GOE Tracy-Widom distribution.

Starting from this, in \cite{bufetov2020absorbing} it is proved that $\max_{1\le k \le N-1}U_n(k)$ has the same distribution as $\max_{1\le k \le N-1}L_{(1,1),(k,N-k)}$, using distributional identities from the colored stochastic six-vertex model, obtained in \cite{borodin2019shift}.
The latter corresponds to the passage time of TASEP with the flat initial condition, and it is known to converge to the Tracy-Widom GOE distribution (see e.g.\ \cite{BFPS07, Sas05}). 
Thus the OSP absorbing time convergence problem is settled in \cite{bufetov2020absorbing}.
However, the joint equality in distribution between $\bU_N$ and $\{L_{(1,1),(k,N-k)}\}_{k=1}^{N-1}$ remains open.
As pointed out in \cite{bufetov2020absorbing}, such joint equality in distribution does not follow directly from any known bijections and also escapes the generality of the method taken there, therefore some new ideas are needed.

In \cite{bisi2020oriented}, several other special cases of joint equality in distribution have been verified.
For example, it is shown that the equality in distribution holds for $2\le N \le 6$, by a computer-assisted proof using the equivalent combinatorial formulation (see Section \ref{ssec:combf}); and it is also shown that $(U_N(1), U_N(N-1))$ has the same distribution as $(L_{(1,1),(1,N-1)}, L_{(1,1),(N-1,1)})$.

Using Theorem \ref{thm:main-de}, we prove the desired joint equality in distribution, thus resolve \cite[Conjecture 1.2]{bisi2020oriented, bisi2020sorting} (also stated as \cite[Conjecture 1.3]{bufetov2020absorbing}).
\begin{theorem}  \label{thm:udv}
The vectors $\bU_N$ and $\{L_{(1,1),(k,N-k)}\}_{k=1}^{N-1}$ are equal in distribution.
\end{theorem}

\subsubsection{Asymptotic results}  \label{ssec:asyr}
With the equality in distribution of Theorem \ref{thm:udv}, we can deduce results on the OSP finishing times from corresponding LPP results, strengthening results in \cite{angel2009oriented}.

Let $\cA_2$ denote the stationary Airy$_2$ process on $\R$, then $\cA_2(0)$ has the GUE Tracy-Widom distribution.
The fluctuation of the vector $\bU_N$ converges to $\cA_2$ minus a parabola.
\begin{theorem}  \label{cor:j-airy}
For any $y \in (0, 1)$, as $N\to\infty$, the function \[x\mapsto \frac{(y(1-y))^{1/6}}{(1+2\sqrt{y(1-y)})^{2/3}N^{1/3}}(U_N(\lfloor yN+ xN^{2/3} \rfloor) - (1+2\sqrt{y(1-y)})N - \frac{1-2y}{\sqrt{y(1-y)}}xN^{2/3})\] weakly converges to $\cA_2(x)-x^2$, in the topology of uniform convergence on compact sets.
\end{theorem}
We consider the \emph{last swap location} $k_*\in \{1,\cdots, N-1\}$, such that the last swap is between sites $k_*$ and $k_*+1$.
We then get the following convergence about $k_*$.
\begin{theorem}  \label{cor:max-loc}
As $N\to\infty$, $\frac{k_*-N/2}{N^{2/3}}$ converges in distribution, to $\argmax_x \cA_2(x)-x^2$.
\end{theorem}
The limiting distribution of $\argmax_x \cA_2(x)-x^2$ has been studied in \cite{flores2013endpoint,schehr2012extremes}, with explicit formulas given.

In a scale smaller than $N^{2/3}$, the local fluctuations are simple random walks.
\begin{theorem}  \label{cor:loc-srw}
Take any $y\in (0,1)$, and positive integers $K_N$ such that $N^{-2/3}K_N\to 0$ as $N\to\infty$.
Consider two random functions $f_N, g_N: [-K_N, K_N]\cap \Z \to \R$, where $f_N(x)= U_N(\lfloor yN\rfloor+x)-U_N(\lfloor yN\rfloor)$, and $g_N$ is a two-sided random walk, such that $g_N(x+1)-g_N(x)$ are i.i.d. for each $-K_N\le x < K_N$, with 
\[
\PP[g_N(x+1)-g_N(x)=t] =
\begin{cases}
\frac{\sqrt{y(1-y)}}{1+2\sqrt{y(1-y)}}e^{-\frac{\sqrt{y}t}{\sqrt{y}+\sqrt{1-y}}}\;\; & t\ge 0, \\
\frac{\sqrt{y(1-y)}}{1+2\sqrt{y(1-y)}}e^{\frac{\sqrt{1-y}t}{\sqrt{y}+\sqrt{1-y}}}\;\; & t<0.
\end{cases}\]
Then the total variation distance between $f_N$ and $g_N$ decays to zero, as $N\to\infty$.
\end{theorem}
Using the shift-invariance, we can also get the distribution of the OSP local dynamics, near the finishing times. As an example, the following result can be directly deduced from Theorem \ref{cor:loc-srw}.

For any $1\le k \le N$, we consider the last jump of the number $k$ in OSP (of size $N$), and we denote $\LAL(N,k)$ as the event where the last jump of $k$ is to the left.
Noting that this event $\LAL(N,k)$ is equivalent to that $U_N(N+1-k)>U_N(N-k)$ (assuming that $U_N(0)=U_N(N)=0$), we immediately get the following result.
\begin{cor}  \label{cor:left}
Take any $y \in (0,1)$, and a sequence of integers $k_N$ such that $N^{-1}k_N\to y$ as $N\to\infty$.
Then we have that $\PP[\LAL(N,k_N)] \to \frac{\sqrt{y}}{\sqrt{y}+\sqrt{1-y}}$.
\end{cor}

\subsubsection{An equivalent formulation: Young tableaux and sorting networks}  \label{ssec:combf}

In \cite{bisi2020oriented}, it is proved that there is an equivalent formulation of Theorem \ref{thm:udv} in combinatorics, which gives an identity between rational functions raised from Young tableaux and sorting networks. It can be seen as an extension of the remarkable Edelman-Greene correspondence between Young tableaux and sorting networks \cite{EG}.
We record this result here.

We will mostly follow the notations in \cite{bisi2020oriented}.
Denote $\delta_N$ as the Young diagram of $N-1$ rows, where the $k$-th row has $N-k$ boxes, and let $\SYT(\delta_N)$ denote all Young tableaux with shape $\delta_N$, i.e., all $\lambda=\{\lambda_{i,j}\}_{i,j\ge 1, i+j\le N}$, such that these numbers are precisely $1, \ldots, N(N-1)/2$, and $\lambda_{i,j}<\lambda_{i+1,j}\wedge \lambda_{i,j+1}$, for any $i,j\ge 1$, $i+j\le N-1$.
For each $\lambda \in \SYT(\delta_N)$, we denote $\co_\lambda = (\lambda_{N-1,1},\ldots, \lambda_{1,N-1})$ as the vector of the last entries of each row.
Let $\sigma_\lambda \in S_{N-1}$ be a permutation, such that for any $1\le i < j \le N-1$, we have $\sigma_\lambda(i)<\sigma_\lambda(j)$ if and only if $\co_\lambda(i)<\co_\lambda(j)$.
We let $\ocor_\lambda$ be the increasing rearrangement of $\co_\lambda$, i.e., $\ocor_\lambda = \co_\lambda \circ \sigma_\lambda^{-1}$.
We also define a $N(N-1)/2$ dimensional vector $\de_\lambda$ as follows: we consider the sequence of Young diagrams $\emptyset=\delta_\lambda^{(0)}, \delta_\lambda^{(1)}, \ldots, \delta_\lambda^{(N(N-1)/2)}=\delta_N$, such that each $\delta_\lambda^{(k)}$ contains all boxes $(i,j)$ with $\lambda_{i,j}\le k$.
Then for each $0\le k\le N(N-1)/2-1$, we let $\de_\lambda(k)$ be the number of boxes $(i,j)\subset \delta_N \setminus \delta_\lambda^{(k)}$, such that $\delta_\lambda^{(k)}\cup \{(i,j)\}$ is a Young subdiagram of $\delta_N$.
For each permutation $\sigma\in S_{N-1}$, we take the rational function $F_\sigma \in \C(x_1,\ldots,x_{N-1})$ as
\[
F_\sigma(x_1,\ldots,x_{N-1}) = \sum_{\lambda \in \SYT(\delta_N), \sigma_\lambda = \sigma} \prod_{k=1}^{N-1} \prod_{i=\ocor_\lambda(k-1)+1}^{\ocor_\lambda(k)} \frac{1}{x_k+\de_\lambda(i)}.
\]

There is a similar set of quantities in the sorting network.
Let $\SN_N$ be the collection of all sorting networks from $(1,\ldots,N)$ to $(N,\ldots,1)$, where each $s=(s_1, \ldots, s_{N(N-1)/2})\in \SN_N$ is a sequence of numbers, each in $\{1,\ldots, N-1\}$, such that the numbers at sites $s_k$ and $s_k+1$ are swapped at step $k$.
We let $\la_s= (\la_s(1),\ldots, \la_s(N-1))$ such that $\la_s(k)=\max\{1\le i \le N(N-1)/2: s_i=k\}$ for each $1\le k \le N-1$.
Let $\sigma_s \in S_{N-1}$ be a permutation, such that for any $1\le i < j \le N-1$, we have $\sigma_s(i)<\sigma_s(j)$ if and only if $\la_s(i)<\la_s(j)$, and let $\ola_s=\la_s\circ \sigma_s$.
For each $0\le k \le N(N-1)/2-1$, we let $\de_s(k)=|\{1\le i \le N-1: v^{(k)}(i) < v^{(k)}(i+1)\}|$, where $v^{(k)}\in S_{N-1}$ is the configuration of the sorting network $s$ after the $k$-th step.
For each permutation $\sigma\in S_{N-1}$, we take the rational function $G_\sigma \in \C(x_1,\ldots,x_{N-1})$ as
\[
G_\sigma(x_1,\ldots,x_{N-1}) = \sum_{s \in \SN_N, \sigma_s = \sigma} \prod_{k=1}^{N-1} \prod_{i=\ola_s(k-1)+1}^{\ola_s(k)} \frac{1}{x_k+\de_s(i)}.
\]
As proved in \cite{bisi2020oriented}, from the density functions and taking certain Fourier transforms, Theorem \ref{thm:udv} is equivalent to the following.
\begin{theorem}  \label{thm:EG}
For $N\ge 2$ and any $\sigma\in S_{N-1}$, we have that 
\[
F_\sigma(x_1,\ldots,x_{N-1}) = G_\sigma(x_1,\ldots,x_{N-1}).
\]
\end{theorem}
In \cite{stanley1984number} Stanley computed $|\SN_N|$ and noticed that it equals $|\SYT(\delta_N)|$, later Edelman and Greene \cite{EG} gave a bijection $\EG_N: \SYT_N\to \SN_N$, which played crucial roles in the study of uniform random sorting networks in e.g.\ \cite{angel2007random, angel2012pattern}.
The bijection $\EG_N$ also gives that $\la_{\EG_N(\lambda)} = \co_\lambda$ and 
$\sigma_{\EG_N(\lambda)} = \sigma_\lambda$ for any $\lambda \in \SYT_N$; and also
$|\{\lambda \in \SYT(\delta_N), \sigma_\lambda = \sigma\}| = |\{s \in \SN_N, \sigma_s = \sigma\}|$.
However, it is difficult to use the Edelman-Greene correspondence to study the OSP, or to deduce  Theorem \ref{thm:EG}, because (as pointed out in \cite{bisi2020oriented}) the identity \[\prod_{k=1}^{N-1} \prod_{i=\ocor_\lambda(k-1)+1}^{\ocor_\lambda(k)} \frac{1}{x_k+\de_\lambda(i)} = \prod_{k=1}^{N-1} \prod_{i=\ola_{\EG(\lambda)}(k-1)+1}^{\ola_{\EG(\lambda)}(k)} \frac{1}{x_k+\de_{\EG(\lambda)}(i)}\] is not generally true for any $\lambda \in \SYT_N$.
Thus one cannot find a direct proof of Theorem \ref{thm:EG} via a bijection (that may be different from $\EG_N$) between $\SYT_N$ and $\SN_N$.

\section{Proof strategy}  \label{sec:stra}

In this section, we explain the general approach to prove Theorem \ref{thm:main-de} and Theorem \ref{thm:main-sa}.

As discussed in Section \ref{ssec:dis}, using the colored stochastic six-vertex model distributional identities from \cite{borodin2019shift, galashin2020symmetries}, we can prove a single-time equality in distribution for the colored TASEP model (Theorem \ref{thm:eqgal}).
Our main work is to upgrade Theorem \ref{thm:eqgal} to Theorem \ref{thm:main-de}, which is the multi-time distributional identity.

For better expository, we shall avoid the notion of the height function, and write the arguments using the passage times.
Recall that for each $A\in \Z$ and $B, C\in \N$, the passage time $T_{B,C}^A$ is the first time when there are at least $C$ particles and on or to the right of site $A+B+1-C$ with color $\le A$.
For any $t\ge 0$, we denote $I^t[A,B,C]$ as the event where $T_{B,C}^A\le t$.
Then $I^t[A,B,C]$ is also equivalent to the event $\cH^A(A+B+1-C,t)\ge C$.
We can then write Theorem \ref{thm:eqgal} and Theorem \ref{thm:main-de} as follows.
\begin{customthm}{1.7*}  \label{thm:eqgals}
For any $t>0$, and \[\{A_i\}_{i=1}^\ell \in \Z^\ell,\;\{A_i'\}_{i=1}^n \in \Z^n,\;\{B_i\}_{i=1}^\ell,\; \{C_i\}_{i=1}^\ell \in \N^\ell,\;\{B_i'\}_{i=1}^n, \{C_i'\}_{i=1}^n \in \N^n,\] such that
\[
A_1,\ldots, A_\ell < A_1',\ldots, A_n',
\]
and
\[
A_1+B_1-C_1,\ldots, A_\ell+B_\ell-C_\ell \ge A_1'+B_1'-C_1',\ldots, A_n'+B_n'-C_n',
\]
we have
\[
\PP\left[\left(\bigcap_{i=1}^\ell I^t[A_i,B_i,C_i]\right)\cap \left(\bigcap_{i=1}^n I^t[A_i',B_i',C_i']\right)\right]
=
\PP\left[\left(\bigcap_{i=1}^\ell I^t[A_i+1,B_i,C_i]\right)\cap \left(\bigcap_{i=1}^n I^t[A_i',B_i',C_i']\right)\right].
\]
\end{customthm}

\begin{customthm}{1.1*}  \label{thm:main-des}
Let $g\in\N$, $k_1,\ldots, k_g\in\N$, and take $A_{i,j}\in \Z$, $B_{i,j}, C_{i,j} \in\N$, for all $1\le i \le g$ and $1\le j \le k_i$.
Let $1\le \iota < g$, and for any $1\le i\le g$ and $1\le j \le k_i$ we let $A_{i,j}^+ = A_{i,j}+\don[i> \iota]$. Suppose that for any $1\le i<i' \le g$, and $1\le j \le k_i, 1\le j' \le k_{i'}$ we have 
\begin{equation}  \label{eq:assumabc}
A_{i,j} \le A_{i',j'},\quad A_{i,j}^++B_{i,j} \ge A_{i',j'}^++B_{i',j'}, \quad A_{i,j}^+-C_{i,j} \ge A_{i',j'}^+-C_{i',j'}.
\end{equation}
Then for any $t_1,\ldots, t_g >0$, we have
\[
\PP\left[\bigcap_{i=1}^g\bigcap_{j=1}^{k_i} I^{t_i}[A_{i,j},B_{i,j},C_{i,j}]\right]
=
\PP\left[\bigcap_{i=1}^g\bigcap_{j=1}^{k_i} I^{t_i}[A_{i,j}^+,B_{i,j},C_{i,j}]\right].
\]
\end{customthm}
This theorem is equivalent to Theorem \ref{thm:main-de} because for any $A_1, A_2 \in \Z$ and $B_1,B_2, C_1,C_2 \in \N$, the condition $\cR_{B_1,C_1}^{A_1}\le \cR_{B_2,C_2}^{A_2}$ is equivalent to that $A_1\le A_2$, and $A_1+B_1\ge A_2+B_2$, $A_1-C_1\ge A_2-C_2$.

The general strategy to prove Theorem \ref{thm:main-des} is to do a careful induction on the number of different times, i.e., the parameter $g$.
At each inductive step, we exploit the fact that the passage times have analytic distribution functions, and use some conditional independence of the passage times.
To illustrate the arguments for each inductive step, we consider the following example involving two passage times, which can be viewed as a special case of Theorem \ref{thm:main-des}.

\begin{ex}  \label{ex:simple}
Take some $B,C\in\N$, $B,C\ge 2$.
For any $t_1, t_2 > 0$, we have that
\[
\PP[T^0_{B,1}\le t_1, T^0_{1,C}\le t_2] = \PP[T^0_{B,1}\le t_1, T^1_{1,C}\le t_2].
\]
\end{ex}
As we only concern about two certain passage times, we can simplify the model to contain only two particles: the first one $P_1$ is the rightmost particle in $\opi^0$, which starts at site $0$ and jumps to the right;
the second one $P_2$ is the leftmost hole in $\opi^0$ or $\opi^1$, which starts at site $1$ or $2$, and jumps to the left.
The waiting times for each jump is $\Exp(1)$ independently, except for when $P_1$ is to left next to $P_2$: then they swap with $\Exp(1)$ waiting time.

\begin{figure}[hbt!]
    \centering
    
\begin{subfigure}[t]{0.48\textwidth}
         \centering
\begin{tikzpicture}[line cap=round,line join=round,>=triangle 45,x=4cm,y=4cm]
\clip(-2.65,-0.07) rectangle (-.15,0.235);

\draw [line width=1.2pt, opacity=0.3] (-2.7,0.1) -- (-0.75,0.1);
\draw [fill=white] (-2.6,0.1) circle (3.0pt);;
\draw [fill=white] (-2.5,0.1) circle (3.0pt);;
\draw [fill=white] (-2.4,0.1) circle (3.0pt);;
\draw [fill=white] (-2.3,0.1) circle (3.0pt);;
\draw [fill=white] (-2.2,0.1) circle (3.0pt);;
\draw [fill=white] (-2.1,0.1) circle (3.0pt);;
\draw [fill=white] (-2.,0.1) circle (3.0pt);;
\draw [fill=white] (-1.9,0.1) circle (3.0pt);;
\draw [fill=white] (-1.8,0.1) circle (3.0pt);;
\draw [fill=blue] (-1.7,0.1) circle (3.0pt);;
\draw [fill=red] (-1.6,0.1) circle (3.0pt);;
\draw [fill=white] (-1.5,0.1) circle (3.0pt);;
\draw [fill=white] (-1.4,0.1) circle (3.0pt);;
\draw [fill=white] (-1.3,0.1) circle (3.0pt);;
\draw [fill=white] (-1.2,0.1) circle (3.0pt);;
\draw [fill=white] (-1.1,0.1) circle (3.0pt);;
\draw [fill=white] (-1.,0.1) circle (3.0pt);;
\draw [fill=white] (-0.9,0.1) circle (3.0pt);;
\draw [fill=white] (-0.8,0.1) circle (3.0pt);;

\begin{scriptsize}
\foreach \i in {-10,-9,...,9}
{
\draw node[anchor=south] at  (\i/10-1.7,0.12) {\i};
}

\draw node[anchor=west] at  (-0.72, 0.1) {time $0$};
\end{scriptsize}

\draw [-stealth] (-1.6,0.06) to [bend left] (-2.4,0.06);

\draw [-stealth] (-1.7,0.06) to [bend right] (-1.1,0.06);

\end{tikzpicture}
\end{subfigure}
\hfill     
\begin{subfigure}[t]{0.48\textwidth}
         \centering
\begin{tikzpicture}[line cap=round,line join=round,>=triangle 45,x=4cm,y=4cm]
\clip(-2.65,-0.07) rectangle (-.75,0.235);

\draw [line width=1.2pt, opacity=0.3] (-2.7,0.1) -- (-0.1,0.1);
\draw [fill=white] (-2.6,0.1) circle (3.0pt);;
\draw [fill=white] (-2.5,0.1) circle (3.0pt);;
\draw [fill=white] (-2.4,0.1) circle (3.0pt);;
\draw [fill=white] (-2.3,0.1) circle (3.0pt);;
\draw [fill=white] (-2.2,0.1) circle (3.0pt);;
\draw [fill=white] (-2.1,0.1) circle (3.0pt);;
\draw [fill=white] (-2.,0.1) circle (3.0pt);;
\draw [fill=white] (-1.9,0.1) circle (3.0pt);;
\draw [fill=white] (-1.8,0.1) circle (3.0pt);;
\draw [fill=blue] (-1.7,0.1) circle (3.0pt);;
\draw [fill=white] (-1.6,0.1) circle (3.0pt);;
\draw [fill=red] (-1.5,0.1) circle (3.0pt);;
\draw [fill=white] (-1.4,0.1) circle (3.0pt);;
\draw [fill=white] (-1.3,0.1) circle (3.0pt);;
\draw [fill=white] (-1.2,0.1) circle (3.0pt);;
\draw [fill=white] (-1.1,0.1) circle (3.0pt);;
\draw [fill=white] (-1.,0.1) circle (3.0pt);;
\draw [fill=white] (-0.9,0.1) circle (3.0pt);;
\draw [fill=white] (-0.8,0.1) circle (3.0pt);;

\begin{scriptsize}
\foreach \i in {-10,-9,...,10}
{
\draw node[anchor=south] at  (\i/10-1.7,0.12) {\i};
}
\end{scriptsize}

\draw [-stealth] (-1.5,0.06) to [bend left] (-2.3,0.06);

\draw [-stealth] (-1.7,0.06) to [bend right] (-1.1,0.06);

\end{tikzpicture}
\end{subfigure}

\begin{subfigure}[t]{0.48\textwidth}
         \centering
\begin{tikzpicture}[line cap=round,line join=round,>=triangle 45,x=4cm,y=4cm]
\clip(-2.65,-0.07) rectangle (-.15,0.235);

\draw [line width=1.2pt, opacity=0.3] (-2.7,0.1) -- (-0.75,0.1);
\draw [fill=white] (-2.6,0.1) circle (3.0pt);;
\draw [fill=white] (-2.5,0.1) circle (3.0pt);;
\draw [fill=white] (-2.4,0.1) circle (3.0pt);;
\draw [fill=white] (-2.3,0.1) circle (3.0pt);;
\draw [fill=white] (-2.2,0.1) circle (3.0pt);;
\draw [fill=white] (-2.1,0.1) circle (3.0pt);;
\draw [fill=white] (-2.,0.1) circle (3.0pt);;
\draw [fill=white] (-1.9,0.1) circle (3.0pt);;
\draw [fill=red] (-1.8,0.1) circle (3.0pt);;
\draw [fill=white] (-1.7,0.1) circle (3.0pt);;
\draw [fill=white] (-1.6,0.1) circle (3.0pt);;
\draw [fill=blue] (-1.5,0.1) circle (3.0pt);;
\draw [fill=white] (-1.4,0.1) circle (3.0pt);;
\draw [fill=white] (-1.3,0.1) circle (3.0pt);;
\draw [fill=white] (-1.2,0.1) circle (3.0pt);;
\draw [fill=white] (-1.1,0.1) circle (3.0pt);;
\draw [fill=white] (-1.,0.1) circle (3.0pt);;
\draw [fill=white] (-0.9,0.1) circle (3.0pt);;
\draw [fill=white] (-0.8,0.1) circle (3.0pt);;

\begin{scriptsize}
\foreach \i in {-10,-9,...,9}
{
\draw node[anchor=south] at  (\i/10-1.7,0.12) {\i};
}
\draw node[anchor=west] at  (-0.72, 0.1) {time $T_{2,1}^0$};
\end{scriptsize}

\draw [-stealth] (-1.8,0.06) to [bend left] (-2.4,0.06);

\draw [-stealth] (-1.5,0.06) to [bend right] (-1.1,0.06);

\end{tikzpicture}
\end{subfigure}
\hfill     
\begin{subfigure}[t]{0.48\textwidth}
         \centering
\begin{tikzpicture}[line cap=round,line join=round,>=triangle 45,x=4cm,y=4cm]
\clip(-2.65,-0.07) rectangle (-.75,0.235);

\draw [line width=1.2pt, opacity=0.3] (-2.7,0.1) -- (-0.1,0.1);
\draw [fill=white] (-2.6,0.1) circle (3.0pt);;
\draw [fill=white] (-2.5,0.1) circle (3.0pt);;
\draw [fill=white] (-2.4,0.1) circle (3.0pt);;
\draw [fill=white] (-2.3,0.1) circle (3.0pt);;
\draw [fill=white] (-2.2,0.1) circle (3.0pt);;
\draw [fill=white] (-2.1,0.1) circle (3.0pt);;
\draw [fill=white] (-2.,0.1) circle (3.0pt);;
\draw [fill=white] (-1.9,0.1) circle (3.0pt);;
\draw [fill=white] (-1.8,0.1) circle (3.0pt);;
\draw [fill=red] (-1.7,0.1) circle (3.0pt);;
\draw [fill=white] (-1.6,0.1) circle (3.0pt);;
\draw [fill=blue] (-1.5,0.1) circle (3.0pt);;
\draw [fill=white] (-1.4,0.1) circle (3.0pt);;
\draw [fill=white] (-1.3,0.1) circle (3.0pt);;
\draw [fill=white] (-1.2,0.1) circle (3.0pt);;
\draw [fill=white] (-1.1,0.1) circle (3.0pt);;
\draw [fill=white] (-1.,0.1) circle (3.0pt);;
\draw [fill=white] (-0.9,0.1) circle (3.0pt);;
\draw [fill=white] (-0.8,0.1) circle (3.0pt);;

\begin{scriptsize}
\foreach \i in {-10,-9,...,10}
{
\draw node[anchor=south] at  (\i/10-1.7,0.12) {\i};
}
\end{scriptsize}

\draw [-stealth] (-1.7,0.06) to [bend left] (-2.3,0.06);

\draw [-stealth] (-1.5,0.06) to [bend right] (-1.1,0.06);

\end{tikzpicture}
\end{subfigure}

\caption{
An illustration of $T_{B,1}^0, T_{1,C}^0$ (left) and $T_{B,1}^0, T_{1,C}^1$ (right) in Example \ref{ex:simple} (with $B=6$ and $C=8$).
Here $P_1$ is denoted by the blue particle, and $P_2$ is denoted by the red particle.
Starting from $T_{2,1}^0$ the evolutions of both particles will be independent; therefore it suffices to show that (up to a shift by $1$) the locations of the red particles at time $T_{2,1}^0$ have the same distribution, in the left and right settings.
To achieve this, we use that $T_{B',1}^0\vee T_{1,C}^0$ and $T_{B',1}^0\vee T_{1,C}^1$ have the same distribution (by Theorem \ref{thm:eqgal}), and consider all large enough $B'$.
}  
\label{fig:rbp}
\end{figure}
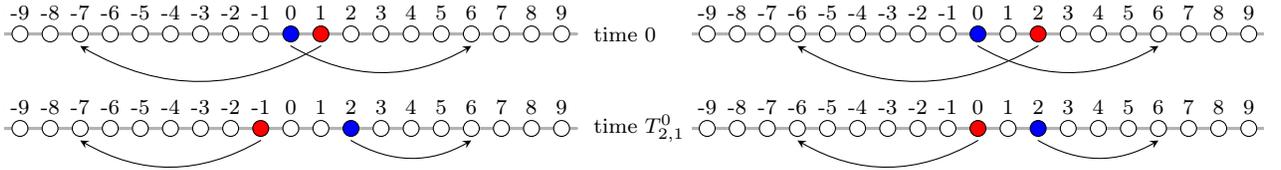

We just consider the case where $t_1< t_2$, and the case where $t_1>t_2$ follows a similar argument.
A key observation is that, since the time $T^0_{2,1}$, i.e., the time when $P_1$ arrives at site $2$, the future evolutions of $P_1$ and $P_2$ will be independent.
This is because, no matter whether $P_2$ starts from site $1$ or $2$, at time $T^0_{2,1}$ it must be to the left of $P_1$, thus these two particles have swapped already, and all future waiting times will be independent.
In particular, $T^0_{B,1}-T^0_{2,1}$ will just be the sum of $B-2$ i.i.d. $\Exp(1)$ random variables.

We consider functions $f: t\mapsto \PP[T^0_{2,1}\le t, T^0_{1,C}\le t_2]$ and $f^+: t\mapsto \PP[T^0_{2,1}\le t, T^1_{1,C}\le t_2]$.
Thus it suffices to show that $f(t)=f^+(t)$ for any $t\in [0, t_2]$.
We use Theorem \ref{thm:eqgal} to deduce this.
For any large $B'\in \N$, $B'\ge 2$, we have $\PP[T^0_{B',1}, T^0_{1,C}\le t_2] = \PP[T^0_{B',1}, T^1_{1,C}\le t_2]$, by Theorem \ref{thm:eqgal}.
We can also write
\[
\PP[T^0_{B',1}, T^0_{1,C}\le t_2] = \int_0^{t_2} f(t)
\frac{(t_2-t)^{B'-3}e^{-(t_2-t)}}{(B'-3)!}  dt,
\]
and
\[
\PP[T^0_{B',1}, T^1_{1,C}\le t_2] = \int_0^{t_2} f^+(t)
\frac{(t_2-t)^{B'-3}e^{-(t_2-t)}}{(B'-3)!}  dt.
\]
Thus we have
\begin{equation}  \label{eq:stra-demo}
\int_0^{t_2} (f(t)-f^+(t))
(t_2-t)^{B'-3}e^{-(t_2-t)} dt=0, \quad \forall B'\in\N,\; B'\ge 3.    
\end{equation}
Suppose that the function $f-f^+$ is analytic, we can deduce that $f-f^+ = 0$.
Otherwise, one can find some $\beta \in \Z_{\ge 0}$ and $D\neq 0$, such that $\lim_{t\to 0}t^{-\beta}(f(t)-f^+(t)) = D$.
By taking $B'$ large, we can make $(t_2-t)^{B'-3}e^{-(t_2-t)}$ decay fast as $t$ grows from $0$, and get a contradiction with \eqref{eq:stra-demo}.

Our proof for Theorem \ref{thm:main-des} follows the same strategy.
\begin{enumerate}
    \item We will find certain stopping times, which we call `cutting times', like the time $T^0_{2,1}$ in the above example. We prove that for certain two sets of particles, they evolve independently after these times.
    \item The next step will be to show that the cutting time distribution is invariant under shift. We will get formulas similar to \eqref{eq:stra-demo}.
    In some settings, at this step, we use the induction hypothesis (instead of Theorem \ref{thm:eqgal}).
\item We then prove that the probability density function of cutting times is analytic.
\item We analyze the transition probability from cutting times (i.e., the distribution of $T^0_{B',1}-T^0_{2,1}$ in the example): we will show that by taking the corresponding $B'$s large enough, the transition distribution concentrates on small $t$. 
Thus we conclude that the cutting time distributions are the same.
\end{enumerate}

We remark that in the setting of multiple times and multiple particles, the independence property in Step 1 is delicate.
Thus one main difficulty in our proof is designing the appropriate setup, 
mainly including the induction setup and the choice of the cutting times, to get the independence property.

For Theorem \ref{thm:main-sa}, it is deduced from Theorem \ref{thm:main-de}, plus some arguments using the independence of the Poisson field in different areas.

The oriented swap process identity (Theorem \ref{thm:udv}) can be obtained by repeatedly applying Theorem \ref{thm:main-de}. One issue is that it is on a finite interval, rather than $\Z$; and for this, we use input from \cite{angel2009oriented}, which translates the colored TASEP into a finite interval version via a pair of truncation operators.
Such truncation procedure also appeared in \cite{bufetov2020absorbing} and was used to study the OSP absorbing time $\max_{1\le k \le N-1}U_n(k)$, using the single-time equality in distribution of the colored TASEP.
The asymptotic results in Section \ref{ssec:asyr} are proved using Theorem \ref{thm:udv} and various existing results on LPP.

\subsection*{Organization of the remaining text}
The remaining text mainly focuses on the proofs.
In Section \ref{sec:prelim} we make some preparations, including proving Theorem \ref{thm:eqgal} using input from the colored stochastic six-vertex model (Section \ref{ssec:6v}), setting up notations and proving basic results on projections of the colored TASEP and the Poisson field (Section \ref{ssec:proj} and \ref{ssec:poif}), and proving results on the analyticity of certain probability density functions (Section \ref{ssec:density}).
Theorem \ref{thm:main-des} is proved in the next two sections: Section \ref{sec:ind-setup} is for the induction setup and Step 1 above, and Section \ref{sec:equal-cut-info} is for more detailed arguments on the remaining steps.
In Section \ref{sec:exten} we prove Theorem \ref{thm:main-sa} from Theorem \ref{thm:main-de} and deduce Theorem \ref{thm:conv-colored}, and we also discuss some possible further extensions of the results in Section \ref{ssec:shift-cons}.
In Section \ref{sec:osp} we prove Theorem \ref{thm:udv} and explain how it implies the asymptotic results.

\section{Preliminaries}   \label{sec:prelim}

\subsection{Input from the colored stochastic six-vertex model}  \label{ssec:6v}

We first prove Theorem \ref{thm:eqgal}, using the main result of \cite{galashin2020symmetries}.

Our notations for the colored stochastic six-vertex model follow those in \cite{borodin2019shift}.
We consider the model as random colored up-right paths in the positive quadrant $\N^2$.
All the paths enter the quadrant from the left boundary, such that for each $i\in \N$, there is a path of color $i$ entering from the left in row $i$.
Given the entering paths, they progress in the up-right direction
within the quadrant.
For each vertex of the lattice, given the colors of the entering paths along the bottom and left adjacent edges, we choose the colors of the exiting paths along the top and right edges
according to the following probabilities (see Figure \ref{fig:6v}).
Let $0\le b_2 < b_1 < 1$ be two parameters. 
Suppose the path entering from the bottom is in color $i$, and the path entering from the left is in color $j$.
If $i\le j$, then with probability $b_1$, the path exiting from the top is in color $i$, and the path exiting from the right is in color $j$;
and with probability $1-b_1$, the path exiting from the top is in color $j$, and the path exiting from the right is in color $i$.
If $i>j$, it has the same transition probability, with $b_1$ replaced by $b_2$.
We also use color $0$ to encode the absence of a path.

\begin{figure}[hbt!]
    \centering
\begin{tikzpicture}[line cap=round,line join=round,>=triangle 45,x=1.2cm,y=1.2cm]
\clip(-1.2,4.7) rectangle (3.1,7.6);

\draw [-stealth,color=blue](0,7) -- (1,7);
\draw [-stealth,color=red](0.5,6.5) -- (0.5,7.5);

\draw [-stealth,color=blue](2,7) -- (2.45,7) -- (2.5,7.05) -- (2.5,7.5);
\draw [-stealth,color=red](2.5,6.5) -- (2.5,6.95) -- (2.55,7) -- (3,7);

\draw (0,7) node[anchor=east, color=blue]{$j$};
\draw (0.5,6.5) node[anchor=north, color=red]{$i$};

\draw (2,7) node[anchor=east, color=blue]{$j$};
\draw (2.5,6.5) node[anchor=north, color=red]{$i$};

\draw (-0.8,5.8) node[anchor=north]{$i\le j$};
\draw (-0.8,5.2) node[anchor=north]{$i> j$};

\draw (0.5,5.8) node[anchor=north]{$b_1$};
\draw (2.5,5.8) node[anchor=north]{$1-b_1$};

\draw (0.5,5.2) node[anchor=north]{$b_2$};
\draw (2.5,5.2) node[anchor=north]{$1-b_2$};

\end{tikzpicture}
\caption{
Probabilities at each vertex of the colored stochastic six-vertex model. 
}  
\label{fig:6v}
\end{figure}
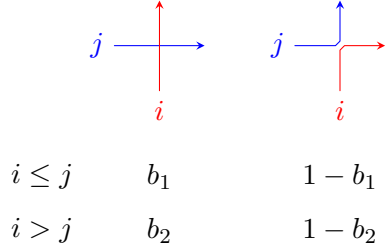

We next define the height function of the colored stochastic six-vertex model.
For each $m\in \N$ and $x, y \in (\N-\frac{1}{2})^2$, we denote
$\cH^m_{6v}(x,y)$ as the number of paths that are in color $\ge m$, and pass below $(x,y)$. 

A particular case of the main result in \cite[Theorem 1.6]{galashin2020symmetries} is as follows, in the above notations.
\begin{theorem}
\label{thm:gal-6v}
Let $\ell, n \in \N$, $\{m_i\}_{i=1}^\ell \in \N^\ell$, $\{m_i'\}_{i=1}^n\in\N^n$,
and $\{(x_i, y_i)\}_{i=1}^\ell \in (\N-\frac{1}{2})^{2\ell}$, $\{(x_i', y_i')\}_{i=1}^n \in (\N-\frac{1}{2})^{2n}$, such that
\begin{enumerate}
    \item $\max_{1\le i \le \ell}m_i < \min_{1\le i \le n}m_i'$, 
    \item $x_1 \le \cdots \le x_\ell \le x_1' \le \cdots \le x_n'$, and $y_1 \ge \cdots \ge y_\ell \ge y_1' \ge \cdots \ge y_n'$.
\end{enumerate}
Then these $\ell+n$ dimensional vectors  \[(\{\cH^{m_i}_{6v}(x_i,y_i)\}_{i=1}^\ell, \{\cH^{m_i'}_{6v}(x_i',y_i')\}_{i=1}^n)
\quad\text{and}\quad
(\{\cH^{m_i+1}_{6v}(x_i,y_i+1)\}_{i=1}^\ell, \{\cH^{m_i'}_{6v}(x_i',y_i')\}_{i=1}^n)\]
have the same distribution.
\end{theorem}

We next state the limit transition to TASEP.
It is first observed in \cite{borodin2016stochastic}, and then proved in \cite{aggarwal2017convergence}, for the stochastic six-vertex model without colors; the colored version follows the same proof.
Recall the TASEP height function: for any $A, x\in\Z$, $\cH^A(x,t)$ is the number of particles with color $\le A$ at site $\ge x$ at time $t$; or equivalently, it is the number of particles at site $\ge x$ in $\opi^A_t$.
\begin{theorem}  \label{thm:6v-ctasep}
Take $b_1=\varepsilon$ and $b_2=0$, then there is
\[
\cH_{6v}^{A+1}(\lfloor \varepsilon^{-1}t\rfloor + 1/2, \lfloor \varepsilon^{-1}t\rfloor + x-1/2) + A-x+1 \to \cH^A(x,t),
\]
in the sense of joint convergence in distribution for any finitely many $(t, A, x) \in [0,\infty) \times \Z_{\ge 0} \times \Z$.
\end{theorem}
By combining this with Theorem \ref{thm:gal-6v}, we can prove Theorem \ref{thm:eqgal}
\begin{proof}[Proof of Theorem \ref{thm:eqgal}]
It suffices to consider the case where each $A_i\ge 0$ and each $A_i'\ge 0$.
Denote $\cE$ as the event where
\[
\cH_{6v}^{A_i+1}(\lfloor \varepsilon^{-1}t\rfloor + 1/2, \lfloor \varepsilon^{-1}t\rfloor + A_i+B_i-C_i+1/2) \ge B_i 
\]
for each $1\le i \le \ell$, and $\cE_+$ as the event where
\[
\cH_{6v}^{A_i+2}(\lfloor \varepsilon^{-1}t\rfloor + 1/2, \lfloor \varepsilon^{-1}t\rfloor + A_i+B_i-C_i+3/2) \ge B_i 
\]
for each $1\le i \le \ell$.
Also denote $\cE'$ as the event where
\[
\cH_{6v}^{A_i'+1}(\lfloor \varepsilon^{-1}t\rfloor + 1/2, \lfloor \varepsilon^{-1}t\rfloor + A_i'+B_i'-C_i'+1/2) \ge B_i' 
\]
for each $1\le i \le n$.
By Theorem \ref{thm:gal-6v} we have $\PP[\cE\cap \cE'] = \PP[\cE_+\cap \cE']$.
Then by Theorem \ref{thm:6v-ctasep}, the conclusion follows by sending $\varepsilon \to 0$.
\end{proof}

\subsection{Projections to finitely many particles}  \label{ssec:proj}
For simplicity of notation and arguments, we introduce the following `projections' of the colored TASEP $\zeta$, which are (colored or uncolored) TASEPs with finitely many particles.
We note that we take a slight misuse of notions (here and for the rest of this paper): in a colored TASEP, unless otherwise noted, a hole means a particle colored by $\infty$, and a particle means one colored by a finite number.\\

\noindent\textbf{Finite-particles and single-color.} For any $A\in \Z$ and $C\in\N$, we denote $\opi^{A,C}=(\opi_t^{A,C})_{t\ge 0}$ as the process, where for each $x \in \Z$ and $t\ge 0$, we have
$\opi_t^{A,C}(x)=0$ if $\opi_t^A(x)=0$ and $|\{y\ge x: \opi_t^A(y)=0\}| \le C$; otherwise we have $\opi_t^{A,C}(x)=\infty$.
In words, $\opi_t^{A,C}$ keeps the $C$ rightmost particles in $\opi_t^A$ and changes other particles to holes. This is also equivalent to applying to $\opi_t^A$ the `$C$ cut-off operator' from \cite{angel2009oriented} (see also Section \ref{sec:osp}).

Then $\opi^{A,C}$ encodes a TASEP with $C$ particles, starting from sites $A-C+1,\cdots, A$, where $0$ denotes particles and $\infty$ denotes holes, and it evolves using the same Poisson clock as $\zeta$.\\

\noindent\textbf{Finite-particles and multi-color.} 
For any $A\in \Z$ and $C\in\N$, let $\hmu^{A,C}=(\hmu_t^{A,C})_{t\ge 0}$ be the process defined as follows.
For each $x\in \Z$, $t\ge 0$, we let $\hmu^{A,C}_t(x)=\min\{1\le i \le C: \opi_t^{A-C+i,i}(x)=0\}\cup\{\infty\}$.\\

One can think of $\hmu_t^{A,C}$ as a colored version of $\opi_t^{A,C}$.
In fact, the locations of particles $\{x\in \Z: \hmu_t^{A,C}(x)<\infty\}$ and $\{x\in \Z: \opi_t^{A,C}(x)=0\}$ are the same.
The colors in $\hmu_t^{A,C}$ are obtained in the following way.
Consider the sequence $\opi_t^{A-C+1,1}, \opi_t^{A-C+2,2}, \ldots, \opi_t^{A,C}$. These configurations contain $1, 2, \ldots, C$ particles, respectively;
and they can be viewed as a procedure of adding one particle at a step to reach $\opi_t^{A,C}$.
The particle added at the $i$-th step is colored by $i$ in $\hmu_t^{A,C}$.
Thus one can also think of $\hmu_t^{A,C}$ as a way of encoding the sequence $\opi_t^{A-C+1,1}, \opi_t^{A-C+2,2}, \ldots, \opi_t^{A,C}$, as they precisely contain the same information (i.e., they can determine each other).
See Figure \ref{fig:projscol} for an illustration of this equivalence.
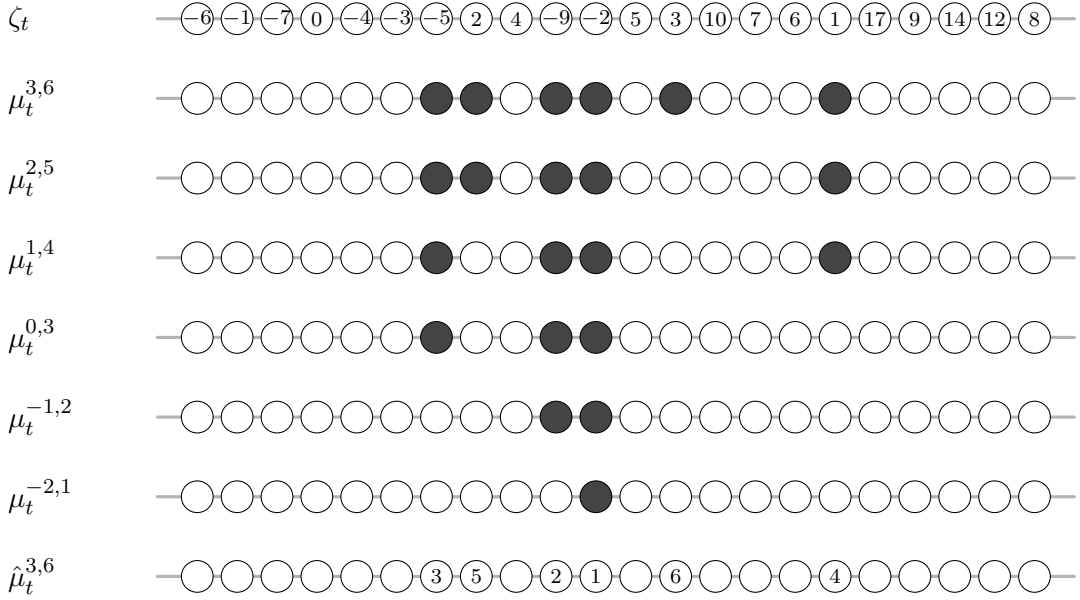
\begin{figure}[hbt!]
    \centering
\begin{tikzpicture}[line cap=round,line join=round,>=triangle 45,x=5.3cm,y=5.3cm]
\clip(-3.1,-0.7) rectangle (-0.35,0.9);

\draw [line width=1.2pt, opacity=0.3] (-2.7,-0.6) -- (-0.4,-0.6);
\draw [fill=white] (-2.6,-0.6) circle (6.0pt);
\draw [fill=white] (-2.5,-0.6) circle (6.0pt);
\draw [fill=white] (-2.4,-0.6) circle (6.0pt);
\draw [fill=white] (-2.3,-0.6) circle (6.0pt);
\draw [fill=white] (-2.2,-0.6) circle (6.0pt);
\draw [fill=white] (-2.1,-0.6) circle (6.0pt);
\draw [fill=white] (-2.,-0.6) circle (6.0pt);
\draw [fill=white] (-1.9,-0.6) circle (6.0pt);
\draw [fill=white] (-1.8,-0.6) circle (6.0pt);
\draw [fill=white] (-1.7,-0.6) circle (6.0pt);
\draw [fill=white] (-1.6,-0.6) circle (6.0pt);
\draw [fill=white] (-1.5,-0.6) circle (6.0pt);
\draw [fill=white] (-1.4,-0.6) circle (6.0pt);
\draw [fill=white] (-1.3,-0.6) circle (6.0pt);
\draw [fill=white] (-1.2,-0.6) circle (6.0pt);
\draw [fill=white] (-1.1,-0.6) circle (6.0pt);
\draw [fill=white] (-1.,-0.6) circle (6.0pt);
\draw [fill=white] (-0.9,-0.6) circle (6.0pt);
\draw [fill=white] (-0.8,-0.6) circle (6.0pt);
\draw [fill=white] (-0.7,-0.6) circle (6.0pt);
\draw [fill=white] (-0.6,-0.6) circle (6.0pt);
\draw [fill=white] (-0.5,-0.6) circle (6.0pt);

\draw [line width=1.2pt, opacity=0.3] (-2.7,-0.4) -- (-0.4,-0.4);
\draw [fill=white] (-2.6,-0.4) circle (6.0pt);
\draw [fill=white] (-2.5,-0.4) circle (6.0pt);
\draw [fill=white] (-2.4,-0.4) circle (6.0pt);
\draw [fill=white] (-2.3,-0.4) circle (6.0pt);
\draw [fill=white] (-2.2,-0.4) circle (6.0pt);
\draw [fill=white] (-2.1,-0.4) circle (6.0pt);
\draw [fill=white] (-2.,-0.4) circle (6.0pt);
\draw [fill=white] (-1.9,-0.4) circle (6.0pt);
\draw [fill=white] (-1.8,-0.4) circle (6.0pt);
\draw [fill=white] (-1.7,-0.4) circle (6.0pt);
\draw [fill=uuuuuu] (-1.6,-0.4) circle (6.0pt);
\draw [fill=white] (-1.5,-0.4) circle (6.0pt);
\draw [fill=white] (-1.4,-0.4) circle (6.0pt);
\draw [fill=white] (-1.3,-0.4) circle (6.0pt);
\draw [fill=white] (-1.2,-0.4) circle (6.0pt);
\draw [fill=white] (-1.1,-0.4) circle (6.0pt);
\draw [fill=white] (-1.,-0.4) circle (6.0pt);
\draw [fill=white] (-0.9,-0.4) circle (6.0pt);
\draw [fill=white] (-0.8,-0.4) circle (6.0pt);
\draw [fill=white] (-0.7,-0.4) circle (6.0pt);
\draw [fill=white] (-0.6,-0.4) circle (6.0pt);
\draw [fill=white] (-0.5,-0.4) circle (6.0pt);

\draw [line width=1.2pt, opacity=0.3] (-2.7,-0.2) -- (-0.4,-0.2);
\draw [fill=white] (-2.6,-0.2) circle (6.0pt);
\draw [fill=white] (-2.5,-0.2) circle (6.0pt);
\draw [fill=white] (-2.4,-0.2) circle (6.0pt);
\draw [fill=white] (-2.3,-0.2) circle (6.0pt);
\draw [fill=white] (-2.2,-0.2) circle (6.0pt);
\draw [fill=white] (-2.1,-0.2) circle (6.0pt);
\draw [fill=white] (-2.,-0.2) circle (6.0pt);
\draw [fill=white] (-1.9,-0.2) circle (6.0pt);
\draw [fill=white] (-1.8,-0.2) circle (6.0pt);
\draw [fill=uuuuuu] (-1.7,-0.2) circle (6.0pt);
\draw [fill=uuuuuu] (-1.6,-0.2) circle (6.0pt);
\draw [fill=white] (-1.5,-0.2) circle (6.0pt);
\draw [fill=white] (-1.4,-0.2) circle (6.0pt);
\draw [fill=white] (-1.3,-0.2) circle (6.0pt);
\draw [fill=white] (-1.2,-0.2) circle (6.0pt);
\draw [fill=white] (-1.1,-0.2) circle (6.0pt);
\draw [fill=white] (-1.,-0.2) circle (6.0pt);
\draw [fill=white] (-0.9,-0.2) circle (6.0pt);
\draw [fill=white] (-0.8,-0.2) circle (6.0pt);
\draw [fill=white] (-0.7,-0.2) circle (6.0pt);
\draw [fill=white] (-0.6,-0.2) circle (6.0pt);
\draw [fill=white] (-0.5,-0.2) circle (6.0pt);

\draw [line width=1.2pt, opacity=0.3] (-2.7,0) -- (-0.4,0);
\draw [fill=white] (-2.6,0) circle (6.0pt);
\draw [fill=white] (-2.5,0) circle (6.0pt);
\draw [fill=white] (-2.4,0) circle (6.0pt);
\draw [fill=white] (-2.3,0) circle (6.0pt);
\draw [fill=white] (-2.2,0) circle (6.0pt);
\draw [fill=white] (-2.1,0) circle (6.0pt);
\draw [fill=uuuuuu] (-2.,0) circle (6.0pt);
\draw [fill=white] (-1.9,0) circle (6.0pt);
\draw [fill=white] (-1.8,0) circle (6.0pt);
\draw [fill=uuuuuu] (-1.7,0) circle (6.0pt);
\draw [fill=uuuuuu] (-1.6,0) circle (6.0pt);
\draw [fill=white] (-1.5,0) circle (6.0pt);
\draw [fill=white] (-1.4,0) circle (6.0pt);
\draw [fill=white] (-1.3,0) circle (6.0pt);
\draw [fill=white] (-1.2,0) circle (6.0pt);
\draw [fill=white] (-1.1,0) circle (6.0pt);
\draw [fill=white] (-1.,0) circle (6.0pt);
\draw [fill=white] (-0.9,0) circle (6.0pt);
\draw [fill=white] (-0.8,0) circle (6.0pt);
\draw [fill=white] (-0.7,0) circle (6.0pt);
\draw [fill=white] (-0.6,0) circle (6.0pt);
\draw [fill=white] (-0.5,0) circle (6.0pt);

\draw [line width=1.2pt, opacity=0.3] (-2.7,0.2) -- (-0.4,0.2);
\draw [fill=white] (-2.6,0.2) circle (6.0pt);
\draw [fill=white] (-2.5,0.2) circle (6.0pt);
\draw [fill=white] (-2.4,0.2) circle (6.0pt);
\draw [fill=white] (-2.3,0.2) circle (6.0pt);
\draw [fill=white] (-2.2,0.2) circle (6.0pt);
\draw [fill=white] (-2.1,0.2) circle (6.0pt);
\draw [fill=uuuuuu] (-2.,0.2) circle (6.0pt);
\draw [fill=white] (-1.9,0.2) circle (6.0pt);
\draw [fill=white] (-1.8,0.2) circle (6.0pt);
\draw [fill=uuuuuu] (-1.7,0.2) circle (6.0pt);
\draw [fill=uuuuuu] (-1.6,0.2) circle (6.0pt);
\draw [fill=white] (-1.5,0.2) circle (6.0pt);
\draw [fill=white] (-1.4,0.2) circle (6.0pt);
\draw [fill=white] (-1.3,0.2) circle (6.0pt);
\draw [fill=white] (-1.2,0.2) circle (6.0pt);
\draw [fill=white] (-1.1,0.2) circle (6.0pt);
\draw [fill=uuuuuu] (-1.,0.2) circle (6.0pt);
\draw [fill=white] (-0.9,0.2) circle (6.0pt);
\draw [fill=white] (-0.8,0.2) circle (6.0pt);
\draw [fill=white] (-0.7,0.2) circle (6.0pt);
\draw [fill=white] (-0.6,0.2) circle (6.0pt);
\draw [fill=white] (-0.5,0.2) circle (6.0pt);

\draw [line width=1.2pt, opacity=0.3] (-2.7,0.4) -- (-0.4,0.4);
\draw [fill=white] (-2.6,0.4) circle (6.0pt);
\draw [fill=white] (-2.5,0.4) circle (6.0pt);
\draw [fill=white] (-2.4,0.4) circle (6.0pt);
\draw [fill=white] (-2.3,0.4) circle (6.0pt);
\draw [fill=white] (-2.2,0.4) circle (6.0pt);
\draw [fill=white] (-2.1,0.4) circle (6.0pt);
\draw [fill=uuuuuu] (-2.,0.4) circle (6.0pt);
\draw [fill=uuuuuu] (-1.9,0.4) circle (6.0pt);
\draw [fill=white] (-1.8,0.4) circle (6.0pt);
\draw [fill=uuuuuu] (-1.7,0.4) circle (6.0pt);
\draw [fill=uuuuuu] (-1.6,0.4) circle (6.0pt);
\draw [fill=white] (-1.5,0.4) circle (6.0pt);
\draw [fill=white] (-1.4,0.4) circle (6.0pt);
\draw [fill=white] (-1.3,0.4) circle (6.0pt);
\draw [fill=white] (-1.2,0.4) circle (6.0pt);
\draw [fill=white] (-1.1,0.4) circle (6.0pt);
\draw [fill=uuuuuu] (-1.,0.4) circle (6.0pt);
\draw [fill=white] (-.9,0.4) circle (6.0pt);
\draw [fill=white] (-0.8,0.4) circle (6.0pt);
\draw [fill=white] (-0.7,0.4) circle (6.0pt);
\draw [fill=white] (-0.6,0.4) circle (6.0pt);
\draw [fill=white] (-0.5,0.4) circle (6.0pt);

\draw [line width=1.2pt, opacity=0.3] (-2.7,0.6) -- (-0.4,0.6);
\draw [fill=white] (-2.6,0.6) circle (6.0pt);
\draw [fill=white] (-2.5,0.6) circle (6.0pt);
\draw [fill=white] (-2.4,0.6) circle (6.0pt);
\draw [fill=white] (-2.3,0.6) circle (6.0pt);
\draw [fill=white] (-2.2,0.6) circle (6.0pt);
\draw [fill=white] (-2.1,0.6) circle (6.0pt);
\draw [fill=uuuuuu] (-2.,0.6) circle (6.0pt);
\draw [fill=uuuuuu] (-1.9,0.6) circle (6.0pt);
\draw [fill=white] (-1.8,0.6) circle (6.0pt);
\draw [fill=uuuuuu] (-1.7,0.6) circle (6.0pt);
\draw [fill=uuuuuu] (-1.6,0.6) circle (6.0pt);
\draw [fill=white] (-1.5,0.6) circle (6.0pt);
\draw [fill=uuuuuu] (-1.4,0.6) circle (6.0pt);
\draw [fill=white] (-1.3,0.6) circle (6.0pt);
\draw [fill=white] (-1.2,0.6) circle (6.0pt);
\draw [fill=white] (-1.1,0.6) circle (6.0pt);
\draw [fill=uuuuuu] (-1.,0.6) circle (6.0pt);
\draw [fill=white] (-0.9,0.6) circle (6.0pt);
\draw [fill=white] (-0.8,0.6) circle (6.0pt);
\draw [fill=white] (-0.7,0.6) circle (6.0pt);
\draw [fill=white] (-0.6,0.6) circle (6.0pt);
\draw [fill=white] (-0.5,0.6) circle (6.0pt);

\draw [line width=1.2pt, opacity=0.3] (-2.7,0.8) -- (-0.4,0.8);
\draw [fill=white] (-2.6,0.8) circle (6.0pt);
\draw [fill=white] (-2.5,0.8) circle (6.0pt);
\draw [fill=white] (-2.4,0.8) circle (6.0pt);
\draw [fill=white] (-2.3,0.8) circle (6.0pt);
\draw [fill=white] (-2.2,0.8) circle (6.0pt);
\draw [fill=white] (-2.1,0.8) circle (6.0pt);
\draw [fill=white] (-2.,0.8) circle (6.0pt);
\draw [fill=white] (-1.9,0.8) circle (6.0pt);
\draw [fill=white] (-1.8,0.8) circle (6.0pt);
\draw [fill=white] (-1.7,0.8) circle (6.0pt);
\draw [fill=white] (-1.6,0.8) circle (6.0pt);
\draw [fill=white] (-1.5,0.8) circle (6.0pt);
\draw [fill=white] (-1.4,0.8) circle (6.0pt);
\draw [fill=white] (-1.3,0.8) circle (6.0pt);
\draw [fill=white] (-1.2,0.8) circle (6.0pt);
\draw [fill=white] (-1.1,0.8) circle (6.0pt);
\draw [fill=white] (-1.,0.8) circle (6.0pt);
\draw [fill=white] (-.9,0.8) circle (6.0pt);
\draw [fill=white] (-0.8,0.8) circle (6.0pt);
\draw [fill=white] (-0.7,0.8) circle (6.0pt);
\draw [fill=white] (-0.6,0.8) circle (6.0pt);
\draw [fill=white] (-0.5,0.8) circle (6.0pt);

\begin{scriptsize}
\draw (-2.6,0.8) node[anchor=center]{$-6$};
\draw (-2.5,0.8) node[anchor=center]{$-1$};
\draw (-2.4,0.8) node[anchor=center]{$-7$};
\draw (-2.3,0.8) node[anchor=center]{$0$};
\draw (-2.2,0.8) node[anchor=center]{$-4$};
\draw (-2.1,0.8) node[anchor=center]{$-3$};
\draw (-2.0,0.8) node[anchor=center]{$-5$};
\draw (-1.9,0.8) node[anchor=center]{$2$};
\draw (-1.8,0.8) node[anchor=center]{$4$};
\draw (-1.7,0.8) node[anchor=center]{$-9$};
\draw (-1.6,0.8) node[anchor=center]{$-2$};
\draw (-1.5,0.8) node[anchor=center]{$5$};
\draw (-1.4,0.8) node[anchor=center]{$3$};
\draw (-1.3,0.8) node[anchor=center]{$10$};
\draw (-1.2,0.8) node[anchor=center]{$7$};
\draw (-1.1,0.8) node[anchor=center]{$6$};
\draw (-1.0,0.8) node[anchor=center]{$1$};
\draw (-0.9,0.8) node[anchor=center]{$17$};
\draw (-0.8,0.8) node[anchor=center]{$9$};
\draw (-0.7,0.8) node[anchor=center]{$14$};
\draw (-0.6,0.8) node[anchor=center]{$12$};
\draw (-0.5,0.8) node[anchor=center]{$8$};

\draw (-2.,-0.6) node[anchor=center]{$3$};
\draw (-1.9,-0.6) node[anchor=center]{$5$};
\draw (-1.7,-0.6) node[anchor=center]{$2$};
\draw (-1.6,-0.6) node[anchor=center]{$1$};
\draw (-1.4,-0.6) node[anchor=center]{$6$};
\draw (-1.,-0.6) node[anchor=center]{$4$};
\end{scriptsize}

\draw (-3.1,0.8) node[anchor=west]{$\zeta_t$};
\draw (-3.1,0.6) node[anchor=west]{$\opi_t^{3,6}$};
\draw (-3.1,0.4) node[anchor=west]{$\opi_t^{2,5}$}; 
\draw (-3.1,0.2) node[anchor=west]{$\opi_t^{1,4}$};
\draw (-3.1,0.) node[anchor=west]{$\opi_t^{0,3}$};
\draw (-3.1,-0.2) node[anchor=west]{$\opi_t^{-1,2}$};
\draw (-3.1,-0.4) node[anchor=west]{$\opi_t^{-2,1}$};
\draw (-3.1,-0.6) node[anchor=west]{$\hmu_t^{3,6}$};
\end{tikzpicture}
\caption{An illustration of the equivalence between $\hmu_t^{3,6}$ and the sequence $\opi_t^{-2,1}, \opi_t^{-1,2}, \ldots, \opi_t^{3,6}$.}  
\label{fig:projscol}
\end{figure}

We then verify the following statement.

\begin{lemma}  \label{lem:law-hmu}
The process $\hmu^{A,C}$ is a colored TASEP on $\Z$, where there are only finitely many particles, at sites $A-C+1, \cdots, A$ initially, and are colored by $1, \cdots, C$, respectively.
\end{lemma}
\begin{proof}
For each $1\le i < C$, $t\ge 0$, and $x\in \Z$, if $\opi_t^{A-C+i,i}(x)=0$, we have that $\opi_t^{A-C+i}(x)=0$, and $|\{y\ge x: \opi_t^{A-C+i}(y)=0\}| \le i$.
Thus we have $\opi_t^{A-C+i+1}(x)=0$, and $|\{y\ge x: \opi_t^{A-C+i+1}(y)=0\}| \le i+1$, so $\opi_t^{A-C+i+1,i+1}(x)=0$.
This means that there is exactly one $x\in \Z$ with $\opi_t^{A-C+i,i}(x)=\infty$ and $\opi_t^{A-C+i+1,i+1}(x)=0$, which is the only $x\in \Z$ with $\hmu_t^{A,C}(x)=i$.
We then get the conclusion from the fact that for each $1\le i\le C$, $\opi^{A-C+i,i}$ is a TASEP with $i$ particles starting from $A-C+1,\cdots, A-C+i$, and they evolve with the same Poisson clocks.
\end{proof}

While this projection only involves finitely many particles, it still contains much information on passage times in the original process.

\begin{lemma}  \label{lem:proj-info}
For any $A'\in\Z$, $B',C'\in\N$, such that $A'\le A$, $A'-C'\ge A-C$, the number $T^{A'}_{B',C'}$ is determined by $\hmu_t^{A,C}$, as
\[
T^{A'}_{B',C'}=\inf\{t>0: |\{x\in\Z: x\ge A'+B'+1-C', \hmu_t^{A,C}(x)\le A'-A+C \}|\ge C'\}.
\]
In particular, the configuration $\hmu_t^{A',C'}$ is determined by $\hmu_t^{A,C}$.
\end{lemma}
The fact that $\hmu_t^{A,C}$ contains sufficient information to determine $\hmu_t^{A',C'}$ can be understood as follows.
As we have discussed above, $\hmu_t^{A,C}$ is equivalent to $\opi_t^{A-C+1,1}, \opi_t^{A-C+2,2}, \ldots, \opi_t^{A,C}$.
For each $\opi_t^{A-C+i,i}$ in this sequence, by taking the $j$ rightmost particles (for $1\le j \le i$), we get $\opi_t^{A-C+i,j}$.
This way we can obtain each one of $\opi_t^{A'-C'+1,1}, \opi_t^{A'-C'+2,2}, \ldots, \opi_t^{A',C'}$, therefore get $\hmu_t^{A',C'}$.
See Figure \ref{fig:projct} for illustrations of $\hmu_t^{A,C}$ and  $\hmu_t^{A',C'}$ (with $A=1$, $C=8$, $A'=0$, $C'=5$).

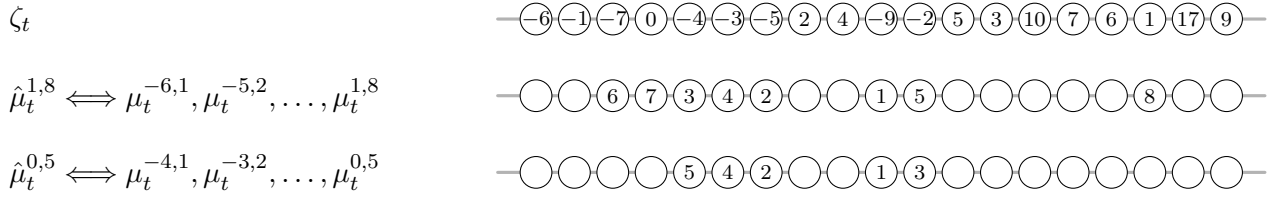
\begin{figure}[hbt!]
    \centering
    \begin{tikzpicture}[line cap=round,line join=round,>=triangle 45,x=5.1cm,y=5.1cm]
\clip(-4,0.1) rectangle (-0.35,.7);

\draw [line width=1.2pt, opacity=0.3] (-2.7,0.2) -- (-0.7,0.2);
\draw [fill=white] (-2.6,0.2) circle (6.0pt);
\draw [fill=white] (-2.5,0.2) circle (6.0pt);
\draw [fill=white] (-2.4,0.2) circle (6.0pt);
\draw [fill=white] (-2.3,0.2) circle (6.0pt);
\draw [fill=white] (-2.2,0.2) circle (6.0pt);
\draw [fill=white] (-2.1,0.2) circle (6.0pt);
\draw [fill=white] (-2.,0.2) circle (6.0pt);
\draw [fill=white] (-1.9,0.2) circle (6.0pt);
\draw [fill=white] (-1.8,0.2) circle (6.0pt);
\draw [fill=white] (-1.7,0.2) circle (6.0pt);
\draw [fill=white] (-1.6,0.2) circle (6.0pt);
\draw [fill=white] (-1.5,0.2) circle (6.0pt);
\draw [fill=white] (-1.4,0.2) circle (6.0pt);
\draw [fill=white] (-1.3,0.2) circle (6.0pt);
\draw [fill=white] (-1.2,0.2) circle (6.0pt);
\draw [fill=white] (-1.1,0.2) circle (6.0pt);
\draw [fill=white] (-1.,0.2) circle (6.0pt);
\draw [fill=white] (-0.9,0.2) circle (6.0pt);
\draw [fill=white] (-0.8,0.2) circle (6.0pt);

\draw [line width=1.2pt, opacity=0.3] (-2.7,0.4) -- (-0.7,0.4);
\draw [fill=white] (-2.6,0.4) circle (6.0pt);
\draw [fill=white] (-2.5,0.4) circle (6.0pt);
\draw [fill=white] (-2.4,0.4) circle (6.0pt);
\draw [fill=white] (-2.3,0.4) circle (6.0pt);
\draw [fill=white] (-2.2,0.4) circle (6.0pt);
\draw [fill=white] (-2.1,0.4) circle (6.0pt);
\draw [fill=white] (-2.,0.4) circle (6.0pt);
\draw [fill=white] (-1.9,0.4) circle (6.0pt);
\draw [fill=white] (-1.8,0.4) circle (6.0pt);
\draw [fill=white] (-1.7,0.4) circle (6.0pt);
\draw [fill=white] (-1.6,0.4) circle (6.0pt);
\draw [fill=white] (-1.5,0.4) circle (6.0pt);
\draw [fill=white] (-1.4,0.4) circle (6.0pt);
\draw [fill=white] (-1.3,0.4) circle (6.0pt);
\draw [fill=white] (-1.2,0.4) circle (6.0pt);
\draw [fill=white] (-1.1,0.4) circle (6.0pt);
\draw [fill=white] (-1.,0.4) circle (6.0pt);
\draw [fill=white] (-.9,0.4) circle (6.0pt);
\draw [fill=white] (-0.8,0.4) circle (6.0pt);

\draw [line width=1.2pt, opacity=0.3] (-2.7,0.6) -- (-0.7,0.6);
\draw [fill=white] (-2.6,0.6) circle (6.0pt);
\draw [fill=white] (-2.5,0.6) circle (6.0pt);
\draw [fill=white] (-2.4,0.6) circle (6.0pt);
\draw [fill=white] (-2.3,0.6) circle (6.0pt);
\draw [fill=white] (-2.2,0.6) circle (6.0pt);
\draw [fill=white] (-2.1,0.6) circle (6.0pt);
\draw [fill=white] (-2.,0.6) circle (6.0pt);
\draw [fill=white] (-1.9,0.6) circle (6.0pt);
\draw [fill=white] (-1.8,0.6) circle (6.0pt);
\draw [fill=white] (-1.7,0.6) circle (6.0pt);
\draw [fill=white] (-1.6,0.6) circle (6.0pt);
\draw [fill=white] (-1.5,0.6) circle (6.0pt);
\draw [fill=white] (-1.4,0.6) circle (6.0pt);
\draw [fill=white] (-1.3,0.6) circle (6.0pt);
\draw [fill=white] (-1.2,0.6) circle (6.0pt);
\draw [fill=white] (-1.1,0.6) circle (6.0pt);
\draw [fill=white] (-1.,0.6) circle (6.0pt);
\draw [fill=white] (-.9,0.6) circle (6.0pt);
\draw [fill=white] (-0.8,0.6) circle (6.0pt);

\begin{scriptsize}
\draw (-2.6,0.6) node[anchor=center]{$-6$};
\draw (-2.5,0.6) node[anchor=center]{$-1$};
\draw (-2.4,0.6) node[anchor=center]{$-7$};
\draw (-2.3,0.6) node[anchor=center]{$0$};
\draw (-2.2,0.6) node[anchor=center]{$-4$};
\draw (-2.1,0.6) node[anchor=center]{$-3$};
\draw (-2.0,0.6) node[anchor=center]{$-5$};
\draw (-1.9,0.6) node[anchor=center]{$2$};
\draw (-1.8,0.6) node[anchor=center]{$4$};
\draw (-1.7,0.6) node[anchor=center]{$-9$};
\draw (-1.6,0.6) node[anchor=center]{$-2$};
\draw (-1.5,0.6) node[anchor=center]{$5$};
\draw (-1.4,0.6) node[anchor=center]{$3$};
\draw (-1.3,0.6) node[anchor=center]{$10$};
\draw (-1.2,0.6) node[anchor=center]{$7$};
\draw (-1.1,0.6) node[anchor=center]{$6$};
\draw (-1.0,0.6) node[anchor=center]{$1$};
\draw (-0.9,0.6) node[anchor=center]{$17$};
\draw (-0.8,0.6) node[anchor=center]{$9$};

\draw (-2.4,0.4) node[anchor=center]{$6$};
\draw (-2.3,0.4) node[anchor=center]{$7$};
\draw (-2.2,0.4) node[anchor=center]{$3$};
\draw (-2.1,0.4) node[anchor=center]{$4$};
\draw (-2.0,0.4) node[anchor=center]{$2$};
\draw (-1.7,0.4) node[anchor=center]{$1$};
\draw (-1.6,0.4) node[anchor=center]{$5$};
\draw (-1.0,0.4) node[anchor=center]{$8$};

\draw (-2.2,0.2) node[anchor=center]{$5$};
\draw (-2.1,0.2) node[anchor=center]{$4$};
\draw (-2.0,0.2) node[anchor=center]{$2$};
\draw (-1.7,0.2) node[anchor=center]{$1$};
\draw (-1.6,0.2) node[anchor=center]{$3$};
\end{scriptsize}

\draw (-4,0.6) node[anchor=west]{$\zeta_t$};
\draw (-4,0.4) node[anchor=west]{$\hmu_t^{1,8}\Longleftrightarrow \opi_t^{-6,1}, \opi_t^{-5,2}, \ldots, \opi_t^{1,8}$}; 
\draw (-4,0.2) node[anchor=west]{$\hmu_t^{0,5}\Longleftrightarrow \opi_t^{-4,1}, \opi_t^{-3,2}, \ldots, \opi_t^{0,5}$};
\end{tikzpicture}
    \caption{Illustrations of $\zeta_t$, $\hmu_t^{1,8}$, and $\hmu_t^{0,5}$.}
    \label{fig:projct}
\end{figure}

\begin{proof}[Proof of Lemma \ref{lem:proj-info}]
From the construction, for any $x\in\Z$, we have that $\hmu_t^{A,C}(x)\le A'-A+C$ if and only if there is some $1\le i \le A'-A+C$, such that $\opi_t^{A-C+i,i}(x)=0$, thus if and only if $\hmu_t^{A',A'-A+C}(x) < \infty$.
Thus the right-hand side is the first time $t$ when $|\{x\ge A'+B'+1-C': \hmu_t^{A',A'-A+C}(x) < \infty\}| \ge C'$.
Since $A'-A+C\ge C'$, this is also the first time $t$ when $|\{x\ge A'+B'+1-C': \hmu_t^{A'}(x) < \infty\}| \ge C'$, which is precisely $T^{A'}_{B',C'}$.
\end{proof}

\subsection{Poisson field}  \label{ssec:poif}

We use $\Pi$ to denote the Poisson clock of the original colored TASEP $\zeta$; i.e., $\Pi$ is a Poisson field on $\Z\times [0,\infty)$, where for each $x\in\Z$ and any $0\le a < b$, $\Pi(\{x\}\times [a,b])$ is the number of times that the clock on the edge $(x-1, x)$ rings in the time interval $[a,b]$.
Throughout this paper, we shall also assume (the probability $1$ event) that for any two different points $(x_1,t_1)$ and $(x_2,t_2)$ in the Poisson field $\Pi$, we have $t_1\neq t_2$ and $t_1, t_2 >0$.

From now on, for any $A\in\Z$ and $B, C\in\N$, we denote $\cP[A, B, C]=(A+B+1-C, T^A_{B,C})$, which is the point in $\Pi$ corresponding to the passage time $T^A_{B,C}$.

We will use the following two lemmas, which say that for certain events on passage times, they can be determined by $\Pi$ on some subsets of $\Z\times [0,\infty)$.
The first one can be understood as an inductive construction of passage times.
\begin{lemma}  \label{lem:I-deter-gen}
Take any $f:\Z\to [0,\infty)\cup \{\infty\}$, and let $U = \{(x,t): x\in \Z, t\le f(x)\}$ be the hypograph of $f$.
Take any $A\in\Z$ and $B,C \in \N$, and $0 \le s \le f(A+B-C+1)$.
Then the event $I^s[A,B,C] \cap \left(\bigcap_{C'=1}^C I^{f(A+B-C'+1)}[A,B,C'] \right)\cap \left(\bigcap_{B'=1}^B I^{f(A+B'-C+1)}[A,B',C] \right)$ is determined by $\Pi$ on $U$.
\end{lemma}
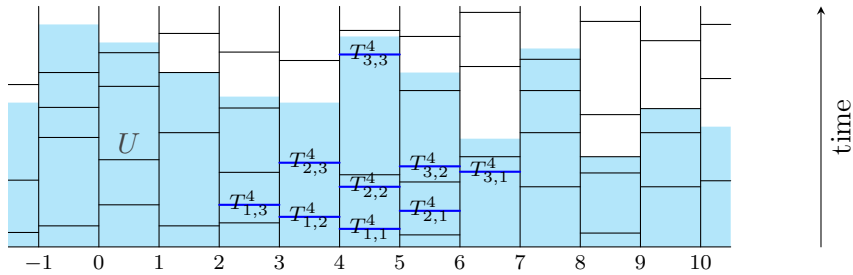
\begin{figure}[hbt!]
    \centering
\begin{tikzpicture}[line cap=round,line join=round,>=triangle 45,x=.8cm,y=.8cm]
\clip(5.5,-0.5) rectangle (19.8,4.7);

\fill[line width=0.pt,color=cyan,fill=cyan,fill opacity=0.3]
(2.5,0) -- (17.5,0) -- (17.5,2) -- (17,2) -- (17,2.3) -- (16,2.3) -- (16,1.5) -- (15,1.5) -- (15,3.3) -- (14,3.3) -- (14,1.8) -- (13,1.8) -- (13,2.9) -- (12,2.9) -- (12,3.5) -- (11,3.5) -- (11,2.4) -- (10,2.4) -- (10,2.5) -- (9,2.5) -- (9,2.9) -- (8,2.9) -- (8,3.4) -- (7,3.4) -- (7,3.7) -- (6,3.7) -- (6,2.4) -- (5,2.4) -- (5,3) -- (4,3) -- (4,2.7) -- (3,2.7) -- (3,1.8) -- (2.5,1.8) -- cycle;

\draw (2.5,0) -- (17.5,0);

\draw (3,0) -- (3,4);
\draw (4,0) -- (4,4);
\draw (5,0) -- (5,4);
\draw (6,0) -- (6,4);
\draw (7,0) -- (7,4);
\draw (8,0) -- (8,4);
\draw (9,0) -- (9,4);
\draw (10,0) -- (10,4);
\draw (11,0) -- (11,4);
\draw (12,0) -- (12,4);
\draw (13,0) -- (13,4);
\draw (14,0) -- (14,4);
\draw (15,0) -- (15,4);
\draw (16,0) -- (16,4);
\draw (17,0) -- (17,4);

\draw [-stealth] (19,0) -- (19,4);

\draw (3,3.2) -- (4,3.2);
\draw (3,2.3) -- (4,2.3);
\draw (3,1.1) -- (4,1.1);
\draw (4,3.8) -- (5,3.8);
\draw (4,2.0) -- (5,2.0);
\draw (4,1.2) -- (5,1.2);
\draw (4,0.1) -- (5,0.1);
\draw (5,2.7) -- (6,2.7);
\draw (5,1.11) -- (6,1.11);
\draw (5,0.24) -- (6,0.24);

\draw (6,2.9) -- (7,2.9);
\draw (6,2.32) -- (7,2.32);
\draw (6,1.82) -- (7,1.82);
\draw (6,0.35) -- (7,0.35);

\draw (7,3.23) -- (8,3.23);
\draw (7,2.67) -- (8,2.67);
\draw (7,1.45) -- (8,1.45);
\draw (7,0.7) -- (8,0.7);

\draw (8,3.55) -- (9,3.55);
\draw (8,2.9) -- (9,2.9);
\draw (8,1.9) -- (9,1.9);
\draw (8,0.35) -- (9,0.35);

\draw (9,3.24) -- (10,3.24);
\draw (9,2.31) -- (10,2.31);
\draw (9,1.24) -- (10,1.24);
\draw [blue] [thick] (9,0.7) -- (10,0.7);
\draw (9,0.4) -- (10,0.4);

\draw (10,3.1) -- (11,3.1);
\draw [blue] [thick] (10,1.4) -- (11,1.4);
\draw [blue] [thick] (10,0.5) -- (11,0.5);

\draw [blue] [thick] (11,3.2) -- (12,3.2);
\draw (11,3.6) -- (12,3.6);
\draw (11,1.2) -- (12,1.2);
\draw [blue] [thick] (11,1.0) -- (12,1.0);
\draw [blue] [thick] (11,0.3) -- (12,0.3);

\draw (12,3.5) -- (13,3.5);
\draw (12,2.6) -- (13,2.6);
\draw [blue] [thick] (12,1.34) -- (13,1.34);
\draw (12,1.08) -- (13,1.08);
\draw [blue] [thick] (12,0.6) -- (13,0.6);
\draw (12,0.2) -- (13,0.2);

\draw (13,3.9) -- (14,3.9);
\draw (13,3.) -- (14,3.);
\draw (13,1.5) -- (14,1.5);
\draw [blue] [thick] (13,1.25) -- (14,1.25);

\draw (14,3.12) -- (15,3.12);
\draw (14,2.6) -- (15,2.6);
\draw (14,1.9) -- (15,1.9);
\draw (14,1.0) -- (15,1.0);

\draw (15,3.75) -- (16,3.75);
\draw (15,2.2) -- (16,2.2);
\draw (15,1.5) -- (16,1.5);
\draw (15,1.23) -- (16,1.23);
\draw (15,0.23) -- (16,0.23);

\draw (16,3.43) -- (17,3.43);
\draw (16,2.3) -- (17,2.3);
\draw (16,1.9) -- (17,1.9);
\draw (16,1.0) -- (17,1.0);

\draw (17,3.7) -- (17.5,3.7);
\draw (17,2.8) -- (17.5,2.8);
\draw (17,1.1) -- (17.5,1.1);

\draw (19,2) node[anchor=north,rotate=90]{time};

\begin{scriptsize}
\draw (3,0) node[anchor=north]{$-4$};
\draw (4,0) node[anchor=north]{$-3$};
\draw (5,0) node[anchor=north]{$-2$};
\draw (6,0) node[anchor=north]{$-1$};
\draw (7,0) node[anchor=north]{$0$};
\draw (8,0) node[anchor=north]{$1$};
\draw (9,0) node[anchor=north]{$2$};
\draw (10,0) node[anchor=north]{$3$};
\draw (11,0) node[anchor=north]{$4$};
\draw (12,0) node[anchor=north]{$5$};
\draw (13,0) node[anchor=north]{$6$};
\draw (14,0) node[anchor=north]{$7$};
\draw (15,0) node[anchor=north]{$8$};
\draw (16,0) node[anchor=north]{$9$};
\draw (17,0) node[anchor=north]{$10$};

\draw (11.5,0.3) node[anchor=center]{$T^4_{1,1}$};
\draw (11.5,1) node[anchor=center]{$T^4_{2,2}$};
\draw (11.5,3.2) node[anchor=center]{$T^4_{3,3}$};

\draw (10.5,0.5) node[anchor=center]{$T^4_{1,2}$};
\draw (10.5,1.4) node[anchor=center]{$T^4_{2,3}$};

\draw (9.5,0.7) node[anchor=center]{$T^4_{1,3}$};

\draw (12.5,0.6) node[anchor=center]{$T^4_{2,1}$};
\draw (12.5,1.34) node[anchor=center]{$T^4_{3,2}$};

\draw (13.5,1.25) node[anchor=center]{$T^4_{3,1}$};
\end{scriptsize}

\draw (7.5,1.4) node[anchor=south, color=uuuuuu]{$U$};

\end{tikzpicture}
\caption{
An illustration of Lemma \ref{lem:I-deter-gen}: the segments between the lines $x-1$ and $x$ represent points of $\Pi$ on $\{x\}\times [0,\infty)$, and the blue ones are those involved in the recursive determination of $T^4_{3,3}$.
}  
\label{fig:detU}
\end{figure}

Essentially, this lemma says that to determine whether all of $\cP[A,B,C]$, $\{\cP[A,B,C']\}_{C'=1}^C$, and $\{\cP[A,B',C]\}_{B'=1}^B$ are contained in $U$, it suffices to reveal $\Pi$ on $U$.
We note that for this statement to be true, one needs to take such a collection of points: whether $\cP[A,B,C]\in \{A+B+1-C\}\times [0,t]$ (i.e., the event $I^t[A,B,C]$) is not determined by $\Pi$ on $\{A+B+1-C\}\times [0,t]$, unless $B=C=1$.

The idea behind this lemma is that $T^A_{B,C}$ can be inductively determined (see Figure \ref{fig:detU}).
Namely, it is the first time that the clock on the edge $(A+B-C, A+B+1-C)$ rings, after $T^A_{B-1,C}\vee T^A_{B,C-1}$.
Thus if both $T^A_{B-1,C}$ and $T^A_{B-1,C}$ are known after revealing $\Pi$ on $U$, one can also determine if $\cP[A,B,C]$ is in $U$.
\begin{proof}[Proof of Lemma \ref{lem:I-deter-gen}]
Denote the event in this lemma as $\cE$. 
For each $B', C' \in \N$, $B'\le B$ and $C'\le C$, we can determine $T^A_{B',C'}$ as the smallest $t > T^A_{B'-1,C'}\vee T^A_{B',C'-1}$ such that there is a point at $(A+B'+1-C', t)$ in $\Pi$; here for any $B', C'$ we take $T^A_{B',0}=T^A_{0,C'}=0$.
We can then inductively determine if the event $\cE$ holds, as follows. Suppose that we've determined $T^A_{B'-1,C'}\don[T^A_{B'-1,C'} \le f(A+B'-C')]$ and $T^A_{B',C'-1}\don[T^A_{B',C'-1} \le f(A+B'+2-C')]$.
Then by $\Pi$ on $U$ we can then determine if $I^{f(A+B'+1-C')}[A,B',C']$ holds; and if it holds, we can further determine the value of $T^A_{B',C'}$.
Thus the conclusion follows.
\end{proof}
Our second lemma states that the randomness can be leveraged: given some `boundary conditions', further passage times are independent of the randomness `below the boundary'.
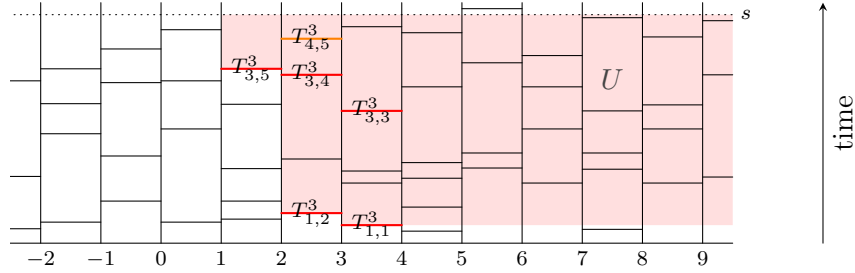
\begin{figure}[hbt!]
    \centering
\begin{tikzpicture}[line cap=round,line join=round,>=triangle 45,x=.8cm,y=.8cm]
\clip(5.5,-0.5) rectangle (19.8,4.7);

\fill[line width=0.pt,color=pink,fill=pink,fill opacity=0.5]
(9,3.8) -- (17.5,3.8) -- (17.5,0.3) -- (11,0.3) -- (11,0.5) -- (10,0.5) -- (10,2.9) -- (9,2.9) -- cycle;

\draw [dotted] (2.5,3.8) -- (17.5,3.8);

\draw (2.5,0) -- (17.5,0);

\draw (3,0) -- (3,4);
\draw (4,0) -- (4,4);
\draw (5,0) -- (5,4);
\draw (6,0) -- (6,4);
\draw (7,0) -- (7,4);
\draw (8,0) -- (8,4);
\draw (9,0) -- (9,4);
\draw (10,0) -- (10,4);
\draw (11,0) -- (11,4);
\draw (12,0) -- (12,4);
\draw (13,0) -- (13,4);
\draw (14,0) -- (14,4);
\draw (15,0) -- (15,4);
\draw (16,0) -- (16,4);
\draw (17,0) -- (17,4);

\draw [-stealth] (19,0) -- (19,4);

\draw (3,3.2) -- (4,3.2);
\draw (3,2.3) -- (4,2.3);
\draw (3,1.1) -- (4,1.1);
\draw (4,3.8) -- (5,3.8);
\draw (4,2.0) -- (5,2.0);
\draw (4,1.2) -- (5,1.2);
\draw (4,0.1) -- (5,0.1);
\draw (5,2.7) -- (6,2.7);
\draw (5,1.11) -- (6,1.11);
\draw (5,0.24) -- (6,0.24);

\draw (6,2.9) -- (7,2.9);
\draw (6,2.32) -- (7,2.32);
\draw (6,1.82) -- (7,1.82);
\draw (6,0.35) -- (7,0.35);

\draw (7,3.23) -- (8,3.23);
\draw (7,2.67) -- (8,2.67);
\draw (7,1.45) -- (8,1.45);
\draw (7,0.7) -- (8,0.7);

\draw (8,3.55) -- (9,3.55);
\draw (8,2.7) -- (9,2.7);
\draw (8,1.9) -- (9,1.9);
\draw (8,0.35) -- (9,0.35);

\draw [red] [thick] (9,2.9) -- (10,2.9);
\draw (9,2.31) -- (10,2.31);
\draw (9,1.24) -- (10,1.24);
\draw (9,0.7) -- (10,0.7);
\draw (9,0.4) -- (10,0.4);

\draw [orange] [thick] (10,3.4) -- (11,3.4);
\draw [red] [thick] (10,2.8) -- (11,2.8);
\draw (10,1.4) -- (11,1.4);
\draw [red] [thick] (10,0.5) -- (11,0.5);

\draw [red] [thick] (11,2.2) -- (12,2.2);
\draw (11,3.6) -- (12,3.6);
\draw (11,1.2) -- (12,1.2);
\draw (11,1.0) -- (12,1.0);
\draw [red] [thick] (11,0.3) -- (12,0.3);

\draw (12,3.5) -- (13,3.5);
\draw (12,2.6) -- (13,2.6);
\draw (12,1.34) -- (13,1.34);
\draw (12,1.08) -- (13,1.08);
\draw (12,0.6) -- (13,0.6);
\draw (12,0.2) -- (13,0.2);

\draw (13,3.9) -- (14,3.9);
\draw (13,3.) -- (14,3.);
\draw (13,1.5) -- (14,1.5);
\draw (13,1.25) -- (14,1.25);

\draw (14,3.12) -- (15,3.12);
\draw (14,2.6) -- (15,2.6);
\draw (14,1.9) -- (15,1.9);
\draw (14,1.0) -- (15,1.0);

\draw (15,3.75) -- (16,3.75);
\draw (15,2.2) -- (16,2.2);
\draw (15,1.5) -- (16,1.5);
\draw (15,1.23) -- (16,1.23);
\draw (15,0.23) -- (16,0.23);

\draw (16,3.43) -- (17,3.43);
\draw (16,2.3) -- (17,2.3);
\draw (16,1.9) -- (17,1.9);
\draw (16,1.0) -- (17,1.0);

\draw (17,3.7) -- (17.5,3.7);
\draw (17,2.8) -- (17.5,2.8);
\draw (17,1.1) -- (17.5,1.1);

\draw (19,2) node[anchor=north,rotate=90]{time};

\begin{scriptsize}
\draw (6,0) node[anchor=north]{$-2$};
\draw (7,0) node[anchor=north]{$-1$};
\draw (8,0) node[anchor=north]{$0$};
\draw (9,0) node[anchor=north]{$1$};
\draw (10,0) node[anchor=north]{$2$};
\draw (11,0) node[anchor=north]{$3$};
\draw (12,0) node[anchor=north]{$4$};
\draw (13,0) node[anchor=north]{$5$};
\draw (14,0) node[anchor=north]{$6$};
\draw (15,0) node[anchor=north]{$7$};
\draw (16,0) node[anchor=north]{$8$};
\draw (17,0) node[anchor=north]{$9$};

\draw (11.5,0.3) node[anchor=center]{$T^3_{1,1}$};

\draw (10.5,0.5) node[anchor=center]{$T^3_{1,2}$};

\draw (11.5,2.2) node[anchor=center]{$T^3_{3,3}$};

\draw (10.5,2.8) node[anchor=center]{$T^3_{3,4}$};

\draw (9.5,2.9) node[anchor=center]{$T^3_{3,5}$};

\draw (10.5,3.4) node[anchor=center]{$T^3_{4,5}$};

\draw (17.5,3.8) node[anchor=west]{$s$};
\end{scriptsize}

\begin{tiny}

\end{tiny}
\draw (15.5,2.4) node[anchor=south, color=uuuuuu]{$U$};

\end{tikzpicture}
\caption{
An illustration of Lemma \ref{lem:I-deter-bdy}: given $T^3_{1,1}$, $T^3_{1,2}$, $T^3_{3,3}$, $T^3_{3,4}$, $T^3_{3,5}$, (i.e., $A=3$, $C=5$, $(B_1, B_2, B_3, B_4, B_5)=(1,1,3,3,5)$), the event $T^3_{4,5} \le s$ (i.e., $I^s[3,4,5]$) is determined by $\Pi$ on $U$.
}  
\label{fig:detB}
\end{figure}
\begin{lemma}  \label{lem:I-deter-bdy}
Take any $A\in\Z$ and $B,C\in \N$, $s>0$, and a sequence of integers $0=B_0\le B_1\le \cdots \le B_C \le B$.
Suppose $i_* = \max\{0\le i \le C: B_i=0\}$,
we take $\{t_i\}_{i=1}^C$ such that $t_i= 0 $ for any $i\le i_*$, and $t_i<t_{i+1}<s$ for any $i_*\le i <C$.
Let $U\subset \Z\times [0,\infty)$, where \[U=\bigcup_{i=1}^C\{(x,t): x\ge A+B_i+1-i, t_i<t \le s\}.\]
Conditional on $T^A_{B_i,i}=t_i$ for each $1\le i \le C$, $B_i >0$, we have that the event $I^s[A,B,C]$ is determined by $\Pi$ on $U$.
\end{lemma}
In words, here the `boundary' is $\{(A+B_i+1-i, T^A_{B_i,i})\}_{i=1}^C$, given which the passage time $T^A_{B,C}$ is determined by $\Pi$ above and to the right of it.
See Figure \ref{fig:detB}.
\begin{proof}[Proof of Lemma \ref{lem:I-deter-bdy}]
For simplicity of notation, we denote $T^A_{B',0}=T^A_{0,C'}=0$ for any $B',C'\in\N$.

Let $S=\{(B',C')\in\N^2: B_{C'}<B'\le B\}$. 
For $(B',C')\in S$, we can determine $T_{A,B',C'}\don[T_{A,B',C'} \le s]$ using $\Pi$ on $U$, by induction in $B'+C'$ as follows.
Take any $(B',C')\in S$. If $(B'-1,C')\in S$, we have that 
$T^A_{B'-1,C'}\don[T^A_{B'-1,C'} \le s]$ has been determined, by induction hypothesis; otherwise we must have that $B'-1=B_{C'}$, thus $T^A_{B'-1,C'}=t_{C'}\le s$.
If $(B',C'-1)\in S$, we have that $T^A_{B',C'-1}\don[T^A_{B',C'-1} \le s]$ has been determined, by induction hypothesis; otherwise we must have that $C'-1=0$, thus $T^A_{B',C'-1}=0$.
In any case, both $T^A_{B'-1,C'}\don[T^A_{B'-1,C'} \le s]$ and $T^A_{B',C'-1}\don[T^A_{B',C'-1} \le s]$ haven been by $\Pi$ on $U$.

If $I^s[A,B'-1,C']$ or $I^s[A,B',C'-1]$ does not hold, we have that $I^s[A,B',C']$ does not hold.
If $I^s[A,B'-1,C']\cap I^s[A,B',C'-1]$ holds, we can determine $T_{A,B',C'}\don[T_{A,B',C'} \le s]$, by considering the smallest $t>T_{A,B'-1,C'}\vee T_{A,B',C'-1}$, such that $t\le s$ and there is a point at $(A+B'+1-C',t)$ in $\Pi$.
By $\Pi$ on $U$ we can determine whether such $t$ exists, and its value (if it exists).
Thus by the principle of induction, the conclusion follows.
\end{proof}

\subsection{Sum of exponential random variables} \label{ssec:density}
We need the following basic results, on the analyticity of density functions of sums of independent exponential random variables.
\begin{lemma}  \label{lem:analy-single}
Let $m\in\N$, $E_1, \cdots, E_m$ be i.i.d. $\Exp(1)$ random variables, and $a_1, \cdots, a_m>0$. The function
\[
r\mapsto \PP\left[\sum_{i=1}^m a_iE_i =r\right],
\]
is analytic in $[0, \infty)$, and can be extended to an analytic function on $\C$. 
\end{lemma}

\begin{lemma}  \label{lem:analy-multi}
Let $m\in\N$, $E_1, \cdots, E_m$ be i.i.d. $\Exp(1)$ random variables, and $a_1, \cdots, a_m>0$. 
Take integers $1\le m_1< \cdots < m_k \le m$, and $0<t_1<\cdots < t_k$.
The function
\[
r\mapsto \PP\left[\sum_{i=1}^{m_j} a_iE_i <t_j-r,\; \forall 1\le j \le k \right],
\]
is analytic in $[0, t_1]$, and can be extended to an analytic function on $\C$. 
\end{lemma}
To prove these two lemmas, we consider the following space $\Gamma_n$ of analytic functions on $\C^n$, which consists of all (finite) linear combinations of
\[
(r_1,\cdots, r_n) \mapsto \prod_{i=1}^n r_i^{a_i} e^{b_ir_i}
\]
where $a_1, \cdots, a_n \in \Z_{\ge 0}$, and $b_1, \cdots, b_n \in \R$.
We have the following properties.
\begin{lemma} \label{lem:gam-prop-1}
For any $f_1, f_2 \in \Gamma_1$, there is $f_3 \in \Gamma_1$, such that for any $r>0$, we have
\[
f_3(r) = \int_0^r f_1(s)f_2(r-s) ds.
\]
\end{lemma}

\begin{proof}
Without loss of generality, we can assume that $f_1(r)=r^{a_1}e^{b_1r}$ and $f_2(r)=r^{a_2}e^{b_2r}$, for some $a_1, a_2 \in \Z_{\ge 0}$ and $b_1, b_2 \in \R$.
When $b_1=b_2$, we have
\[
\int_0^r f_1(s)f_2(r-s) ds = e^{b_1r}\int_0^r s^{a_1}(r-s)^{a_2} ds,
\]
and the integral is a polynomial of $r$.
When $b_1\neq b_2$, we have
\[
\int_0^r f_1(s)f_2(r-s) ds = e^{b_2r}\int_0^r s^{a_1}(r-s)^{a_2}e^{(b_1-b_2)s} ds,
\]
and the integral is $F_1(r)+F_2(r)e^{(b_1-b_2)r}$, for some polynomials $F_1, F_2$.
Thus the conclusion follows.
\end{proof}

\begin{lemma} \label{lem:gam-prop-2}
For any $f_1 \in \Gamma_n$, and $0<t_1<\cdots<t_n$, there is $f_2 \in \Gamma_1$, such that
\[
f_2(r) = \int_{\sum_{i'=1}^i s_{i'}<t_i-r, \forall 1\le i \le n}  f_1(s_1, \cdots, s_n) \prod_{i=1}^n ds_i,
\]
for any $r \in [0, t_1]$.
\end{lemma}
To prove this we need the following lemma.
\begin{lemma} \label{lem:gam-prop-3}
For any $f_1 \in \Gamma_{n+1}$, and $t>0$, there is $f_2 \in \Gamma_n$, such that
\[
f_2(s_1,\cdots, s_n) = \int_0^{t-\sum_{i=1}^n s_i}  f_1(s_1, \cdots, s_{n+1})  ds_{n+1},
\]
for any $s_1,\cdots, s_n\ge 0$ with $\sum_{i=1}^n s_i < t$.
\end{lemma}
\begin{proof}
Without loss of generality we assume that $f_1(s_1, \cdots, s_{n+1}) = \prod_{i=1}^{n+1}s_i^{a_i}e^{b_is_i}$,
where each $a_i\in\Z_{\ge 0}$ and $b_i \in \R$.
Then we have
\begin{align*}
\int_0^{t-\sum_{i=1}^n s_i}  \prod_{i=1}^{n+1}s_i^{a_i}e^{b_is_i}  ds_{n+1}
&=
\int_0^{t-\sum_{i=1}^n s_i} s_{n+1}^{a_{n+1}}e^{b_{n+1}s_{n+1}}  ds_{n+1} \prod_{i=1}^{n}s_i^{a_i}e^{b_is_i} \\
&=
\left(F\left(t-\sum_{i=1}^n s_i\right)e^{b_{n+1}(t-\sum_{i=1}^n s_i)} + \frac{a_{n+1}!}{(-b_{n+1})^{a_{n+1}+1}}\right)\prod_{i=1}^{n}s_i^{a_i}e^{b_is_i},    
\end{align*}
where $F$ is a polynomial.
Then the conclusion follows.
\end{proof}

\begin{proof}[Proof of Lemma \ref{lem:gam-prop-2}]
For each $0\le m \le n$, we can inductively prove that there is some $g_m \in \Gamma_{n-m+1}$, such that
\begin{equation} \label{eq:gm-condi}
g_m(s_1,\cdots, s_{n-m}, r) = \int_{\sum_{i'=1}^i s_{i'}<t_i-r, \forall n-m+1\le i \le n}  f_1(s_1, \cdots, s_n) \prod_{i=n-m+1}^n ds_i,
\end{equation}
for any $s_1,\cdots, s_{n-m}, r\ge 0$ with $r+\sum_{i=1}^{n-m} s_i < t_{n-m+1}$.
Indeed, for the base case we just take $g_0(s_1,\cdots, s_n, r)=f_1(s_1, \cdots, s_n)$.
Then given $g_m$ for some $0\le m <n$, by Lemma \ref{lem:gam-prop-3} we let $g_{m+1} \in \Gamma_{n-m}$ be the function with
\[
g_{m+1}(s_1,\cdots, s_{n-m-1}, r) = \int_0^{t_{n-m}-r-\sum_{i=1}^{n-m-1} s_i}  g_m(s_1,\cdots, s_{n-m}, r) ds_{n-m},
\]
for any $s_1,\cdots, s_{n-m-1}, r\ge 0$ with $r+\sum_{i=1}^{n-m-1} s_i < t_{n-m}$. Then this $g_{m+1}$ satisfies \eqref{eq:gm-condi}.
Finally, we just take $f_2=g_n$, and the conclusion follows.
\end{proof}
We can now prove the lemmas on the probabilities about sums of exponential random variables. 
\begin{proof}[Proof of Lemma \ref{lem:analy-single}]
By induction in $m$, the conclusion follows from Lemma \ref{lem:gam-prop-1}, using that for each $1\le i \le m$, the function
$r\mapsto \PP[a_iE_i=r] = e^{-r/a_i}$ can be analytically extended to a function in $\Gamma_1$.
\end{proof}
\begin{proof}[Proof of Lemma \ref{lem:analy-multi}]
Denote $m_0=0$.
By Lemma \ref{lem:analy-single}, for each $1\le j \le k$ we take $f_j \in \Gamma_1$ such that
\[
f_j(s) = \PP\left[\sum_{i=m_{j-1}+1}^{m_j} a_iE_i = s\right].
\]
Then by Lemma \ref{lem:gam-prop-2}, we can find $f\in \Gamma_1$, such that
\[
f(r) = \int_{\sum_{i=1}^j s_i<t_j-r, \forall 1\le j \le k} \prod_{j=1}^k f_j(s_j)  ds_j = \PP\left[\sum_{i=1}^{m_j} a_iE_i <t_j-r,\; \forall 1\le j \le k \right],
\]
for any $r\in [0,t_1]$.
Thus the conclusion follows.
\end{proof}

\section{Inductive setup}   \label{sec:ind-setup}

In this section, we set up the induction framework to prove Theorem \ref{thm:main-de}. As already indicated, we will actually prove the equivalent form of Theorem \ref{thm:main-des}.

Recall the setup of Theorem \ref{thm:main-des}.
Let $\hg=g$ if $t_{\iota}\neq t_{\iota+1}$, and $\hg=g-1$ if $t_{\iota}=t_{\iota+1}$, and we do induction on $\hg$.
The base case where $\hg=1$ follows directly from Theorem \ref{thm:eqgal} (actually the equivalent form of Theorem \ref{thm:eqgals}). Now we assume that $\hg\ge 2$.

Take $1\le \tau \le g$ such that $t_\tau \le t_i$ for any $1\le i \le g$.
Without loss of generality (and by reflection symmetry), we assume that $\tau \le \iota$.
In the case where $t_{\iota}=t_{\iota+1}$, we also assume that $\iota\ge 2$, since otherwise, we must have $\tau=\iota=1$, so we can always set $\tau=2$ instead and do a reflection of particles and colors (note that $\hg\ge 2$ and $g\ge 3$).

For the events involved in Theorem \ref{thm:main-des}, for simplicity of notation we denote
\[
I=\bigcap_{i=\tau+1}^g\bigcap_{j=1}^{k_i}I^{t_i}[A_{i,j},B_{i,j},C_{i,j}],\quad I^+=\bigcap_{i=\tau+1}^g\bigcap_{j=1}^{k_i}I^{t_i}[A_{i,j}^+,B_{i,j},C_{i,j}],
\]
and \[J=\bigcap_{i=1}^\tau\bigcap_{j=1}^{k_i}I^{t_i}[A_{i,j},B_{i,j},C_{i,j}].\]
Namely, we split the events into two groups: $J$ contains those with `group index' $i\le \tau$, and $I$ or $I^+$ contains those with `group index' $i>\tau$.
The reason behind this splitting will be clear later.
Then the goal is to prove that $\PP[J\cap I] = \PP[J \cap I^+]$.
Note that we already have
\begin{equation}  \label{eq:IIpeq}
\PP[I]=\PP[I^+].   
\end{equation}
When $\tau<\iota$ this is by the induction hypothesis of Theorem \ref{thm:main-des}. When $\tau=\iota$ this is by a simple symmetry (of the colored TASEP).\\

\noindent\textbf{A projection.} Our next step is to simplify the problem by considering a projection of the colored TASEP $\zeta$, which we describe now.
Let 
\begin{equation}  \label{eq:defA}
A_*=\min_{\tau < i\le g, 1\le j \le k_i} A_{i,j},
\end{equation}
and take $B_*, C_*$ such that
\begin{equation}  \label{eq:defB}
    A_*+B_*=\max_{\tau < i\le g, 1\le j \le k_i}A_{i,j}+B_{i,j},
\end{equation}
\begin{equation}  \label{eq:defC}
    A_*-C_*=\min_{1 \le i \le \tau, 1\le j \le k_i}A_{i,j}-C_{i,j}.
\end{equation}
Then by \eqref{eq:assumabc}, we have
\begin{equation}  \label{eq:Ap}
A_*\ge \max_{1\le i\le \tau, 1\le j \le k_i} A_{i,j},
\end{equation}
\begin{equation}  \label{eq:Bp}
A_*+B_*\le \min_{1 \le i\le \tau, 1\le j \le k_i}A_{i,j}+B_{i,j},
\end{equation}
\begin{equation}  \label{eq:Cp}
A_*-C_*\ge \max_{\tau < i \le g, 1\le j \le k_i}A_{i,j}-C_{i,j}.
\end{equation}
We also have that $B_*, C_* \in \N$ (i.e., $B_*, C_*>0$). For $B_*$, this is because $B_* \ge B_{i,j}$ for any $\tau<i\le g$ and $1\le j \le k_i$ (due to \eqref{eq:defA} and \eqref{eq:defB}).
For $C_*$, this is because $C_* \ge C_{i,j}$ for any $1\le i\le \tau$ and $1\le j \le k_i$ (due to \eqref{eq:Ap} and \eqref{eq:defC}).

For the event $J$, it suffices to consider the projection $\hmu^{A_*,C_*}$, which (for simplicity of notation) we also denote as $\eta$ from now on.
Indeed, from Lemma \ref{lem:proj-info}, we have that for any $1\le i\le \tau$ and $1\le j \le k_i$, $I^{t_i}[A_{i,j},B_{i,j},C_{i,j}]$ is measurable with respect to the sigma-algebra generated by $\eta$, since $A_{i,j}\le A_*$ and $A_{i,j}-C_{i,j}\ge A_*-C_*$.

By Lemma \ref{lem:law-hmu}, $\eta$ is a colored TASEP on $\Z$, where initially there are particles at sites $A_*-C_*+1, \cdots, A_*$, which are colored by $1, \cdots, C_*$.\\

Our strategy of proving $\PP[J\cap I] = \PP[J \cap I^+]$ is to leverage $I, I^+$ and $J$ using some \emph{cutting information} which we define now.\\

\noindent\textbf{Cutting information.}
We start by defining a series of stopping times in $\eta$, which we call the \emph{cutting times}:
let $R_0=0$; and for each $1\le i\le C_*$, we let $R_i$ be the first time after $R_{i-1}$ when the site $A_*+B_*+1-i$ is occupied by a particle in $\eta$; i.e.
\[
R_i = \inf\{t>R_{i-1}: \eta_t(A_*+B_*+1-i) < \infty\}.
\]
Then almost surely we have that $0<R_1 < \cdots < R_{C_*}$. We also denote $R_{C_*+1}=\infty$ for simplicity of notation in the later text.
For each $1\le i\le C_*$ we let $L_i=\eta_{R_i}(A_*+B_*+1-i)$, 
the color of the particle arriving at site $A_*+B_*+1-i$ at time $R_i$.
Our cutting information is the following collection of random variables:
\[
\{\don[R_i \le t_\tau]R_i\}_{i=1}^{C_*}, \quad \{\don[R_i \le t_\tau]L_i\}_{i=1}^{C_*}.
\]
By leveraging $I, I^+$, and $J$, our goal is to show that $J\cap I$ and $J\cap I^+$ have the same probability conditioned on this cutting information.
Then by taking the expectation over the cutting information, the conclusion follows.\\

Given \eqref{eq:IIpeq}, the above task can be accomplished in two steps, given by the next two propositions.
For simplicity of notation, from now on, for each $1\le k\le C_*$ we denote $\bR^k=\{R_i\}_{i=1}^k$ and $\bL^k=\{L_i\}_{i=1}^k$.
We also denote
\[
\R_<\{k,t\} = \{\{r_i\}_{i=1}^k: 0<r_1<\cdots<r_k<t\},
\]
for any $t>0$, and let $\cL\{k\}$ denote the collection of all $\{\ell_i\}_{i=1}^k$, with $\ell_1,\cdots, \ell_k$ being mutually different numbers in $\{1,\cdots, C_*\}$.
\begin{prop}  \label{prop:equal-condi-cut}
For any $1\le k \le C_*$, $\br\in \R_<\{k,t_\tau\}$, and $\bell\in\cL\{k\}$, we have
\begin{equation}  \label{eq:equal-cut}
\PP[J\mid \bR^k=\br, R_{k+1}>t_\tau, \bL^k=\bell, I] \\ = \PP[J \mid 
\bR^k=\br, R_{k+1}>t_\tau, \bL^k=\bell, I^+].
\end{equation}
\end{prop}
This proposition actually corresponds to Step 1 in Section \ref{sec:stra}.
Given it, the remaining task to show that $I$ and $I^+$ have the same probability conditioned the cutting information $\{\don[R_i \le t_\tau]R_i\}_{i=1}^{C_*}$ and $\{\don[R_i \le t_\tau]L_i\}_{i=1}^{C_*}$.
This is given by the following identity.
\begin{prop} \label{prop:equal-p}
For any $1\le k \le C_*$, $\br\in \R_<\{k,t_\tau\}$, and $\bell\in\cL\{k\}$, we have
\begin{equation}  \label{eq:equal-p}
\PP[\bR^k=\br, R_{k+1}>t_\tau, \bL^k=\bell \mid I] \\ = \PP[\bR^k=\br, R_{k+1}>t_\tau, \bL^k=\bell \mid I^+].   
\end{equation}
\end{prop}
We will prove Proposition \ref{prop:equal-condi-cut} later in this section. Proposition \ref{prop:equal-p} will be proved in the next section, assuming that Theorem \ref{thm:main-des} is true for any smaller $\hg$.
Assuming these two propositions, it is now immediate to deduce Theorem \ref{thm:main-des} from them.
\begin{proof}[Proof of Theorem \ref{thm:main-des}]
For any $1\le k \le C_*$, $\br\in \R_<\{k,t_\tau\}$, and $\bell\in\cL\{k\}$, by Proposition \ref{prop:equal-p} and \eqref{eq:IIpeq} we have
\[
\PP[I\mid \bR^k=\br, R_{k+1}>t_\tau, \bL^k=\bell] = \PP[I^+\mid \bR^k=\br, R_{k+1}>t_\tau, \bL^k=\bell].    
\]
Then by Proposition \ref{prop:equal-condi-cut} we have that
\[
\PP[J\cap I\mid \bR^k=\br, R_{k+1}>t_\tau, \bL^k=\bell] \\ = \PP[J\cap I^+\mid \bR^k=\br, R_{k+1}>t_\tau, \bL^k=\bell].    
\]
By multiplying both sides by $\PP[\bR^k=\br, R_{k+1}>t_\tau, \bL^k=\bell]$, integrating over $\br\in \R_<\{k,t_\tau\}$, and summing over $\bell\in\cL\{k\}$ and $1\le k\le C_*$, we have that $\PP[J\cap I]=\PP[J\cap I^+]$.
Thus the proof concludes by the principle of induction.
\end{proof}

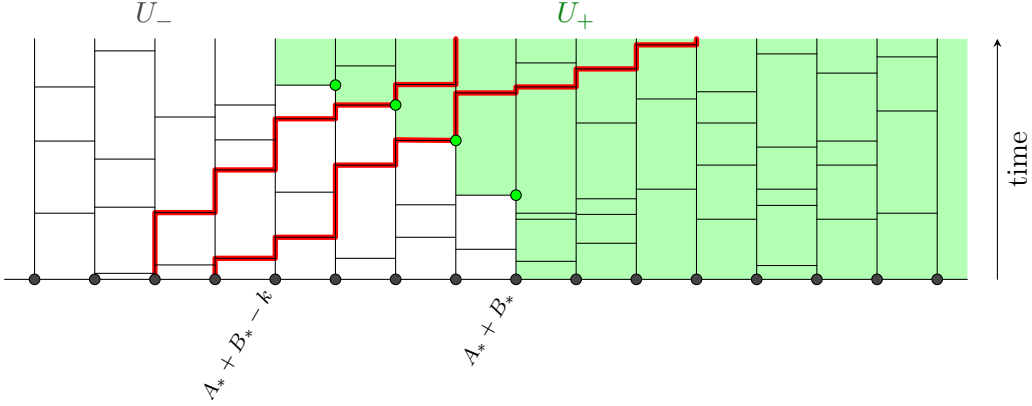
\begin{figure}[hbt!]
    \centering
\begin{tikzpicture}[line cap=round,line join=round,>=triangle 45,x=.8cm,y=.8cm]
\clip(2,-2) rectangle (19.8,4.7);

\fill[line width=0.pt,color=green,fill=green,fill opacity=0.3]
(7,4) -- (18.5,4) -- (18.5,0) -- (11,0) -- (11,1.4) -- (10,1.4) -- (10,2.31) -- (9,2.31) -- (9,2.9) -- (8,2.9) -- (8,3.23) -- (7,3.23) -- cycle;

\draw [line width=2pt, red] plot coordinates {(6,0) (6,0.35) (7,0.35)  (7,0.7) (8,0.7) (8,1.9) (9,1.9) (9,2.32) (10,2.31) (10,3.1) (11,3.1) (11,3.2) (12,3.2) (12,3.5) (13,3.5) (13,3.9) (14,3.9) (14,4)};

\draw [line width=2pt, red] plot coordinates {(5,0) (5,1.11) (6,1.11) (6,1.82) (7,1.82) (7,2.67) (8,2.67) (8,2.9) (9,2.9) (9,3.24) (10,3.24) (10,4)};

\draw (2.5,0) -- (18.5,0);

\draw (3,0) -- (3,4);
\draw (4,0) -- (4,4);
\draw (5,0) -- (5,4);
\draw (6,0) -- (6,4);
\draw (7,0) -- (7,4);
\draw (8,0) -- (8,4);
\draw (9,0) -- (9,4);
\draw (10,0) -- (10,4);
\draw (11,0) -- (11,4);
\draw (12,0) -- (12,4);
\draw (13,0) -- (13,4);
\draw (14,0) -- (14,4);
\draw (15,0) -- (15,4);
\draw (16,0) -- (16,4);
\draw (17,0) -- (17,4);
\draw (18,0) -- (18,4);

\draw [-stealth] (19,0) -- (19,4);

\draw (3,3.2) -- (4,3.2);
\draw (3,2.3) -- (4,2.3);
\draw (3,1.1) -- (4,1.1);
\draw (4,3.8) -- (5,3.8);
\draw (4,2.0) -- (5,2.0);
\draw (4,1.2) -- (5,1.2);
\draw (4,0.1) -- (5,0.1);
\draw (5,2.7) -- (6,2.7);
\draw (5,1.11) -- (6,1.11);
\draw (5,0.24) -- (6,0.24);

\draw (6,2.9) -- (7,2.9);
\draw (6,2.32) -- (7,2.32);
\draw (6,1.82) -- (7,1.82);
\draw (6,0.35) -- (7,0.35);

\draw (7,3.23) -- (8,3.23);
\draw (7,2.67) -- (8,2.67);
\draw (7,1.45) -- (8,1.45);
\draw (7,0.7) -- (8,0.7);

\draw (8,3.55) -- (9,3.55);
\draw (8,2.9) -- (9,2.9);
\draw (8,1.9) -- (9,1.9);
\draw (8,0.35) -- (9,0.35);

\draw (9,3.24) -- (10,3.24);
\draw (9,2.31) -- (10,2.31);
\draw (9,1.24) -- (10,1.24);
\draw (9,0.7) -- (10,0.7);

\draw (10,3.1) -- (11,3.1);
\draw (10,1.4) -- (11,1.4);
\draw (10,0.5) -- (11,0.5);

\draw (11,3.2) -- (12,3.2);
\draw (11,3.6) -- (12,3.6);
\draw (11,1.1) -- (12,1.1);
\draw (11,1.0) -- (12,1.0);
\draw (11,0.3) -- (12,0.3);

\draw (12,3.5) -- (13,3.5);
\draw (12,2.6) -- (13,2.6);
\draw (12,1.34) -- (13,1.34);
\draw (12,1.08) -- (13,1.08);
\draw (12,0.6) -- (13,0.6);

\draw (13,3.9) -- (14,3.9);
\draw (13,3.) -- (14,3.);
\draw (13,1.5) -- (14,1.5);

\draw (14,3.12) -- (15,3.12);
\draw (14,2.6) -- (15,2.6);
\draw (14,1.9) -- (15,1.9);
\draw (14,1.0) -- (15,1.0);

\draw (15,3.75) -- (16,3.75);
\draw (15,2.2) -- (16,2.2);
\draw (15,1.5) -- (16,1.5);
\draw (15,1.23) -- (16,1.23);
\draw (15,0.23) -- (16,0.23);

\draw (16,3.43) -- (17,3.43);
\draw (16,2.3) -- (17,2.3);
\draw (16,1.9) -- (17,1.9);
\draw (16,1.0) -- (17,1.0);

\draw (17,3.7) -- (18,3.7);
\draw (17,2.8) -- (18,2.8);
\draw (17,1.1) -- (18,1.1);

\draw [fill=uuuuuu] (3,0) circle (2.0pt);
\draw [fill=uuuuuu] (4,0) circle (2.0pt);
\draw [fill=uuuuuu] (5,0) circle (2.0pt);
\draw [fill=uuuuuu] (6,0) circle (2.0pt);
\draw [fill=uuuuuu] (7,0) circle (2.0pt);
\draw [fill=uuuuuu] (8,0) circle (2.0pt);
\draw [fill=uuuuuu] (9,0) circle (2.0pt);
\draw [fill=uuuuuu] (10,0) circle (2.0pt);
\draw [fill=uuuuuu] (11,0) circle (2.0pt);
\draw [fill=uuuuuu] (12,0) circle (2.0pt);
\draw [fill=uuuuuu] (13,0) circle (2.0pt);
\draw [fill=uuuuuu] (14,0) circle (2.0pt);
\draw [fill=uuuuuu] (15,0) circle (2.0pt);
\draw [fill=uuuuuu] (16,0) circle (2.0pt);
\draw [fill=uuuuuu] (17,0) circle (2.0pt);
\draw [fill=uuuuuu] (18,0) circle (2.0pt);

\draw [fill=green] (8,3.23) circle (2.0pt);
\draw [fill=green] (9,2.9) circle (2.0pt);
\draw [fill=green] (10,2.31) circle (2.0pt);
\draw [fill=green] (11,1.4) circle (2.0pt);

\draw (19,2) node[anchor=north,rotate=90]{time};

\begin{scriptsize}
\draw (7,0) node[anchor=east,rotate=60]{$A_*+B_*-k\;$};

\draw (11,0) node[anchor=east,rotate=60]{$A_*+B_*\;$};
\end{scriptsize}

\draw (5,4) node[anchor=south, color=uuuuuu]{$U_-$};
\draw (12,4) node[anchor=south, color=darkgreen]{$U_+$};

\end{tikzpicture}
\caption{
An illustration of the Poisson field $\Pi$ and the areas $U_-$ and $U_+$, in the proof of Proposition \ref{prop:equal-condi-cut}.
The green points indicate space-time locations where the colors are given by $\{L_i\}_{i=1}^k$.
The red paths are the trajectories of the first two particles in $\opi^{A_{i,j}}$, for some $1\le i\le \tau$ and $1\le j \le k_i$.
}  
\label{fig:inds}
\end{figure}
For the rest of this section, we prove Proposition \ref{prop:equal-condi-cut}. The idea is as follows.
First, the cutting information can already determine some events on the configuration at $t_\tau$.
In particular, for some $\bell\in\cL\{k\}$, 
the event $R_{k+1}>t_\tau, \bL^k=\bell$ may contradict $J$.
In this case, both sides of \eqref{eq:equal-cut} equal $0$.

In the remaining case, we will show that conditional on the cutting information, $J$ and $I$ are independent, and similarly, $J$ and $I^+$ are independent.
Therefore both sides of \eqref{eq:equal-cut} equal the probability of $J$ conditional on the cutting information.
The conditional independence between $J$ and $I$ is proved by considering the Poisson field $\Pi$.
We will split the space $\Z\times [0,\infty)$ into two parts (in a way determined by the cutting information), and show that $J$ and $I$ are determined by $\Pi$ on these two parts, respectively.
\begin{proof}[Proof of Proposition \ref{prop:equal-condi-cut}]
As indicated above, the main task is to prove the independence between $J$ and $I$, conditioned on the cutting information.

\noindent\textbf{Infeasible cutting information.} We first rule out cutting information that contradicts $J$.
Namely, consider $1\le k \le C_*$ and $\bell=\{\ell_i\}_{i=1}^k\in\cL\{k\}$ that satisfy
\begin{equation}  \label{eq:equal-cut-pf1}
|\{i: 1\le i \le k, \ell_i\le C_*+A_{\tau,j}-A_* \}| < C_{\tau,j},
\end{equation}
for some $1\le j \le k_\tau$.
If $\bL^k=\bell$ and $R_{k+1} > t_\tau$, in $\eta_{t_\tau}$ all the particles on or to the right of $A_*+B_*-k$ are those with colors $\ell_1, \ldots, \ell_k$. 
We next show that this contradicts $I^{t_\tau}[A_{\tau,j},B_{\tau,j},C_{\tau,j}]$.
For this, we now consider the particles in $\eta_{t_\tau}$ that are on or to the right of $A_*+B_*-k$, with color $\le C_*+A_{\tau,j}-A_*$.
The number of such particles is at most
\[
|\{i: 1\le i \le k, \ell_i\le C_*+A_{\tau,j}-A_* \}|
\]
Then (in $\eta_{t_\tau}$) the number of particles on or to the right of $A_*+B_*+1-C_{\tau,j}$ with color $\le C_*+A_{\tau,j}-A_*$ is at most
\[
|\{i: 1\le i \le k, \ell_i\le C_*+A_{\tau,j}-A_* \}| + (C_{\tau,j}-1-k)\vee 0,
\]
which is $< C_{\tau,j}$ by \eqref{eq:equal-cut-pf1}.
This can be equivalently written as
\[
|\{x\in\Z: x\ge A_*+B_*+1-C_{\tau,j}, \eta_{t_\tau}(x)\le C_*+A_{\tau,j}-A_* \}| < C_{\tau,j}.
\]
Since $A_*+B_*\le A_{\tau,j}+B_{\tau,j}$ (by \eqref{eq:Bp}), we now have
\[
|\{x\in\Z: x\ge A_{\tau,j}+B_{\tau,j}+1-C_{\tau,j}, \eta_{t_\tau}(x)\le C_*+A_{\tau,j}-A_* \}| < C_{\tau,j},
\]
which (by Lemma \ref{lem:proj-info}) precisely implies that $T^{A_{\tau,j}}_{B_{\tau,j},C_{\tau,j}}>t_\tau$, i.e., $I^{t_\tau}[A_{\tau,j},B_{\tau,j},C_{\tau,j}]$ does not hold.
Thus we have that both sides of \eqref{eq:equal-cut} equal $0$.

\noindent\textbf{Feasible cutting information.} We next consider cutting information that makes $J$ feasible. Namely, take $1\le k \le C_*$ and $\bell=\{\ell_i\}_{i=1}^k\in\cL\{k\}$, we assume that \eqref{eq:equal-cut-pf1} is not true for any $1\le j \le k_\tau$.
Then for any $1\le i \le \tau$ and $1\le j \le k_i$, we have
\begin{equation}  \label{eq:equal-cut-pf2}
|\{i': 1\le i' \le k, \ell_{i'}\le C_*+A_{i,j}-A_* \}| \ge C_{i,j}.      
\end{equation}
(For $i=\tau$ this is precisely the contrary of \eqref{eq:equal-cut-pf1}; for $i<\tau$ this is due to $A_{i,j}\le A_{\tau, 1}$ and $C_{i,j}\le C_{\tau, 1}$ from \eqref{eq:assumabc}.)
Take $\br=\{r_i\}_{i=1}^k \in \R_<\{k, t_\tau\}$.
For simplicity of notation, below we denote $\cE_-$ as the event $\bR^k=\br, R_{k+1}>t_\tau, \bL^k=\bell$.
We will show that conditional on $\cE_-$, the events $I$ and $J$ are independent (and similarly the events $I^+$ and $J$ are also independent).

We recall the Poisson field $\Pi$ on $\Z\times [0,\infty)$, where for each $x\in\Z$ and any $0\le a < b$, $\Pi(\{x\}\times [a,b])$ is the number of times that the clock on the edge $(x-1, x)$ rings, in the time interval $[a, b]$.
Also recall that for any $A\in\Z$ and $B, C\in\N$, we denote $\cP[A, B, C]=(A+B+1-C, T^A_{B,C})$.

As indicated above, we split $\Z\times [0,\infty)$ as follows.
Denote
\[
U^-:=\{(x,t):x\le A_*+B_*-k\} \cup \bigcup_{i'=1}^k\{ (x,t): 0\le t\le r_{i'}, x\le A_*+B_*-i'+1 \},
\]
and $U^+ = \Z\times [0,\infty) \setminus U^-$. See Figure \ref{fig:inds}.
By Lemma \ref{lem:I-deter-gen}, we have that the event $\cE_-$ is measurable with respect to $\Pi$ on $U^-$,
therefore it is independent of $\Pi$ on $U^+$.

The conditional independence between $I$ and $J$ is proved in the following two steps.\\
 
\noindent\textbf{Step 1: $I$ is measurable with respect to $\Pi$ on $U^-$, conditional on $\cE_-$.}
In this step we will apply Lemma \ref{lem:I-deter-gen}; for that, we need to consider a sequence of points, which we describe now.

Under $\cE_-$, for any $1\le i' \le k$, in $\eta_{r_{i'}}$
the number of particles on or to the right of 
$A_*+B_*+1-i'$ is at least $i'$.
Back to the colored TASEP $\zeta$, this means that 
\[
|\{x\in\Z: x\ge A_*+B_*+1-i', \zeta_{r_{i'}}(x) \le A_*\}| \ge i'.
\]
For any $\tau < i \le g$, $1\le j \le k_i$, by \eqref{eq:defA} and \eqref{eq:defB} we have
\[
|\{x\in\Z: x\ge A_{i,j}+B_{i,j}+1-i', \zeta_{r_{i'}}(x) \le A_{i,j}\}| \ge i'.
\]
This is equivalent to that $T^{A_{i,j}}_{B_{i,j},i'} \le r_{i'}$.
Also, by \eqref{eq:defB} we have $A_{i,j}+B_{i,j}+1-i' \le A_*+B_*+1-i'$, so
\[
\cP[A_{i,j},B_{i,j},i']\in \{ (x,t): 0\le t\le r_{i'}, x\le A_*+B_*+1-i' \}. 
\]

For any $\tau < i \le g$, $1\le j \le k_i$, we also consider $i'$ with $k<i'\le C_{i,j}$. In this case we have $A_{i,j}+B_{i,j}+1-i' \le A_*+B_*-k$ (using \eqref{eq:defB}), so
\[
\cP[A_{i,j},B_{i,j},i']\in \{ (x,t): x\le A_*+B_*-k \}. 
\]

In conclusion, for any $\tau < i \le g$, $1\le j \le k_i$, and $1\le i'\le C_{i,j}$, we have $\cP[A_{i,j},B_{i,j},i'] \in U^-$.
Thus conditional on $\cE_-$, the event $I^{t_i}[A_{i,j},B_{i,j},C_{i,j}]$ is equivalent to 
\[
I^{t_i}[A_{i,j},B_{i,j},C_{i,j}] \cap
\left(\bigcap_{i'=1}^{B_{i,j}}I^{t_i}[A_{i,j},i',C_{i,j}] \right)\cap \{ \cP[A_{i,j},B_{i,j},i'] \in U^-,\; \forall 1\le i'\le C_{i,j} \},\]
which is determined by $\Pi$ on $U^-$, according to Lemma \ref{lem:I-deter-gen}.
Therefore $I=\bigcap_{i=\tau+1}^g\bigcap_{j=1}^{k_i}I^{t_i}[A_{i,j},B_{i,j},C_{i,j}]$ is measurable with respect to $\Pi$ on $U^-$, conditional on $\cE_-$.
Similarly, the event $I^+$ is also measurable with respect to $\Pi$ on $U^-$, conditional on $\cE_-$.
\\

\noindent\textbf{Step 2: $J$ is measurable with respect to $\Pi$ on $U^+$, conditional on $\cE_-$.}
This step is done by using Lemma \ref{lem:I-deter-bdy}.

Take any $1\le i \le \tau$ and $1\le j \le k_i$.
Suppose that the $C_{i,j}$ smallest numbers in the set $\{i': 1\le i' \le k, \ell_{i'}\le C_*+A_{i,j}-A_* \}$ are
\[
\kappa[1] < \cdots < \kappa[C_{i,j}].
\]
Such numbers exist by \eqref{eq:equal-cut-pf2}.
For each $1\le i' \le C_{i,j}$, we take $B'_{i'} \in \N$ such that $A_{i,j}+B'_{i'}-i'+1=A_*+B_*-\kappa[i']+1$.
Then $B'_{i'}\le B'_{i'+1}$ for any $1\le i'<C_{i,j}$.
Also, $T^{A_{i,j}}_{B_{i'}',i'}$ is precisely 
\[
\inf\{t>0: |\{x\in\Z: x\ge A_*+B_*+1-\kappa[i'], \eta_t(x)\le  C_*+A_{i,j}-A_*\}| \ge i'\},
\]
according to Lemma \ref{lem:proj-info}.
Under the event $\cE_-$ this equals $r_{\kappa[i']}$, thus $T^{A_{i,j}}_{B_{i'}',i'} = r_{\kappa[i']}$. 
Therefore we can apply Lemma \ref{lem:I-deter-bdy} to conclude that, conditional on $\cE_-$, the event  $I^{t_i}[A_{i,j},B_{i,j},C_{i,j}]$ is determined by $\Pi$ on
\[
\bigcup_{i'=1}^{C_{i,j}}  \{(x,t):x\ge A_{i,j}+B'_{i'}-i'+1, r_{\kappa[i']}<t\le t_i\} \subset \bigcup_{i'=1}^k  \{(x,t):x\ge A_*+B_*-i'+1, r_{i'}<t\le t_i\},
\]
which is disjoint from $U^-$, thus contained in $U^+$.\\

With the above two steps, we conclude that the events $I$ and $J$ are independent conditional on $\cE_-$, and the events $I^+$ and $J$ are also independent conditional on $\cE_-$. So $\PP[J\mid \cE_- \cap I] = \PP[J\mid \cE_-] = \PP[J\mid \cE_-\cap I^+]$, and the conclusion follows.
\end{proof}

\section{Equal probabilities conditional on cutting information}   \label{sec:equal-cut-info}
We prove Proposition \ref{prop:equal-p} in this section.
The main inputs include (1) information proved by the induction hypothesis, i.e., Theorem \ref{thm:main-des} for any smaller $\hg$, and (2) certain probability density functions of the cutting information are analytic.
We start by proving the second input.

\subsection{Analyticity}
Recall the integers $A_*, B_*, C_*$ defined via
\[
A_*=\min_{\tau < i\le g, 1\le j \le k_i} A_{i,j}, \quad  A_*+B_*=\max_{\tau < i\le g, 1\le j \le k_i}A_{i,j}+B_{i,j} \quad A_*-C_*=\min_{1 \le i \le \tau, 1\le j \le k_i}A_{i,j}-C_{i,j},
\]
where $A_{i,j}\in \Z$, $B_{i,j}, C_{i,j} \in\N, t_i>0$ are from Theorem \ref{thm:main-des}, and $1\le \tau \le g$ such that $t_\tau =\min\{t_i: 1\le i \le g\}$.
As in the previous section, we concern about $\eta=\hmu^{A_*,C_*}$, which (according to Lemma \ref{lem:law-hmu}) is a colored TASEP on $\Z$, initially containing particles at sites $A_*-C_*+1, \cdots, A_*$, with colors $1, \cdots, C_*$ respectively. 
Then recall that the cutting information consists of $\{\don[R_i \le t_\tau]R_i\}_{i=1}^{C_*}$ and $\{\don[R_i \le t_\tau]L_i\}_{i=1}^{C_*}$, where $R_0=0$, and $R_i= \inf\{t>R_{i-1}: \eta_t(A_*+B_*+1-i) < \infty\}$ is the first time after $R_{i-1}$ when the site $A_*+B_*+1-i$ is occupied by a particle in $\eta$, and $L_i=\eta_{R_i}(A_*+B_*+1-i)$ is the color of that particle.
Also recall the notations $\bR^k=\{R_i\}_{i=1}^k$, $\bL^k=\{L_i\}_{i=1}^k$, and the sets $\R_<\{k,t\}$, $\cL\{k\}$, for each $1\le k\le C_*$ and $t>0$.

We shall consider probability density functions joint with the event $I=\bigcap_{i=\tau+1}^g\bigcap_{j=1}^{k_i}I^{t_i}[A_{i,j},B_{i,j},C_{i,j}]$ on passage times, or the (partially) shifted version $I^+=\bigcap_{i=\tau+1}^g\bigcap_{j=1}^{k_i}I^{t_i}[A_{i,j}^+,B_{i,j},C_{i,j}]$. (Recall from Theorem \ref{thm:main-des} that $A_{i,j}^+ = A_{i,j}+\don[i> \iota]$ for some $1\le \iota < g$; and note we have assumed that $\tau \le \iota$.)
\begin{lemma} \label{lem:analytic}
For any $1\le k \le C_*$, $\bell\in \cL\{k\}$, and $\hat{I}=I$ or $I^+$, the function
\[
\br \mapsto 
\PP[\bR^k=\br, \bL^k=\bell, \hat{I}]
\]
is analytic in $\R_<\{k,t_\tau\}$,
and it can be analytically extended to $\C^k$.
\end{lemma}
Our proof is by writing the cutting times as sums of independent exponential random variables and using results from Section \ref{ssec:density}.
\begin{proof}[Proof of Lemma \ref{lem:analytic}]
We prove for the case where $\hat{I}=I$, and the other case where $\hat{I}=I^+$ follows similarly.

We need to consider a projection of the colored TASEP $\zeta$ that is different from $\eta=\hmu^{A_*,C_*}$, because each $T_{B_{i,j},C_{i,j}}^{A_{i,j}}$ for $\tau < i \le g$ is not determined by $\eta$.

We take integers $A_0$, $B_0$, $C_0$, which are different from $A_*$, $B_*$, $C_*$, as follows.
We let $A_0 = \max_{\tau<i\le g, 1\le j \le k_i}A_{i,j}$, and take $B_0, C_0$ such that
\[
A_0+B_0=\max_{\tau<i\le g, 1\le j \le k_i}A_{i,j}+B_{i,j},\quad
A_0-C_0=\min_{\tau<i\le g, 1\le j \le k_i}A_{i,j}-C_{i,j}.
\]
We have $B_0 \in \N$ (i.e., $B_0>0$), since there exist some $i_*, j_*$ such that $A_0=A_{i_*,j_*}$, and then $B_0\ge B_{i_*,j_*}$; we also have $C_0 \in \N$ (i.e., $C_0>0$), since $C_0\ge C_{i,j}$ for any $\tau < i \le g$ and $1\le j \le k_i$.
For any $\tau < i \le g$ and $1\le j \le k_i$, we have $A_0\ge A_{i,j}\ge A_*$ and $A_0-C_0\le A_{i,j}-C_{i,j}\le A_*-C_*$, according to \eqref{eq:defA} and \eqref{eq:Cp};
so by Lemma \ref{lem:proj-info}, the process $\eta=\hmu^{A_*,C_*}$ and the time $T_{B_{i,j},C_{i,j}}^{A_{i,j}}$ are determined by $\hmu^{A_0,C_0}$.

We consider a process $\xi$, which can be thought of as $\hmu^{A_0,C_0}$, while any particle disappears once arrives at site $A_0+B_0+1$.
To be more precise, for any $x\in\Z$ and $t\ge 0$, we let $\xi_t(x)=\infty$ if $x> A_0+B_0$, and $\xi_t(x)=\hmu_t^{A_0,C_0}(x)$ otherwise.
As $A_0+B_0=A_*+B_*\ge A_{i,j} + B_{i,j}$ for any $\tau < i \le g$ and $1\le j \le k_i$, we have that $\eta$ on $\Z\cap (-\infty, A_*+B_*]$ is determined by $\xi$, and $T_{B_{i,j},C_{i,j}}^{A_{i,j}}$ is a certain swap time in $\xi$.

Suppose the total number of swaps (of two particles, or of a particle jumping to an empty site) in the process $\xi$ is $M$.
Then we always have that $M\le B_0C_0+C_0^2$.
This bound is from the following reasoning: for the particle colored $i$, it moves from $A_0-C_0+i$ to $A_0+B_0+1$, and it can be swapped with a smaller colored particle (i.e., make a left jump) $i-1$ times;
thus it can jump to the right at most $(A_0+B_0+1) - (A_0-C_0+i) + (i-1) = B_0+C_0$ times.

For $1\le i\le M$, we take $X_i\in \Z$ such that the $i$-th swap (in $\xi$) is between sites $X_i$ and $X_i+1$, and it happens at time $w_i$. Take $w_0=0$, and for each $1\le i\le M$ we let $E_i=w_i-w_{i-1}$.
Then conditional on $M$ and $\{X_i\}_{i=1}^M$, we have that $\{E_i\}_{i=1}^M$ are independent, and the distribution for $E_i$ is $\Exp(1/a_i)$, where $a_i$ is the number of possible swaps at time $w_{i-1}$; i.e.,
\[
a_i= |\{x\in\Z: \xi_{w_{i-1}}(x) < \xi_{w_{i-1}}(x+1) \}|.
\]
Note that for each $i$, such $a_i$ is determined by $\{X_j\}_{j=1}^M$, since $\xi_{w_{i-1}}$ is determined by $\{X_j\}_{j=1}^{i-1}$.

Below we write $\bX=\{X_i\}_{i=1}^M$. From the bound of $M$, we have that there are finitely many $\bx=\{x_i\}_{i=1}^m$, such that $\PP[\bX=\bx] > 0$.
It remains to show that for any such $\bx$, the function
\[
\br \mapsto 
\PP[\bR^k=\br, \bL^k=\bell, I \mid \bX=\bx]
\]
is analytic in $\R_<\{k,t_\tau\}$,
and it can be analytically extended to $\C^k$.

We note that $\bL^k$ is determined by $\bX$, so it remains to consider the function
\begin{equation}  \label{eq:r-func}
\br \mapsto 
\PP[\bR^k=\br, I \mid \bX=\bx].    
\end{equation}
For each $1\le i \le k$, we have $R_i \in \{w_{i'}\}_{i'=1}^M$; and we let $\kappa[i]$ be the number such that
$R_i=w_{\kappa[i]}$.
Then we have $\kappa[1]<\kappa[2]<\cdots<\kappa[k]$.
For each $\tau< i\le g$ and $1\le j\le k_i$, we have $T^{A_{i,j}}_{B_{i,j},C_{i,j}} \in \{w_{i'}\}_{i'=1}^M$; and we let $\hkappa[i]$ be the number such that
$\max_{1\le j\le k_i} T^{A_{i,j}}_{B_{i,j},C_{i,j}} =w_{\hkappa[i]}$.
We note that the indices $\{\kappa[i]\}_{i=1}^k$ and $\{\hkappa[i]\}_{i=\tau+1}^g$ are determined by $\bX$.

Since $\{E_i\}_{i=1}^M$ are independent exponential random variables,
we can then write the function \eqref{eq:r-func} as (supposing that $\br=\{r_i\}_{i=1}^k$, $r_0=0$, and $\kappa[0]=0$)
\[
\PP\left[\sum_{i'=\kappa[k]+1}^{\hkappa[i]} E_{i'} < t_i-r_k, \; \forall \tau< i\le g, \hkappa[i] > \kappa[k]\mid \bX=\bx\right]
\prod_{i=1}^k\PP\left[\sum_{i'=\kappa[i-1]+1}^{\kappa[i]} E_{i'} = r_i-r_{i-1}\mid \bX=\bx\right].
\]
By applying Lemma \ref{lem:analy-single} or Lemma \ref{lem:analy-multi} to each factor, we get that this is analytic in $\R_<\{k,t_\tau\}$, and can be analytically extended to $\C^k$.
Thus the conclusion follows.
\end{proof}

\subsection{Induction hypothesis with fewer colors}  \label{ssec:hywfc}
For the rest of this section, we assume that Theorem \ref{thm:main-des} is true for any smaller $\hg$.
We now state the input from this induction hypothesis, through the following event.\\

\noindent\textbf{High-speed event $\cE$.}
Take any $1\le k \le C_*$, and let $\lambda\in\N$.
For each $1\le i\le k$ we denote $m_i=2^{\lambda^{k+1-i}}$.
We assume that $\lambda$ is large enough so that $m_1>\cdots > m_k>A_*+B_*$. 
Denote $t_*=t_{\tau+1}\ge t_\tau$. 
Take $\bell=\{\ell_i\}_{i=1}^k\in\cL\{k\}$, and denote $\cE$ as the event where $\{\eta_{t_*}(m_i)\}_{i=1}^k=\bell$ and $\eta_{t_*}(x)=\infty, x\in [m_k,\infty)\setminus \{m_i\}_{i=1}^k$.
In words, $\cE$ is the event that in $\eta_{t_*}$, the particle colored $\ell_i$ is at site $m_i$. See Figure \ref{fig:EEpI} for an illustration.

We design $\cE$ this way because later we will send $\lambda\to \infty$, so that all these $m_i$ are well-separated.
Then under $\cE$, the particles with colors in $\bell$ have high and very different speeds, thus their evolutions are roughly independent of each other.
This enables us to extract the cutting information. The procedure of sending $\lambda\to\infty$ is similar to taking $B'$ large enough in Example \ref{ex:simple}.
\begin{lemma} \label{lem:equal-p1}
We have $\PP[\cE \mid I] = \PP[\cE \mid I^+]$.
\end{lemma}
We now explain how this can be deduced from the induction hypothesis of Theorem \ref{thm:main-des}.
By \eqref{eq:IIpeq}, we need to show that $\PP[\cE\cap I] = \PP[\cE\cap I^+]$.
The event $\cE$ can be characterized by all the random variables $\don[T^A_{B,C}\le t_*]=\don[I^{t_*}[A,B,C]]$, with $A\le A_*$, $A-C\ge A_*-C_*$, $A+B-C+1\ge m_k$.
Then we use the induction hypothesis of Theorem \ref{thm:main-des}, and the principle of inclusion-exclusion,
to show that the joint law of these random variables with $\don[I]$ is the same as the joint law of these random variables with $\don[I^+]$.

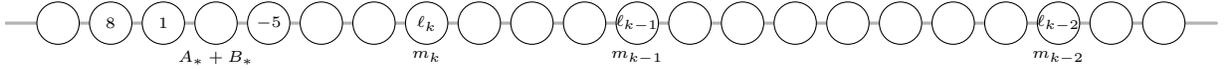
\begin{figure}[hbt!]
    \centering
\begin{tikzpicture}[line cap=round,line join=round,>=triangle 45,x=7cm,y=7cm]
\clip(-2.7,-0.7) rectangle (-0.35,-0.5);

\draw [line width=1.2pt, opacity=0.3] (-2.7,-0.6) -- (-0.4,-0.6);
\draw [fill=white] (-2.6,-0.6) circle (8.0pt);
\draw [fill=white] (-2.5,-0.6) circle (8.0pt);
\draw [fill=white] (-2.4,-0.6) circle (8.0pt);
\draw [fill=white] (-2.3,-0.6) circle (8.0pt);
\draw [fill=white] (-2.2,-0.6) circle (8.0pt);
\draw [fill=white] (-2.1,-0.6) circle (8.0pt);
\draw [fill=white] (-2.,-0.6) circle (8.0pt);
\draw [fill=white] (-1.9,-0.6) circle (8.0pt);
\draw [fill=white] (-1.8,-0.6) circle (8.0pt);
\draw [fill=white] (-1.7,-0.6) circle (8.0pt);
\draw [fill=white] (-1.6,-0.6) circle (8.0pt);
\draw [fill=white] (-1.5,-0.6) circle (8.0pt);
\draw [fill=white] (-1.4,-0.6) circle (8.0pt);
\draw [fill=white] (-1.3,-0.6) circle (8.0pt);
\draw [fill=white] (-1.2,-0.6) circle (8.0pt);
\draw [fill=white] (-1.1,-0.6) circle (8.0pt);
\draw [fill=white] (-1.,-0.6) circle (8.0pt);
\draw [fill=white] (-0.9,-0.6) circle (8.0pt);
\draw [fill=white] (-0.8,-0.6) circle (8.0pt);
\draw [fill=white] (-0.7,-0.6) circle (8.0pt);
\draw [fill=white] (-0.6,-0.6) circle (8.0pt);
\draw [fill=white] (-0.5,-0.6) circle (8.0pt);

\begin{tiny}
\draw (-2.6,0.8) node[anchor=center]{$-6$};
\draw (-2.5,0.8) node[anchor=center]{$-1$};
\draw (-2.4,0.8) node[anchor=center]{$-7$};
\draw (-2.3,0.8) node[anchor=center]{$0$};
\draw (-2.2,0.8) node[anchor=center]{$-4$};
\draw (-2.1,0.8) node[anchor=center]{$-3$};
\draw (-2.0,0.8) node[anchor=center]{$-5$};
\draw (-1.9,0.8) node[anchor=center]{$2$};
\draw (-1.8,0.8) node[anchor=center]{$4$};
\draw (-1.7,0.8) node[anchor=center]{$-9$};
\draw (-1.6,0.8) node[anchor=center]{$-2$};
\draw (-1.5,0.8) node[anchor=center]{$5$};
\draw (-1.4,0.8) node[anchor=center]{$3$};
\draw (-1.3,0.8) node[anchor=center]{$10$};
\draw (-1.2,0.8) node[anchor=center]{$7$};
\draw (-1.1,0.8) node[anchor=center]{$6$};
\draw (-1.0,0.8) node[anchor=center]{$1$};
\draw (-0.9,0.8) node[anchor=center]{$17$};
\draw (-0.8,0.8) node[anchor=center]{$9$};
\draw (-0.7,0.8) node[anchor=center]{$14$};
\draw (-0.6,0.8) node[anchor=center]{$12$};
\draw (-0.5,0.8) node[anchor=center]{$8$};

\draw (-2.5,-0.6) node[anchor=center]{$8$};
\draw (-2.4,-0.6) node[anchor=center]{$1$};
\draw (-2.2,-0.6) node[anchor=center]{$-5$};
\draw (-1.9,-0.6) node[anchor=center]{$\ell_k$};
\draw (-1.5,-0.6) node[anchor=center]{$\ell_{k-1}$};
\draw (-0.7,-0.6) node[anchor=center]{$\ell_{k-2}$};

\draw (-2.3,-0.64) node[anchor=north]{$A_*+B_*$};
\draw (-1.9,-0.64) node[anchor=north]{$m_k$};
\draw (-1.5,-0.64) node[anchor=north]{$m_{k-1}$};
\draw (-0.7,-0.64) node[anchor=north]{$m_{k-2}$};
\end{tiny}

\end{tikzpicture}
\caption{An illustration of the event $\cE$ on $\eta_{t_*}=\hmu^{A_*,C_*}_{t_*}$: for each $1\le i \le k$, there is a particle at site $m_i$ with label $m_i$; and there is no other particle to the right of $m_k$.}  
\label{fig:EEpI}
\end{figure}
\begin{proof}[Proof of Lemma \ref{lem:equal-p1}]
We note that $\cE$ is equivalent to the following event (denoted as $\cE'$):
\begin{enumerate}
    \item $T^A_{B,C} \le t_*$, for any $A\le A_*$, $A-C\ge A_*-C_*$, $A+B-C+1 \ge m_k$, such that $|\{1\le i \le k: A_*-C_*+\ell_i\le A, m_i\ge A+B-C+1\}|\ge C$;
    \item $T^A_{B,C} > t_*$, for any $A\le A_*$, $A-C\ge A_*-C_*$, $A+B-C+1 \ge m_k$, such that $|\{1\le i \le k: A_*-C_*+\ell_i\le A, m_i\ge A+B-C+1\}|< C$.
\end{enumerate}
It is straightforward to check that $\cE'$ is implied by $\cE$.
For the other direction: from $\cE'$ we can determine $|\{x\ge x_0: \eta_{t_*}(x)\le i\}|$ for each $x_0\ge m_k$ and each $1\le i \le C_*$ (by Lemma \ref{lem:proj-info});
thus by varying $x_0$ we can determine the set $\{x\ge m_k: \eta_{t_*}(x)\le i\}$ for each $1\le i \le C_*$; and by further varying $i$ we can determine $\eta_{t_*}(x)$ for each $x\ge m_k$, and this can determine $\cE$.

We also note that $\cE'$ can be written as the intersection of $I^{t_*}[A,B,C]$ for all $A,B,C$ in the first point above, minus the union of $I^{t_*}[A,B,C]$ for all $A,B,C$ in the second point above.
To show that $\PP[\cE'\cap I] = \PP[\cE' \cap I^+]$, we consider the joint distributions of
\[
\{\don[I^{t_*}[A,B,C]]\}_{A\le A_*, A-C\ge A_*-C_*, A+B-C+1 \ge m_k}, \don[I],
\]
and
\[
\{\don[I^{t_*}[A,B,C]]\}_{A\le A_*, A-C\ge A_*-C_*, A+B-C+1 \ge m_k}, \don[I^+].
\]
It then suffices to show that the above two sets of random variables are equal in distribution.
For this, we take an arbitrary subset of $\{(A, B, C): A\le A_*$, $A-C\ge A_*-C_*$, $A+B-C+1 \ge m_k\}$, denoted by $\{(A_{0,j}, B_{0,j}, C_{0,j})\}_{j=1}^{k_0} \in \Z^{k_0} \times \N^{k_0}\times \N^{k_0}$, for some $k_0 \in \N$.
We claim that
\begin{equation}  \label{eq:IABCeq}
\PP\left[\bigcap_{j=1}^{k_0}I^{t_*}[A_{0,j}, B_{0,j}, C_{0,j}] \cap I\right] = \PP\left[\bigcap_{j=1}^{k_0}I^{t_*}[A_{0,j}, B_{0,j}, C_{0,j}] \cap I^+\right].    
\end{equation}
This follows from the induction hypothesis of Theorem \ref{thm:main-des}, and below we explain how the parameters are taken, and check that all the required conditions are satisfied (in particular, $\hg$ is smaller).
We apply Theorem \ref{thm:main-des} to $g'\in\N$, $k_1',\ldots, k_{g'}'\in\N$, and $A_{i,j}'\in \Z$, $B_{i,j}', C_{i,j}' \in\N$ for $1\le i \le g'$ and $1\le j \le k_i'$, and $1\le \iota' < g'$, $t_1',\ldots, t_{g'}'$, which are taken as follows: if $\tau < \iota$, we have
\begin{enumerate}
    \item $g'=g-\tau$, $\iota'=\iota-\tau$, $t_i'=t_{\tau+i}$ for each $1\le i \le g'$,
    \item $k_1'=k_{\tau+1}+k_0$, $k_i' = k_{\tau+i}$ for each $1<i\le g'$,
    \item $\{A_{1,j}'\}_{j=1}^{k_1'} = \{A_{0,j}\}_{j=1}^{k_0}\cup \{A_{\tau+1,j}\}_{j=1}^{k_{\tau+1}}$, $\{B_{1,j}'\}_{j=1}^{k_1'} = \{B_{0,j}\}_{j=1}^{k_0}\cup \{B_{\tau+1,j}\}_{j=1}^{k_{\tau+1}}$, $\{C_{1,j}'\}_{j=1}^{k_1'} = \{C_{0,j}\}_{j=1}^{k_0}\cup \{C_{\tau+1,j}\}_{j=1}^{k_{\tau+1}}$; and $A_{i,j}'=A_{\tau+i,j}$, $B_{i,j}'=B_{\tau+i,j}$, $C_{i,j}'=C_{\tau+i,j}$ for each $1<i\le g'$ and $1\le j \le k_i'$,
\end{enumerate}
and if $\tau=\iota$, we have
\begin{enumerate}
    \item $g'=g-\tau+1$, $\iota'=1$, $t_1'=t_*=t_{\tau+1}$, and $t_i'=t_{\tau+i-1}$ for each $1< i \le g'$,
    \item $k_1'=k_0$, $k_i' = k_{\tau+i-1}$ for each $1<i\le g'$,
    \item $\{A_{1,j}'\}_{j=1}^{k_1'} = \{A_{0,j}\}_{j=1}^{k_0}$, $\{B_{1,j}'\}_{j=1}^{k_1'} = \{B_{0,j}\}_{j=1}^{k_0}$, $\{C_{1,j}'\}_{j=1}^{k_1'} = \{C_{0,j}\}_{j=1}^{k_0}$; and $A_{i,j}'=A_{\tau+i-1,j}$, $B_{i,j}'=B_{\tau+i-1,j}$, $C_{i,j}'=C_{\tau+i-1,j}$ for each $1<i\le g'$ and $1\le j \le k_i'$.
\end{enumerate}
In words, we take the groups of passage times with indexes from $\tau+1$ to $g$, and also the extra group $\{(A_{0,j}, B_{0,j}, C_{0,j})\}_{j=1}^{k_0}$. Whenever $\tau<\iota$ we combine this new group and the one with index $\tau+1$. (In the case where $\tau=\iota$, we cannot combine these two groups because a relative shift happens between them.)

Under this setting, if $t_\iota \neq t_{\iota+1}$, the corresponding $\hg'$ equals $g'=g-\tau<g=\hg$; otherwise, if $\tau < \iota$ we have $\hg'=g'-1=g-\tau-1<g-1=\hg$, and if $\tau = \iota$ we have $\tau\ge 2$, so $\hg'=g'-1=g-\tau < g-1=\hg$.
Thus by the induction hypothesis, we conclude that \eqref{eq:IABCeq} holds.
Then by the principle of inclusion-exclusion, and taking different $\{(A_{0,j}, B_{0,j}, C_{0,j})\}_{j=1}^{k_0}$, we get $\PP[\cE'\cap I] = \PP[\cE' \cap I^+]$. So with \eqref{eq:IIpeq} the conclusion follows.
\end{proof}

\subsection{Partial order and induction of cutting colors}
Let $\Lambda=\bigcup_{k=1}^{C_*}\cL\{k\}$.
We define a partial order $\prec$ on the $\Lambda$: for any $\bell$ and $\bell'$, we let $\bell\prec\bell'$, if one of the following conditions are satisfied:
\begin{enumerate}
    \item $\bell\in\cL\{k\}$ and $\bell'\in\cL\{k'\}$, for some $1\le k<k'\le C_*$, and the first $k'$ coordinates of $\bell$ and $\bell'$ are the same.
    \item Both $\bell, \bell'\in\cL\{k\}$ for some $1\le k \le C_*$, and one can apply a `swap operation' to $\bell$ to obtain $\bell'$; i.e., if we write $\bell=\{\ell_i\}_{i=1}^{k}$ and $\bell'=\{\ell_i'\}_{i=1}^{k}$, then there exists some $1\le i <k$ with $\ell_i>\ell_{i+1}$, such that $\ell_i'=\ell_{i+1}$ and $\ell_{i+1}=\ell_i'$, and $\ell_j=\ell_j'$ for any $1\le j \le k$, $j\not\in\{i,i+1\}$.
\end{enumerate}
Also, suppose $\bell\prec\bell'$ and $\bell'\prec\bell''$, we let $\bell\prec\bell''$.
In other words, for any $\bell\in \Lambda$, the sequences that are greater than it under $\prec$ are those that can be obtained in the following way: first, apply a sequence of swap operations; then take the first a few coordinates.
Under this ordering, $\{C_*, \cdots, 1\}$ is the minimum element in $\Lambda$, and $\emptyset$ is the maximum element in $\Lambda$.

\begin{ex}   \label{ex:order}
For $C_*=5$, we have $\{3,5,2,1\}\prec\{2,3\}$, because we can take a sequence of swap operations $\{3,5,2,1\}\to \{3,2,5,1\}\to\{2,3,5,1\}$, and take the first two numbers.
On the other hand, we must have  $\{3,5,2,1\}\not\prec\{5,2\}$, since any sequence of swap operations cannot move $5$ to the beginning.
\end{ex}

For simplicity of notation, we let $\bell\preceq\bell'$ if $\bell\prec\bell'$ or $\bell=\bell'$.

We prove Proposition \ref{prop:equal-p} using Lemma \ref{lem:equal-p1} and induction in $\bell \in \Lambda$, under the ordering $\prec$. Below we shall fix $1\le k \le C_*$ and $\bell =\{\ell_i\}_{i=1}^{k}\in \cL\{k\}$, and assume that Proposition \ref{prop:equal-p} is true for any $\bell'\prec\bell$.
We next set up the following transition probabilities.
Recall the vectors $\bR^{k'}=\{R_i\}_{i=1}^{k'}$, $\bL^k=\{L_i\}_{i=1}^{k'}$ for each $1\le k'\le C_*$, and that the cutting information consists of $\{\don[R_i \le t_\tau]R_i\}_{i=1}^{C_*}, \{\don[R_i \le t_\tau]L_i\}_{i=1}^{C_*}$.
\\

\noindent\textbf{Transition probabilities.} For each $1\le k'\le C_*$, $\bell'\in\cL\{k'\}$, $\br\in\R_<\{k',t_\tau\}$, and $\hat{I}=I$ or $I^+$, we denote
\[
G(\bell', \br, \hat{I}) =
\PP[\bR^{k'}=\br, R_{k'+1}>t_\tau, \bL^{k'}=\bell' \mid \hat{I}].
\]
In words, $G$ is the distribution density function of the cutting information (up to time $t_\tau$), conditioned on $I$ or $I^+$.
The identity \eqref{eq:equal-p} in the statement of Proposition \ref{prop:equal-p} is then equivalent to 
\begin{equation}  \label{eq:equal-pe}
G(\bell, \br, I) = G(\bell, \br, I^+),
\end{equation}
for any $\br\in \R_<\{k,t_\tau\}$, and our goal is to prove it.
The induction hypothesis of Proposition \ref{prop:equal-p} is translated into the fact that 
\begin{equation}  \label{eq:equal-pei}
G(\bell', \br, I) = G(\bell', \br, I^+),
\end{equation}
for any $k\le k'\le C_*$, $\bell'\prec\bell$ with $\bell'\in\cL\{k'\}$, and any $\br\in \R_<\{k',t_\tau\}$.

Define the event $\cE$ to be the `high-speed event' as in the previous subsection, for the fixed $\bell$.
For $\hat{I}=I$ or $I^+$, and each $1\le k'\le C_*$, $\bell'\in\cL\{k'\}$, $\br\in\R_<\{k',t_\tau\}$, we denote
\[
F(\bell', \br, \hat{I}) = \PP[\cE \mid \bR^{k'}=\br, R_{k'+1}>t_\tau, \bL^{k'}=\bell', \hat{I}].
\]
In words, $F$ is the probability of $\cE$ conditioned on the cutting information and $\hat{I}$.
We note that the definition of $F$ depends on $\bell$, since the event $\cE$ depends on $\bell$.\\

Now we can expand $\PP[\cE\mid \hat{I}]$ as follows:
\begin{equation}  \label{eq:wr-ce}
\begin{split}
\PP[\cE\mid \hat{I}] &= \sum_{k',\bell'\in\Lambda\{k'\}} \int_{\br\in \R_<\{k',t_\tau\} }
\PP[\cE, \bR^{k'}=\br, R_{k'+1}>t_\tau, \bL^{k'}=\bell'\mid \hat{I}] d\br \\
&=\sum_{k',\bell'\in\Lambda\{k'\}} \int_{\br\in \R_<\{k',t_\tau\}}
F(\bell', \br, \hat{I}) G(\bell', \br, \hat{I})d\br.
\end{split}
\end{equation}
By Lemma \ref{lem:equal-p1}, the second line of \eqref{eq:wr-ce} is the same for $\hat{I}=I$ or $I^+$; namely, we have
\[
\sum_{k',\bell'\in\Lambda\{k'\}} \int_{\br\in \R_<\{k',t_\tau\}}
F(\bell', \br, {I}) G(\bell', \br, {I})d\br
=
\sum_{k',\bell'\in\Lambda\{k'\}} \int_{\br\in \R_<\{k',t_\tau\}}
F(\bell', \br, {I^+}) G(\bell', \br, {I^+})d\br.
\]
We will analyze the summand for each $\bell'\in\Lambda$:
\begin{itemize}
    \item For $\bell'\not\preceq\bell$, we will show that $F(\bell', \br, \hat{I})$ is very small (compared to $F(\bell, \br, \hat{I})$), when $\lambda \to \infty$ (Lemma \ref{lem:npr} below). Therefore each summand in both sides with such $\bell'$ can be ignored in the $\lambda \to \infty$ limit.
    \item For $\bell'\preceq\bell$, we will show that $|F(\bell', \br, I)-F(\bell', \br, I^+)|$ is very small (compared to $F(\bell, \br, \hat{I})$), when $\lambda \to \infty$ (Lemma \ref{lem:pr} below). 
    As a consequence, 
    for the summands in the right-hand side such $\bell'$, we can replace $F(\bell', \br, I^+)$ by $F(\bell', \br, I)$.
    Using the induction hypothesis \eqref{eq:equal-pei}, we can ignore the summands with $\bell'\prec\bell$ (in both sides).
    Therefore we conclude that the integrals
    \[
    \int_{\br\in \R_<\{k,t_\tau\}}
F(\bell, \br, I) G(\bell, \br, I) d\br
    \]
    and 
        \[
    \int_{\br\in \R_<\{k,t_\tau\}}
F(\bell, \br, I) G(\bell, \br, I^+) d\br
    \]
    are roughly equal, as $\lambda\to\infty$.
    \item From the rough inequality between the above two integrals, we can extract the equality between $G(\bell, \cdot, I^+)$ and $G(\bell, \cdot, I)$ (i.e., prove \eqref{eq:equal-pe}), by using (1) the fact that they are both analytic, and (2) refined $\lambda\to\infty$ asymptotic behavior of $F(\bell, \cdot, I)$ (Lemma \ref{lem:asy-m-t} below). 
\end{itemize}

As already indicated, the main steps in the above strategy are given by various estimates on $F$, i.e., the probability of $\cE$ conditioned on the cutting information (up to time $t_\tau$) and $\hat{I}$.
Below we give these estimates.

Recall that $\cE$ is the event that in $\eta_{t_*}$, the particle colored $\ell_i$ is at site $m_i=2^{\lambda^{k+1-i}}$ (for each $1\le i \le k$).
When $\lambda$ is large, this requires that the particle colored $\ell_i$ move roughly $m_i$ steps within $t_*$ time.
Thus $F$ should be of order at most $\prod_{i=1}^k \frac{t_*^{m_i}}{m_i!}$.
Our first estimate states that conditioned on certain cutting information, $F$ is smaller.
\begin{lemma}  \label{lem:npr}
For any $1\le k'\le C_*$, $\bell'\in\cL\{k'\}$, $\bell'\not\preceq\bell$, and $\alpha_1, \cdots, \alpha_k \in \Z_{\ge 0}$, and $\hat{I} = I$ or $\hat{I}=I^+$, we have
\[
\prod_{i=1}^{k} \frac{m_i^{\alpha_i}\cdot m_i!}{t_*^{m_i}}
\int_{\br\in \R_<\{k', t_\tau\}}
F(\bell', \br, \hat{I}) d\br \to 0,
\]
as we send $\lambda\to \infty$.
\end{lemma}
We now explain this lemma heuristically.
Suppose that $\bL^{k'}=\bell'$ and $\cE$ holds. If each number in $\bell$ is also contained in $\bell'$,
we must have that $\bell'\preceq\bell$, by the definition of the ordering $\prec$.
Therefore, for any $\bell'\not\preceq\bell$, the events $\bL^{k'}=\bell'$ and $\cE$ imply that there is a number $\ell_j$ in $\bell$ that is not contained in $\bell'$.
If in addition $R_{k'+1}>t_\tau$ is assumed, in $\eta_{t_\tau}$ the particle colored $\ell_j$ is still on or to the left of site $A_*+B_*-k'$.
Thus this particle needs to move roughly $m_j$ steps between time $t_\tau$ and time $t_*$, and that happens with probability of order at most $\frac{(t_*-t_\tau)^{m_j}}{m_j!}$.
In summary, for $\bell'\not\preceq\bell$ the function $F(\bell', \br, \hat{I})$ is of order at most $((t_*-t_\tau)/t_*)^{m_j}\prod_{i=1}^k \frac{t_*^{m_i}}{m_i!}$, which is much smaller than $\prod_{i=1}^{k} \frac{t_*^{m_i}}{m_i^{\alpha_i}\cdot m_i!}$ for large $\lambda$.

Our second estimate compares the probabilities of $\cE$ conditioned on $I$ or $I^+$ (i.e., $F(\cdot,\cdot,I)$ versus $F(\cdot,\cdot,I^+)$).
\begin{lemma}  \label{lem:pr}
For any $k\le k'\le C_*$, $\bell'\in\cL\{k'\}$, $\bell'\preceq\bell$, and $\alpha_1, \cdots, \alpha_k \in \Z_{\ge 0}$, we have
\[
\prod_{i=1}^{k} \frac{m_i^{\alpha_i}\cdot m_i!}{t_*^{m_i}}
\int_{\br\in \R_<\{k', t_\tau\}}
\big|F(\bell', \br, I) 
-F(\bell', \br, I^+) \big|
d\br \to 0,
\]
as we send $\lambda\to \infty$.
\end{lemma}
The heuristics behind this lemma is as follows.
In $F(\bell', \br, I)$ and $F(\bell', \br, I^+)$, the cutting information up to time $t_\tau$ are conditioned to be the same.
The difference is on $I$ versus $I^+$, and that can only affect the cutting information after time $t_\tau$.
For the cutting information after time $t_\tau$ to affect the event $\cE$, there must exist a particle that is on or to the left of $A_*+B_*-k'$ in $\eta_{t_\tau}$, and interacts with the particles with colors in $\bell$ before time $t_*$.
Then this particular particle should move at least order $m_k$ steps before time $t_*$, because under $\cE$, at time $t_\tau$ each particle with color in  $\bell$ is likely to have already moved order $m_k$ steps.
Note that the probability for a particle to move $cm_k$ steps in $t_*-t_\tau$ time is roughly $\frac{(t_*-t_\tau)^{cm_k}}{(cm_k)!}$ (here $c>0$ is a constant).
Then the difference $|F(\bell', \br, I) 
-F(\bell', \br, I^+)|$ is expected to be of order at most $\frac{(t_*-t_\tau)^{cm_k}}{(cm_k)!}\prod_{i=1}^k \frac{t_*^{m_i}}{m_i!}$, which is much smaller than $\prod_{i=1}^{k} \frac{t_*^{m_i}}{m_i^{\alpha_i}\cdot m_i!}$ for large $\lambda$.

The next result is a more refined asymptotic result of $F$, with a scaling of times.
\begin{lemma}  \label{lem:asy-m-t}
Take any $s_1,\ldots,s_k>0$.
As we send $\lambda\to \infty$, there is
\begin{equation}\label{eq:asy-m-t}
\prod_{i=1}^k \frac{(m_i-(A_*+B_*+1-i))!}{t_*^{m_i-(A_*+B_*+1-i)}}
F(\bell, \{\frac{s_i}{m_i}\}_{i=1}^{k}, I) \to e^{-kt_*}\prod_{i=1}^k e^{-s_i/t_*}.     
\end{equation}
In addition, when $\lambda$ is large enough (depending on $A_*, B_*, C_*, t_*, t_\tau$), for any $0<\frac{s_1}{m_1}<\cdots<\frac{s_k}{m_k}<t_\tau$, the left-hand side is bounded by $\prod_{i=1}^k 2e^{-s_i/t_*}$.
\end{lemma}
This lemma can be explained by the following more careful analysis of the event $\cE$.
For each $1\le i\le k$, the events $R_i=\frac{s_i}{m_i}$ and $L_i=\ell_i$ imply that at time $\frac{s_i}{m_i}$, the particle colored $\ell_i$ is at site $A_*+B_*+1-i$.
For $\cE$ to happen, this particle needs to move $m_i-(A_*+B_*+1-i)$ steps between time $\frac{s_i}{m_i}$ and time $t_*$.
As $m_i=2^{\lambda^{k+1-i}}$, when $\lambda$ is large, these $k$ particles travel at very different and high speeds, and one can regard their moves as being independent of each other.
Thus $F(\bell, \{\frac{s_i}{m_i}\}_{i=1}^{k}, I)$ should be roughly $\prod_{i=1}^k e^{-t_*}\frac{(t_*-s_i/m_i)^{m_i-(A_*+B_*+1-i)}}{(m_i-(A_*+B_*+1-i))!}$, and \eqref{eq:asy-m-t} follows.

We leave the proofs of the above lemmas to the next subsections, and now finish the proof of Proposition \ref{prop:equal-p} assuming them.
\begin{proof}[Proof of Proposition \ref{prop:equal-p}]
For $\PP[\cE \mid I] = \PP[\cE \mid I^+]$ from Lemma \ref{lem:equal-p1}, we apply \eqref{eq:wr-ce} to both hand sides.
Take any $\alpha_1, \cdots, \alpha_k \in \Z_{\ge 0}$, we multiply both sides by $\prod_{i=1}^{k} \frac{m_i^{\alpha_i}\cdot m_i!}{t_*^{m_i}}$,  and get
\begin{align*}
&\sum_{k',\bell'\in\Lambda\{k'\}}
\prod_{i=1}^{k} \frac{m_i^{\alpha_i}\cdot m_i!}{t_*^{m_i}} \int_{\br\in \R_<\{k',t_\tau\} }
F(\bell', \br, I) G(\bell', \br, I)d\br
\\
=
&\sum_{k',\bell'\in\Lambda\{k'\}}
\prod_{i=1}^{k} \frac{m_i^{\alpha_i}\cdot m_i!}{t_*^{m_i}} \int_{\br\in \R_<\{k',t_\tau\} }
F(\bell', \br, I^+) G(\bell', \br, I^+)d\br.
\end{align*}
We will then send $\lambda\to\infty$, and analyze the behavior of each term in the summations on both sides.

We claim that $G(\bell', \br, \hat{I})$ is bounded, uniformly in $\hat{I} = I$ or $\hat{I}=I^+$, and $\bell'\in\Lambda$, $\br\in \R_<\{k',t_\tau\}$.
Indeed, we have
\begin{align*}
G(\bell', \br, \hat{I}) = & \PP[\bR^{k'}=\br, R_{k'+1}>t_\tau, \bL^{k'}=\bell' \mid \hat{I}] \\
\le &\PP[\hat{I}]^{-1}\PP[\bR^{k'}=\br, \bL^{k'}=\bell', \hat{I}].
\end{align*}
By Lemma \ref{lem:analytic}, for each $\bell'$ and $\hat{I}$ this is analytic as a function of $\br\in \R_<\{k',t_\tau\}$ with an extension to $\C^{k'}$, thus must be bounded uniformly in $\br$.
As there are only finitely many choices of $\bell'$ and $\hat{I}$, the above claim then follows.

With this claim, by Lemma \ref{lem:npr}, 
for any $\bell'\not\preceq\bell$ and $\hat{I}=I$ or $I^+$, we have
\[
\prod_{i=1}^{k} \frac{m_i^{\alpha_i}\cdot m_i!}{t_*^{m_i}}
\int_{\br\in \R_<\{k',t_\tau\}}
F(\bell', \br, \hat{I})
G(\bell', \br, \hat{I})
d\br \to 0,
\]
as $\lambda\to\infty$.
For any $\bell'\prec\bell$, by the induction hypothesis (of this proposition), we have that $G(\bell', \br, I)
=G(\bell', \br, I^+)$.
Thus the difference
\[
\prod_{i=1}^{k} \frac{m_i^{\alpha_i}\cdot m_i!}{t_*^{m_i}}
\int_{\br\in \R_<\{k',t_\tau\}}
F(\bell', \br, I)
G(\bell', \br, I)
d\br
-
\prod_{i=1}^{k} \frac{m_i^{\alpha_i}\cdot m_i!}{t_*^{m_i}}
\int_{\br\in \R_<\{k',t_\tau\}}
F(\bell', \br, I^+)
G(\bell', \br, I^+)
d\br
\]
equals
\[
\prod_{i=1}^{k} \frac{m_i^{\alpha_i}\cdot m_i!}{t_*^{m_i}}
\int_{\br\in \R_<\{k',t_\tau\}}
\big(
F(\bell', \br, I)
- F(\bell', \br, I^+)
\big)
G(\bell', \br, I) d\br.
\]
As $G(\bell', \br, I)$ is uniformly bounded, by Lemma \ref{lem:pr} we have that this expression $\to 0$ as $\lambda\to\infty$.

Now only the term where $\bell'=\bell$ remains, i.e., as $\lambda\to\infty$ we have
\[
\prod_{i=1}^{k} \frac{m_i^{\alpha_i}\cdot m_i!}{t_*^{m_i}}
\int_{\br\in \R_<\{k,t_\tau\}}
F(\bell, \br, I)
G(\bell, \br, I)
d\br
-
\prod_{i=1}^{k} \frac{m_i^{\alpha_i}\cdot m_i!}{t_*^{m_i}}
\int_{\br\in \R_<\{k,t_\tau\}}
F(\bell, \br, I^+)
G(\bell, \br, I^+)
d\br \to 0.
\]
By Lemma \ref{lem:pr} we have that as $\lambda\to\infty$,
\[
\prod_{i=1}^{k} \frac{m_i^{\alpha_i}\cdot m_i!}{t_*^{m_i}}
\int_{\br\in \R_<\{k,t_\tau\}}
(F(\bell, \br, I)
-
F(\bell, \br, I^+))
G(\bell, \br, I^+)
d\br \to 0,
\]
since $G(\bell, \br, I^+)$ is uniformly bounded.
Thus we have that
\begin{equation}  \label{eq:diff-cov}
\prod_{i=1}^{k} \frac{m_i^{\alpha_i}\cdot m_i!}{t_*^{m_i}}
\int_{\br\in \R_<\{k,t_\tau\}}
F(\bell, \br, I) (G(\bell, \br, I)
- G(\bell, \br, I^+) )
d\br \to 0,
\end{equation}
as we send $\lambda\to \infty$.

By Lemma \ref{lem:analytic}, the function
\begin{equation}  \label{eq:rmap}
\br \mapsto G(\bell, \br, I)
- G(\bell, \br, I^+)
\end{equation}
is analytic for $\br\in \R_<\{k,t_\tau\}$, and can be analytically extended to $\C^k$.
Thus if it were not identically $0$, by taking the first (in the dictionary order of the exponents) nonzero term in its Taylor expansion, there would exist $\beta_1,\cdots, \beta_k \in \Z_{\ge 0}$, such that for any $s_1,\cdots, s_k>0$, we have
\[
\prod_{i=1}^{k} m_i^{\beta_i}
(G(\bell, \{\frac{s_i}{m_i}\}_{i=1}^{k}, I)
- G(\bell, \{\frac{s_i}{m_i}\}_{i=1}^{k}, I^+)) \to D \prod_{i=1}^{k} s_i^{\beta_i},
\]
where $D\neq 0$ is a constant; and that the absolute value of the left-hand side is bounded by some constant, uniformly in all $m_i$ and $s_i$.
Using this and Lemma \ref{lem:asy-m-t}, by dominated convergence theorem we get a contradiction with \eqref{eq:diff-cov} (since $\alpha_1, \cdots, \alpha_k \in \Z_{\ge 0}$ are arbitrarily taken).
Thus we have that \eqref{eq:rmap} must be $0$, i.e., 
\eqref{eq:equal-pe} holds for any $\br\in \R_<\{k,t_\tau\}$, which is precisely the conclusion we want.
\end{proof}
In the next two subsections, we prove Lemmas \ref{lem:npr}, \ref{lem:pr}, \ref{lem:asy-m-t}, via making rigorous the heuristic explanations following each lemma.

\subsection{Bounds on transition probabilities from cutting information}
In this subsection, we prove Lemma \ref{lem:npr} and \ref{lem:pr}. We start with the pre-limit bounds.
The first one is a straightforward approximation by independent jumps to the right.

Recall that $\cE$ is the `high-speed event' defined in Section \ref{ssec:hywfc}, using some fixed $\bell =\{\ell_i\}_{i=1}^{k}\in \cL\{k\}$.
Also recall the vectors $\bR^{k'}=\{R_i\}_{i=1}^{k'}$, $\bL^k=\{L_i\}_{i=1}^{k'}$ for each $1\le k'\le C_*$, and that the cutting information is given by $\{\don[R_i \le t_\tau]R_i\}_{i=1}^{C_*}, \{\don[R_i \le t_\tau]L_i\}_{i=1}^{C_*}$.
\begin{lemma}  \label{lem:bd-tfc}
Take any $\bell'=\{\ell_i'\}_{i=1}^{C_*}\in\cL\{C_*\}$ (which is a permutation of $1, \cdots, C_*$), and let $\kappa[j]$ be the number such that $\ell_{\kappa[j]}'=\ell_j$ for any $1\le j \le k$.
Take any $\br=\{r_i\}_{i=1}^{C_*}\in\R_<\{C_*,\infty\}$.
If each $r_{\kappa[j]}<t_*$, and $\lambda$ is large enough (depending on $t_*$ and $A_*+B_*$), we have
\[
\PP[\cE \mid \bR^{C_*}=\br, \bL^{C_*}=\bell']
< \prod_{j=1}^k 2e^{-(t_*-r_{\kappa[j]})} \frac{(t_*-r_{\kappa[j]})^{m_j-(A_*+B_*+1-\kappa[j])}}{(m_j-(A_*+B_*+1-\kappa[j]))!}.
\]
If $r_{\kappa[j]}\ge t_*$ for some $1\le j\le k$ then the left-hand side equals zero.
\end{lemma}
\begin{proof}
The second part (where $r_{\kappa[j]}\ge t_*$) is obvious from the definition of $\cE$, so we focus on the first part; i.e., we assume that $r_{\kappa[j]}<t_*$ for the rest of this proof.
In discussion below we also assume the event that $\bR^{C_*}=\br, \bL^{C_*}=\bell'$.

For each $1\le j \le k$, and $i'\in\N$, we let $w_{j,i'}$ be the time when the particle colored $\ell_{\kappa[j]}'=\ell_j$ makes the $i'$-th jump to the right since time $r_{\kappa[j]}$.
Note that here we simply ignore jumps to the left.
We also denote $w_{j,0}=r_{\kappa[j]}$.
Denote $E_{j,i'}$ as the total amount of time, between times $w_{j,i'-1}$ and $w_{j,i'}$, such that the particle right next to the particle colored $\ell_{\kappa[j]}'=\ell_j$ has a larger color.
In other words, $E_{j,i'}$ is the amount of time when the particle colored $\ell_j$ is able to jump to the right, between times $w_{j,i'-1}$ and $w_{j,i'}$.
It can also be understood as the `waiting time' for the particle colored $\ell_j$ to make the $i'$-th jump.
Then we have that $E_{j,i'}$ for all $1\le j \le k$ and $i'\in\N$ are i.i.d. $\Exp(1)$.

Under the event $\cE$ we must have that
\begin{equation}  \label{eq:Ejif}
\sum_{i'=1}^{m_j-(A_*+B_*+1-\kappa[j])} E_{j,i'} \le t_*-r_{\kappa[j]},    
\end{equation}
for each $1\le j \le k$. 
We note that the probability of \eqref{eq:Ejif} is bounded by \[2e^{-(t_*-r_{\kappa[j]})} \frac{(t_*-r_{\kappa[j]})^{m_j-(A_*+B_*+1-\kappa[j])}}{(m_j-(A_*+B_*+1-\kappa[j]))!},\] when $m_j-(A_*+B_*+1-\kappa[j])$ is large enough.
Therefore  the conclusion follows by taking the product over $j$.
\end{proof}
The next pre-limit bound is on the difference between the transition probabilities, for different sets of cutting information.
\begin{lemma}  \label{lem:bd-dtfc}
Take any $\bell'=\{\ell_i'\}_{i=1}^{C_*}, \obell'=\{\ol_i'\}_{i=1}^{C_*}\in\cL\{C_*\}$, and any $\br=\{r_i\}_{i=1}^{C_*}, \obr=\{\oor_i\}_{i=1}^{C_*}\in\R_<\{C_*,\infty\}$.
Let $k\le k'\le C_*$, and suppose that $\ell_i'=\ol_i'$, $r_i=\oor_i$ for each $1\le i \le k'$, and $r_{k'}<t_\tau$, and $r_{k'+1}, \oor_{k'+1}>t_\tau$ if $k'<C_*$.
Let $\kappa[j]$ be the number such that $\ell_{\kappa[j]}'=\ell_j$ for any $1\le j \le k$, and we also assume that each $\kappa[j]\le k'$ (i.e., the colors in $\bell$ are contained in $\{\ell_i'\}_{i=1}^{k'}$).
Then we have
\begin{multline*}
\left|\PP[\cE \mid \bR^{C_*}=\br, \bL^{C_*}=\bell']
-
\PP[\cE \mid \bR^{C_*}=\obr, \bL^{C_*}=\obell']\right|\\
< H \left(\sum_{j=1}^k m_j \left(\frac{t_*-t_\tau}{t_*-r_{\kappa[j]}}\right)^{m_j} \right) \prod_{j=1}^k  \frac{(t_*-r_{\kappa[j]})^{m_j}}{(m_j-(A_*+B_*+1-\kappa[j]))!},
\end{multline*}
if $\lambda>H$.
Here $H$ is a large constant depending on $A_*, B_*, C_*, t_*, t_\tau$.
\end{lemma}
The general idea of this lemma is that, for two sets of cutting information that differ only after time $t_\tau$, the transition probabilities are close. This is because the particles jump fast, and it is unlikely for a particle joining after time $t_\tau$ to catch up with and affect the evolution of those particles to its right.
\begin{proof}[Proof of Lemma \ref{lem:bd-dtfc}]
We take $\oeta$ as a copy of $\eta$, and we consider the evolution of $\eta$ on \[\Omega=\bigcup_{1\le i\le C_*}\{(x,t): x\ge A_*+B_*+1-i, t \ge r_i\},\]
conditional on $\bR^{C_*}=\br, \bL^{C_*}=\bell'$;
and the evolution of $\oeta$ on \[\overline{\Omega}=\bigcup_{1\le i\le C_*}\{(x,t): x\ge A_*+B_*+1-i, t \ge \oor_i\},\]
conditional on $\overline{\bR}^{C_*}=\obr, \overline{\bL}^{C_*}=\obell'$.
Here $\overline{\bR}^{C_*}$ and $\overline{\bL}^{C_*}$ are the $\oeta$ version of $\bR^{C_*}$ and $\bL^{C_*}$.
We couple $\eta$ and $\oeta$ by using the same Poisson clocks on $\Omega \cap \overline{\Omega}$.

We consider an additional particle $p_*:[t_\tau,\infty) \to \Z$, such that $p_*(t_\tau)=A_*+B_*-k'$, and it jumps to the right according to the same Poisson clocks (note that $(A_*+B_*+1-k',t_\tau)\in \Omega \cap \overline{\Omega}$), until $T_*:=\inf\{t\ge t_\tau: \eta_t(p_*(t)) \in \bell\}$, the time when $p_*$ catches up with any particle with color in $\bell$.
After time $T_*$, we let $p_*$ continue to jump to the right with rate $1$, but use independent Poisson clocks.

The purpose of introducing $p_*$ is that, from the coupling using the same Poisson clocks, at any time $t\in [t_\tau, T_*]$ in either $\eta_t$ or $\oeta_t$ there is no particle to the right of $p_*(t)$ with color in $\{\ell_i'\}_{i=k'+1}^{C_*}$ or $\{\ol_i'\}_{i=k'+1}^{C_*}$.
Then we have that $\eta_t$ and $\oeta_t$ are the same on $[p_*(t)+1, \infty)\cap \Z$.
Thus if $T_*\ge t_*$, the event $\cE$ holds for $\eta$ if and only if it holds for $\oeta$.
With symmetry between $\eta$ and $\oeta$, it now suffices to upper bound
\[
\PP[\cE, T_* < t_* \mid \bR^{C_*}=\br, \bL^{C_*}=\bell'].
\]

For each $i'\in\N$, we let $E_{i'}$ be the waiting time between the $i'-1$-th jump and the $i'$-th jump of the particle $p_*$.
For each $1\le j \le k$,
we again take $E_{j,i'}$ as in the proof of Lemma \ref{lem:bd-tfc}; i.e., it is the waiting time for the $i'$-th jump to the right since time $r_{\kappa[j]}$, for the particle colored $\ell_{\kappa[j]}'=\ell_j$.
Then we have that they are all independent $\Exp(1)$ random variables, since before $T_*$ the particle $p_*$ is always to the left of the particles with colors in $\bell$.

The events $\cE$ and $T_* < t_*$ imply the following:
\begin{enumerate}
    \item 
\textbf{Particle with color $\ell_j$ reaches $m_j$:} for each $1\le j \le k$,
\[
\sum_{i'=1}^{m_j-(A_*+B_*+1-\kappa[j])} E_{j,i'} \le t_*-r_{\kappa[j]};
\]
\item 
\textbf{Some particle is caught up by $p_*$:} for some $1\le \oj \le k$ and some $1\le \oi \le m_{\oj}-(A_*+B_*+1-\kappa[\oj])$, $\gamma\in\Z$, $0\le \gamma < C_*$, there is
\[
\sum_{i'=1}^{\oi-\kappa[\oj]+k'} E_{i'} + \sum_{i'=\oi+\gamma}^{m_{\oj}-(A_*+B_*+1-\kappa[\oj])+\gamma} E_{\oj,i'} \le t_*-t_\tau.
\]
Here $\gamma$ corresponds to the number of left steps the particle colored $\ell_{\kappa[\oj]}'=\ell_{\oj}$ has taken before $T_*$, and $A_*+B_*+\oi-\kappa[\oj]$ corresponds to $p_*(T_*)$.
\end{enumerate}
When $\lambda$ is large enough, the probability of the first event (for each $1\le j \le k$) is bounded by
\begin{equation}  \label{eq:fevbd}
2e^{-(t_*-r_{\kappa[j]})} \frac{(t_*-r_{\kappa[j]})^{m_j-(A_*+B_*+1-\kappa[j])}}{(m_j-(A_*+B_*+1-\kappa[j]))!}.    
\end{equation}
For the second event, for fixed $\oj, \oi, \gamma$, when $\lambda$ is large enough the probability is bounded by
\begin{equation}  \label{eq:ttmbd}
2e^{-(t_*-t_\tau)} \frac{(t_*-t_\tau)^{ m_{\oj}-A_*-B_*+k' }}{(m_{\oj}-A_*-B_*+k')!}.    
\end{equation}
Below we use $H$ to denote a large constant depending on $A_*, B_*, C_*, t_*, t_\tau$, and the value can change from line to line.
For \eqref{eq:ttmbd}, we first sum over $1\le \gamma < C_*$, and get an upper bound of $H \frac{(t_*-t_\tau)^{ m_{\oj}-A_*-B_*+k' }}{(m_{\oj}-A_*-B_*+k')!}$; then by summing over $\oi$ we get an upper bound of
\[
Hm_{\oj} \frac{(t_*-t_\tau)^{ m_{\oj}-A_*-B_*+k' }}{(m_{\oj}-A_*-B_*+k')!}
<
Hm_{\oj} \frac{(t_*-t_\tau)^{ m_{\oj}-(A_*+B_*+1-\kappa[\oj]) }}{(m_{\oj}-(A_*+B_*+1-\kappa[\oj]))!}.
\]
Then with the bound \eqref{eq:fevbd} for the probabilities of the first event for each $j\neq \oj$, by summing over $\oj$ the conclusion follows.
\end{proof}
To deduce Lemma \ref{lem:npr} and \ref{lem:pr} from the above two lemmas, we need one additional result, on the relation between $\cE$, $\bR^{C_*}, \bL^{C_*}$, and $I$ or $I^+$.
\begin{lemma} \label{lem:cut-e-i}
Take any $\bell'\in\cL\{C_*\}$, $\hat{I}=I$ or $I^+$, and any $\br=\{r_i\}_{i=1}^{C_*}\in\R_<\{C_*,\infty\}$.
If $\lambda$ is large enough (depending on $t_*$ and $A_*+B_*$), we have
\[
\PP[\cE \mid \bR^{C_*}=\br, \bL^{C_*}=\bell', \hat{I}]
=\PP[\cE \mid \bR^{C_*}=\br, \bL^{C_*}=\bell'].
\]
\end{lemma}
\begin{proof}
This proof follows essentially the same ideas as the (feasible cutting information case of) the proof of Proposition \ref{prop:equal-condi-cut}.

We denote $\cE_-$ as the event where $\bR^{C_*}=\br, \bL^{C_*}=\bell'$.
We will show that conditional on $\cE_-$, the events $\hat{I}$ and $\cE$ are independent.

We recall that $\Pi$ is the Poisson field on $\Z\times [0,\infty)$, where for each $x\in\Z$ and any $0\le a < b$, $\Pi(\{x\}\times [a,b])$ is the number of times that the clock on the edge $(x-1, x)$ rings, in the time interval $[a, b]$.
Also recall that for any $A\in\Z$ and $B, C\in\N$, we denote $\cP[A, B, C]=(A+B+1-C, T^A_{B,C})$.

Denote
\[
U^-:=\{(x,t):x\le A_*+B_*-C_*\} \cup \bigcup_{i=1}^{C_*}\{ (x,t): 0\le t< r_i, x\le A_*+B_*+1-i \},
\]
and $U^+ = \Z\times [0,\infty) \setminus U^-$.
By Lemma \ref{lem:I-deter-gen}, we have that the event $\cE_-$ is measurable with respect to $\Pi$ on $U_-$, so independent of $\Pi$ on $U_+$.

We then claim that conditional on $\cE_-$, the event $\hat{I}$ is determined by $\Pi$ on $U^-$.
Indeed, given $\cE_-$, for any $\tau<i\le g$ and $1\le j \le k_i$, we have that $\cP[A_{i,j},B_{i,j},i'] \in U^-$ and $\cP[A_{i,j}^+,B_{i,j},i'] \in U^-$ for any $1\le i'\le C_{i,j}$.
Thus by Lemma \ref{lem:I-deter-gen}, conditional on $\cE_-$, $T^{A_{i,j}}_{B_{i,j},C_{i,j}}$ and $T^{A_{i,j}^+}_{B_{i,j},C_{i,j}}$ are determined by $\Pi$ on $U^-$.

We next show that conditional on $\cE_-$, the event $\cE$ is measurable with respect to $\Pi$ on $U^+$, when $\lambda$ is large enough.
Recall from the proof of Lemma \ref{lem:equal-p1}, that the event $\cE$ is determined by $I^{t_*}[A,B,C]$, for all $A\in \Z$ and $B, C \in \N$, such that
$A\le A_*$, $A-C\ge A_*-C_*$, and $A+B+1-C\ge m_k$.
By taking $\lambda$ large enough we have that $A+B+1-C \ge A_*+B_*$.
We next show that for such $A,B,C$, the event $I^{t_*}[A,B,C]$ is determined by $\Pi$ on $U^+$, conditional on $\cE_-$.

Indeed, let $\kappa[1] < \cdots < \kappa[C]$ be the first $C$ numbers in the set $\{i: 1\le i \le C_*, \ell'_i\le C_*+A-A_* \}$.
Such numbers exist because $|\{i: 1\le i \le C_*, \ell'_i\le C_*+A-A_* \}| = C_*+A-A_* \ge C$.
For each $1\le i \le C$, we take $B'_i \in \N$ such that $A+B'_i+1-i=A_*+B_*+1-\kappa[i]$.
Then $T^A_{B'_i,i}$ is precisely 
\[
\inf\{t>0: |\{x\in\Z: x\ge A_*+B_*+1-\kappa[i], \eta_t(x)\le  C_*+A-A_*\}| \ge i'\},
\]
according to Lemma \ref{lem:proj-info}.
Under $\cE_-$ this equals $r_{\kappa[i]}$, so $T^{A}_{B'_i,i} = r_{\kappa[i]}$.
Then by Lemma \ref{lem:I-deter-bdy}, conditional on $\cE_-$, the event $I^{t_*}[A,B,C]$ is determined by $\Pi$ on $U^+$.

By considering all such $A, B, C$, we have that $\cE$ is determined by $\Pi$ on $U^+$, conditional on $\cE_-$.
Thus we conclude that the events $\hat{I}$ and $\cE$ are independent conditional on $\cE_-$, and that the conclusion follows.
\end{proof}

We can now prove Lemma \ref{lem:npr} using Lemmas \ref{lem:bd-tfc} and \ref{lem:cut-e-i}.
\begin{proof} [Proof of Lemma \ref{lem:npr}]
We write $\bell'=\{\ell_i'\}_{i=1}^{k'}$.
For each $1\le j \le k$, let $\kappa[j]$ be the number such that $\ell_{\kappa[j]}'=\ell_j$, if $\ell_j\in \bell'$, and $\kappa[j]=0$ otherwise.
By Lemmas \ref{lem:bd-tfc} and \ref{lem:cut-e-i}, for any $\br\in \R_<\{k', t_\tau\}$ we have
\[
F(\bell', \br, \hat{I})
<
H\prod_{j:1\le j \le k, \kappa[j]>0}  \frac{t_*^{m_j}}{(m_j-A_*-B_*-1+\kappa[j])!}
\prod_{j:1\le j \le k, \kappa[j]=0}  \frac{(t_*-t_\tau)^{m_j}}{(m_j-A_*-B_*+k')!}.
\]
where $H$ is a constant depending on $t_*, t_\tau, A_*, B_*, C_*$.

We next integrate over all $\br\in \R_<\{k', t_\tau\}$, and multiply $\prod_{j=1}^{k} \frac{m_j^{\alpha_j}\cdot m_j!}{t_*^{m_j}}$; then we get that
\begin{align*}
&\prod_{j=1}^{k} \frac{m_j^{\alpha_j}\cdot m_j!}{t_*^{m_j}}
\int_{\br\in \R_<\{k', t_\tau\}}
F(\bell', \br, \hat{I})d\br \\
<&
Ht_\tau^{k'}\prod_{j:1\le j \le k, \kappa[j]>0}  \frac{m_j^{\alpha_j}\cdot m_j!}{(m_j-A_*-B_*-1+\kappa[j])!}
\prod_{j:1\le j \le k, \kappa[j]=0}  \frac{m_j^{\alpha_j}\cdot m_j!(1-t_\tau/t_*)^{m_j}}{(m_j-A_*-B_*+k')!}\\
\le &
Ht_\tau^{k'}\prod_{j=1}^k m_j^{\alpha_j+|A_*+B_*|} \prod_{j:1\le j \le k, \kappa[j]=0} (1-t_\tau/t_*)^{m_j}.
\end{align*}
Assuming that there exists some $j$ such that $\kappa[j]=0$, we must have the last line $\to 0$ as $\lambda\to\infty$.

Finally, we consider the case where $\kappa[j]>0$ for each $1\le j \le k$.
In this case, we must have that $k'\ge k$.
We claim that we always have $F(\bell', \br, \hat{I}) = 0$:
otherwise, we can take a sequence of swaps starting from $\bell'$, ending with $\bell$ being the first $k$ coordinates.
Thus we have  $\bell'\preceq\bell$ from the definition of the partial ordering $\prec$, which contradicts with a condition in the statement.
Thus the conclusion holds.
\end{proof}

For Lemma \ref{lem:pr}, it is deduced from Lemmas \ref{lem:bd-dtfc} and \ref{lem:cut-e-i}.
\begin{proof} [Proof of Lemma \ref{lem:pr}]
We write $\bell'=\{\ell_i'\}_{i=1}^{k'}$.
By Lemma \ref{lem:cut-e-i}, we have that for any $\br\in \R_<\{k', t_\tau\}$,
\[
|F(\bell', \br, I) 
-F(\bell', \br, I^+) |
\le  
\sup_{\obell', \obr} \PP[\cE \mid \bR^{C_*}=\obr, \bL^{C_*}=\obell']
-
\inf_{\obell', \obr} \PP[\cE \mid \bR^{C_*}=\obr, \bL^{C_*}=\obell'],
\]
where the $\sup$ and the $\inf$ are over all $\obell'=\{\ol_i'\}_{i=1}^{C_*}\in \cL\{C_*\}$ such that $\ol_i'=\ell_i'$ for each $1\le i\le k'$, and all $\obr=\{\oor_i\}_{i=1}^{C_*}$ such that $\oor_{C_*}>\cdots >\oor_{k'+1}>t_\tau$ and $\oor_i'=r_i'$ for each $1\le i\le k'$.

Since $\bell'\preceq\bell$, for each $1\le j \le k$ we can find $\kappa[j]$ such that $\ell_{\kappa[j]}'=\ell_j$.
Then by Lemma \ref{lem:bd-dtfc}, when $\lambda>H$ we can bound the above by
\begin{align*}
&H \left(\sum_{j=1}^k m_j \left(\frac{t_*-t_\tau}{t_*-r_{\kappa[j]}}\right)^{m_j} \right) \prod_{j=1}^k  \frac{(t_*-r_{\kappa[j]})^{m_j}}{(m_j-(A_*+B_*+1-\kappa[j]))!}\\
\le &
H \left(\sum_{j=1}^k m_j (1-t_\tau/t_*)^{m_j} \right) \prod_{j=1}^k  \frac{t_*^{m_j}}{(m_j-(A_*+B_*+1-\kappa[j]))!}
\end{align*}
for $H$ being a large constant depending on $A_*, B_*, C_*, t_*, t_\tau$.
By integrating over $\br\in \R_<\{k', t_\tau\}$ and multiplying $\prod_{j=1}^{k} \frac{m_j^{\alpha_j}\cdot m_j!}{t_*^{m_j}}$, we get an upper bound of
\[
H t_\tau^{k'} \left(\sum_{j=1}^k m_j (1-t_\tau/t_*)^{m_j} \right) \prod_{j=1}^k  
m_j^{\alpha_j+|A_*+B_*|}
<
kH t_\tau^{k'} m_1 (1-t_\tau/t_*)^{m_k} \prod_{j=1}^k  
m_j^{\alpha_j+|A_*+B_*|}
,
\]
which $\to 0$ as $\lambda\to \infty$.
\end{proof}

\subsection{Asymptote of transition probability from cutting information}

This subsection is devoted to the proof of Lemma \ref{lem:asy-m-t}.
\begin{proof}[Proof of Lemma \ref{lem:asy-m-t}]
For the second statement, by Lemmas \ref{lem:bd-tfc} and \ref{lem:cut-e-i}, when $\lambda$ is large, we can bound the left-hand side of \eqref{eq:asy-m-t} by
\[
\prod_{i=1}^k 2e^{-(t_*-s_i/m_i)} \left(1-\frac{s_i}{m_it_*}\right)^{m_i-(A_*+B_*+1-i)}
<
\prod_{i=1}^k 2e^{-s_i/t_*},
\]
where the inequality is by $\left(1-\frac{s_i}{m_it_*}\right)^{m_i} \le e^{-s_i/t_*}$ and $e^{-(t_*-s_i/m_i)}  \left(1-\frac{s_i}{m_it_*}\right)^{-|A_*+B_*|}<1$ (for $\lambda$ large enough).

Below we prove the first statement.
We denote $\cD$ as the intersection of the events $\bR^{k}=\{\frac{s_i}{m_i}\}_{i=1}^{k}$, $R_{k+1}>t_\tau$, $\bL^{k}=\bell$, and $I$, for simplicity of notation.

Conditional on $\cD$, 
we recall the following setup from proofs in the previous subsection.
For each $1\le j \le k$, and $i'\in\N$, we let $w_{j,i'}$ be the time when the particle colored $\ell_{j}$ makes the $i'$-th jump to the right since $\frac{s_j}{m_j}$ (note that we ignore jumps to the left).
We also denote $w_{j,0}=\frac{s_j}{m_j}$.
Define $E_{j,i'}$ as follows: it is the time between $w_{j,i'-1}$ and $w_{j,i'}$ such that the particle right next to the particle colored $\ell_j$ has a larger color.
Then $E_{j,i'}$ is also the `waiting time' for the particle colored $\ell_j$ to make the $i'$-th jump.

We consider an additional particle $p_*:[t_\tau,\infty) \to \Z$, such that $p_*(t_\tau)=A_*+B_*-k$, and it jumps to the right according to the same Poisson clocks, until time $T_*:=\inf\{t\ge t_\tau: \eta(p_*(t)) \in \bell\}$.
We let $p_*$ continue to jump to the right with rate $1$ after time $T_*$, but use some independent Poisson clocks.
Thus before $T_*$, in $\eta$ only the particles with colors in $\bell$ are to the right of $p_*$.

For each $i'\in\N$, we let $E_{k+1,i'}$ be the waiting time between the $i'-1$-th jump and the $i'$-th jump of $p_*$.
Then all these $E_{j,i'}$, for $1\le j \le k+1$ and $i'\in\N$, are independent $\Exp(1)$ random variables. \\

\noindent\textbf{Proxy events.}
We next consider the event $\cE_1$, defined as follows.
For each $1\le j \le k$ and $1\le \gamma\le m_j-A_*-B_*-1+j$, we have
\[
\frac{s_j}{m_j} + \sum_{i'=1}^\gamma E_{j,i'} < \frac{s_{j+1}}{m_{j+1}} + \sum_{i'=1}^\gamma E_{j+1,i'} .
\]
Here and below we take $m_{k+1}=1$ and $s_{k+1}=t_\tau$.
We can then think of $\cE_1$ as the event where the particle colored $\ell_{j+1}$ does not catch up with the particle colored $\ell_j$, for each $1\le j < k$; and the particle $p_*$  does not catch up with the particle colored $\ell_k$.

We let $\cE_2$ be the event where for each $1\le j \le k$, 
\[
\sum_{i'=1}^{m_j-A_*-B_*-1+j} E_{j,i'} \le t_*-\frac{s_j}{m_j} \le \sum_{i'=1}^{m_j-A_*-B_*+j} E_{j,i'}.
\]
This event is the proxy of $\cE$, assuming $\cD$ and $\cE_1$.
We also let $\cE_2'$ be the event where for each $1\le j \le k$, 
\begin{equation}  \label{eq:ce2j}
\sum_{i'=1}^{m_j-A_*-B_*-1+j} E_{j,i'} \le t_*-\frac{s_j}{m_j} .    
\end{equation}
Then conditional on $\cD$, we have that $\cE\cap \cE_1 = \cE_2\cap \cE_1$, and $\cE, \cE_2\subset \cE_2'$.
We also note that
\[
\PP[\cE_2\mid \cD] = \prod_{j=1}^k e^{-(t_*-\frac{s_j}{m_j})}\frac{(t_*-\frac{s_j}{m_j})^{m_j-A_*-B_*-1+j}}{(m_j-A_*-B_*-1+j)!},
\]
since $\cE_2$ precisely means that for each $1\le j \le k$, for a rate $1$ Poisson point process on $[0,t_*-\frac{s_j}{m_j}]$ there are $m_j-A_*-B_*-1+j$ particles.
This implies that
\begin{equation} \label{eq:e2p-cov}
\prod_{j=1}^k \frac{(m_j-A_*-B_*-1+j)!}{t_*^{m_j-A_*-B_*-1+j}}\PP[\cE_2 \mid \cD] \to e^{-kt_*}\prod_{j=1}^k e^{-s_j/t_*},    
\end{equation}
as we send $\lambda\to \infty$.
It now suffices to bound $\PP[\cE\setminus \cE_1\mid \cD]+ \PP[\cE_2\setminus \cE_1\mid \cD]$, thus suffices to bound $\PP[\cE_2'\setminus \cE_1\mid \cD]$.\\

\noindent\textbf{Bound the difference $\PP[\cE_2'\setminus \cE_1\mid \cD]$.} For each $1\le \oj \le k$, we let $\cE^{(\oj)}$ denote the event where there exists $1\le \gamma\le m_{\oj}-A_*-B_*-1+\oj$, such that
\begin{equation}  \label{eq:ojg}
\sum_{i'=1}^{\gamma} E_{\oj+1,i'} + \sum_{i'=\gamma+1}^{m_{\oj}-A_*-B_*-1+\oj} E_{\oj,i'} \le t_*-\frac{s_{\oj+1}}{m_{\oj+1}}.
\end{equation}
We then have that $\cE_2'\setminus \cE_1\subset \bigcup_{\oj=1}^{k} \cE^{(\oj)}$.

Below we let $H$ denote a large number depending on $t_*, t_\tau, A_*, B_*, C_*$, and its value can change from line to line.
When $\lambda>H$, for given $\gamma$ the probability of \eqref{eq:ojg} is bounded by
\[
2e^{-(t_*-\frac{s_{\oj+1}}{m_{\oj+1}})}\frac{(t_*-\frac{s_{\oj+1}}{m_{\oj+1}})^{m_{\oj}-A_*-B_*-1+\oj}}{(m_{\oj}-A_*-B_*-1+\oj)!}.
\]
By summing over $\gamma$ we get
\[
\PP[\cE^{(\oj)}\mid \cD] < \frac{Hm_{\oj}(t_*-\frac{s_{\oj+1}}{m_{\oj+1}})^{m_{\oj}}}{(m_{\oj}-A_*-B_*-1+\oj)!}.
\]
We also note that \eqref{eq:ce2j} for $j<\oj$ is independent of $\cE^{(\oj)}$.
These imply that
\[
\PP[\cE_2'\setminus \cE_1 \mid \cD]
<
H\sum_{\oj=1}^k
\frac{m_{\oj}(t_*-\frac{s_{\oj+1}}{m_{\oj+1}})^{m_{\oj}}}{(m_{\oj}-A_*-B_*-1+\oj)!}\prod_{j=1}^{\oj-1} \frac{(t_*-\frac{s_j}{m_j})^{m_j}}{(m_j-A_*-B_*-1+j)!},
\]
so
\begin{multline}  \label{eq:cepf1}
\prod_{j=1}^k \frac{(m_j-A_*-B_*-1+j)!}{t_*^{m_j-A_*-B_*-1+j}}\PP[\cE_2'\setminus \cE_1 \mid \cD] \\
<
H\sum_{\oj=1}^k
\frac{m_{\oj}(t_*-\frac{s_{\oj+1}}{m_{\oj+1}})^{m_{\oj}}}{t_*^{m_{\oj}}}\prod_{j=1}^{\oj-1} \frac{(t_*-\frac{s_j}{m_j})^{m_j}}{t_*^{m_j}}
\prod_{j=\oj+1}^k \frac{(m_j-A_*-B_*-1+j)!}{t_*^{m_j}}.
\end{multline}
For the $\oj$-th summand in the right-hand side, as we send $\lambda\to\infty$, we get
\begin{equation}   \label{eq:cepf2}
\frac{m_{\oj}(t_*-\frac{s_{\oj+1}}{m_{\oj+1}})^{m_{\oj}}}{t_*^{m_{\oj}}}  e^{\frac{s_{\oj+1}}{t_*}\frac{m_{\oj}}{m_{\oj+1}}} \to 0,
\end{equation}
since by taking the logarithm, we get
\[
\log(m_{\oj})+\frac{s_{\oj+1}}{t_*}\frac{m_{\oj}}{m_{\oj+1}}\left(1+\frac{t_*m_{\oj+1}}{s_{\oj+1}}\log\left(1-\frac{s_{\oj+1}}{t_*m_{\oj+1}}\right)\right)
< \log(m_{\oj})-\frac{1}{2}\frac{s_{\oj+1}}{t_*}\frac{m_{\oj}}{m_{\oj+1}}\frac{s_{\oj+1}}{t_*m_{\oj+1}}
,\]
where the inequality is by the elementary inequality of $1+x^{-1}\log(1-x) < -x/2$ (for any $0<x<1$). 
As $\lambda\to \infty$ this $\to-\infty$. 
For each $1\le j<\oj$, we have
\begin{equation}   \label{eq:cepf3}
\frac{(t_*-\frac{s_j}{m_j})^{m_j}}{t_*^{m_j}} \to e^{-s_j/t_*},
\end{equation}
and for each $\oj<j\le k$, we have
\begin{equation}   \label{eq:cepf4}
e^{-\frac{1}{k}\frac{s_{\oj+1}}{t_*}\frac{m_{\oj}}{m_{\oj+1}}}
\frac{(m_j-A_*-B_*-1+j)!}{t_*^{m_j}} \to 0
\end{equation}
Thus by \eqref{eq:cepf1}, and combining \eqref{eq:cepf2}, \eqref{eq:cepf3}, \eqref{eq:cepf4}, we conclude that
\[
\prod_{j=1}^k \frac{(m_j-A_*-B_*-1+j)!}{t_*^{m_j-A_*-B_*-1+j}}\PP[\cE_2'\setminus \cE_1 \mid \cD] \to 0,
\]
as $\lambda\to\infty$.
This with \eqref{eq:e2p-cov} implies the conclusion of Lemma \ref{lem:asy-m-t}.
\end{proof}

\section{Extensions of the shift-invariance}  \label{sec:exten}
In this section, we deduce Theorem \ref{thm:main-sa} from Theorem \ref{thm:main-de}, and finish the proofs of Corollary \ref{cor:main-geo} and Theorem \ref{thm:conv-colored}.
We then discuss some directions to extend the shift-invariance.

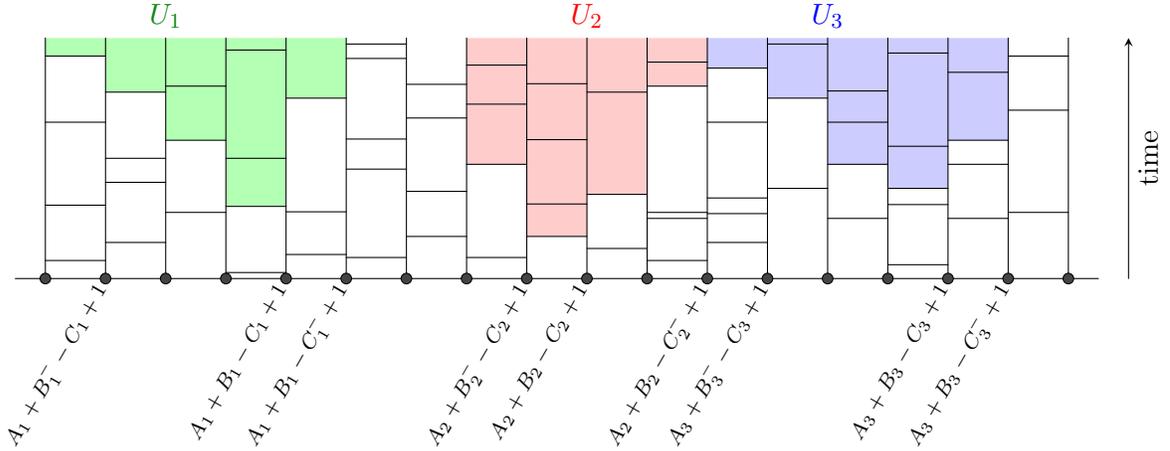
\begin{figure}[hbt!]
    \centering
\begin{tikzpicture}[line cap=round,line join=round,>=triangle 45,x=.8cm,y=.8cm]
\clip(-.5,-3.2) rectangle (19.8,4.7);

\fill[line width=0.pt,color=green,fill=green,fill opacity=0.3]
(1,4) -- (6,4) -- (6,3) -- (5,3) -- (5,1.2) -- (4,1.2) -- (4,2.3) -- (3,2.3) -- (3,3.1) -- (2,3.1) -- (2,3.7) -- (1,3.7) -- cycle;

\fill[line width=0.pt,color=red,fill=red,fill opacity=0.2]
(8,4) -- (12,4) -- (12,3.2) -- (11,3.2) -- (11,1.4) -- (10,1.4) -- (10,0.7) -- (9,0.7) -- (9,1.9) -- (8,1.9) -- cycle;

\fill[line width=0.pt,color=blue,fill=blue,fill opacity=0.2]
(12,4) -- (17,4) -- (17,2.3) -- (16,2.3) -- (16,1.5) -- (15,1.5) -- (15,1.9) -- (14,1.9) -- (14,3) -- (13,3) -- (13,3.5) -- (12,3.5) -- cycle;

\draw (0.5,0) -- (18.5,0);
\draw (1,0) -- (1,4);
\draw (2,0) -- (2,4);
\draw (3,0) -- (3,4);
\draw (4,0) -- (4,4);
\draw (5,0) -- (5,4);
\draw (6,0) -- (6,4);
\draw (7,0) -- (7,4);
\draw (8,0) -- (8,4);
\draw (9,0) -- (9,4);
\draw (10,0) -- (10,4);
\draw (11,0) -- (11,4);
\draw (12,0) -- (12,4);
\draw (13,0) -- (13,4);
\draw (14,0) -- (14,4);
\draw (15,0) -- (15,4);
\draw (16,0) -- (16,4);
\draw (17,0) -- (17,4);
\draw (18,0) -- (18,4);

\draw [-stealth] (19,0) -- (19,4);

\draw (1,3.7) -- (2,3.7);
\draw (1,2.6) -- (2,2.6);
\draw (1,1.22) -- (2,1.22);
\draw (1,0.3) -- (2,0.3);
\draw (2,3.1) -- (3,3.1);
\draw (2,2.0) -- (3,2.0);
\draw (2,1.6) -- (3,1.6);
\draw (2,0.6) -- (3,0.6);
\draw (3,3.2) -- (4,3.2);
\draw (3,2.3) -- (4,2.3);
\draw (3,1.1) -- (4,1.1);
\draw (4,3.8) -- (5,3.8);
\draw (4,2.0) -- (5,2.0);
\draw (4,1.2) -- (5,1.2);
\draw (4,0.1) -- (5,0.1);
\draw (5,3.0) -- (6,3.0);
\draw (5,1.11) -- (6,1.11);
\draw (5,0.4) -- (6,0.4);

\draw (6,3.9) -- (7,3.9);
\draw (6,3.66) -- (7,3.66);
\draw (6,2.32) -- (7,2.32);
\draw (6,1.82) -- (7,1.82);
\draw (6,0.35) -- (7,0.35);

\draw (7,3.23) -- (8,3.23);
\draw (7,2.67) -- (8,2.67);
\draw (7,1.45) -- (8,1.45);
\draw (7,0.7) -- (8,0.7);

\draw (8,3.55) -- (9,3.55);
\draw (8,2.9) -- (9,2.9);
\draw (8,1.9) -- (9,1.9);
\draw (8,0.35) -- (9,0.35);

\draw (9,3.24) -- (10,3.24);
\draw (9,2.31) -- (10,2.31);
\draw (9,1.24) -- (10,1.24);
\draw (9,0.7) -- (10,0.7);

\draw (10,3.1) -- (11,3.1);
\draw (10,1.4) -- (11,1.4);
\draw (10,0.5) -- (11,0.5);

\draw (11,3.2) -- (12,3.2);
\draw (11,3.6) -- (12,3.6);
\draw (11,1.1) -- (12,1.1);
\draw (11,1.0) -- (12,1.0);
\draw (11,0.3) -- (12,0.3);

\draw (12,3.5) -- (13,3.5);
\draw (12,2.6) -- (13,2.6);
\draw (12,1.34) -- (13,1.34);
\draw (12,1.08) -- (13,1.08);
\draw (12,0.6) -- (13,0.6);

\draw (13,3.9) -- (14,3.9);
\draw (13,3.) -- (14,3.);
\draw (13,1.5) -- (14,1.5);

\draw (14,3.12) -- (15,3.12);
\draw (14,2.6) -- (15,2.6);
\draw (14,1.9) -- (15,1.9);
\draw (14,1.0) -- (15,1.0);

\draw (15,3.75) -- (16,3.75);
\draw (15,2.2) -- (16,2.2);
\draw (15,1.5) -- (16,1.5);
\draw (15,1.23) -- (16,1.23);
\draw (15,0.23) -- (16,0.23);

\draw (16,3.43) -- (17,3.43);
\draw (16,2.3) -- (17,2.3);
\draw (16,1.9) -- (17,1.9);
\draw (16,1.0) -- (17,1.0);

\draw (17,3.7) -- (18,3.7);
\draw (17,2.8) -- (18,2.8);
\draw (17,1.1) -- (18,1.1);

\draw [fill=uuuuuu] (1,0) circle (2.0pt);
\draw [fill=uuuuuu] (2,0) circle (2.0pt);
\draw [fill=uuuuuu] (3,0) circle (2.0pt);
\draw [fill=uuuuuu] (4,0) circle (2.0pt);
\draw [fill=uuuuuu] (5,0) circle (2.0pt);
\draw [fill=uuuuuu] (6,0) circle (2.0pt);
\draw [fill=uuuuuu] (7,0) circle (2.0pt);
\draw [fill=uuuuuu] (8,0) circle (2.0pt);
\draw [fill=uuuuuu] (9,0) circle (2.0pt);
\draw [fill=uuuuuu] (10,0) circle (2.0pt);
\draw [fill=uuuuuu] (11,0) circle (2.0pt);
\draw [fill=uuuuuu] (12,0) circle (2.0pt);
\draw [fill=uuuuuu] (13,0) circle (2.0pt);
\draw [fill=uuuuuu] (14,0) circle (2.0pt);
\draw [fill=uuuuuu] (15,0) circle (2.0pt);
\draw [fill=uuuuuu] (16,0) circle (2.0pt);
\draw [fill=uuuuuu] (17,0) circle (2.0pt);
\draw [fill=uuuuuu] (18,0) circle (2.0pt);

\draw (19,2) node[anchor=north,rotate=90]{time};

\begin{scriptsize}
\draw (2,0) node[anchor=east,rotate=60]{$A_1+B_1^--C_1+1$};
\draw (5,0) node[anchor=east,rotate=60]{$A_1+B_1-C_1+1$};
\draw (6,0) node[anchor=east,rotate=60]{$A_1+B_1-C_1^-+1$};

\draw (9,0) node[anchor=east,rotate=60]{$A_2+B_2^--C_2+1$};
\draw (10,0) node[anchor=east,rotate=60]{$A_2+B_2-C_2+1$};
\draw (12,0) node[anchor=east,rotate=60]{$A_2+B_2-C_2^-+1$};

\draw (13,0) node[anchor=east,rotate=60]{$A_3+B_3^--C_3+1$};
\draw (16,0) node[anchor=east,rotate=60]{$A_3+B_3-C_3+1$};
\draw (17,0) node[anchor=east,rotate=60]{$A_3+B_3-C_3^-+1$};
\end{scriptsize}

\draw (3,4) node[anchor=south, color=darkgreen]{$U_1$};
\draw (10,4) node[anchor=south, color=red]{$U_2$};
\draw (14,4) node[anchor=south, color=blue]{$U_3$};

\end{tikzpicture}
\caption{
An illustration of the Poisson field $\Pi$ and the areas $U_i$.
Each horizontal segment indicates a particle in $\Pi$.}  
\label{fig:pce}
\end{figure}

\begin{proof}[Proof of Theorem \ref{thm:main-sa}]
To make the arguments clearer, we start with some assumptions of the parameters in the statement of this theorem.
\begin{itemize}
    \item Without loss of generality, we assume that each $V_i\neq\emptyset$, since otherwise we just ignore it.
    \item By symmetry we assume that $\cR^{A_i}_{B_i,C_i} \le \cR^{A_{i+1}}_{B_{i+1},C_{i+1}}$ and $\cR^{A_i'}_{B_i,C_i} \le \cR^{A_{i+1}'}_{B_{i+1},C_{i+1}}$ for any $1\le i < g$.
    Therefore, for any $1\le i < g$, we have $\cR^{A_i}_{B,C} \le \cR^{A_{i+1}}_{B_{i+1},C_{i+1}}$ and $\cR^{A_i'}_{B,C} \le \cR^{A_{i+1}'}_{B_{i+1},C_{i+1}}$, for all $(B,C)\in [1,B_i]\times[1,C_i]\cap\Z^2$; and  $\cR^{A_i}_{B_i,C_i} \le \cR^{A_{i+1}}_{B,C}$ and $\cR^{A_i'}_{B_i,C_i} \le \cR^{A_{i+1}'}_{B,C}$, for all $(B,C)\in [1,B_{i+1}]\times[1,C_{i+1}]\cap\Z^2$.
    \item For any $1\le i<g$, we also assume that either $A_i<A_{i+1}$ or $A_i'<A_{i+1}'$, i.e., we do not have both $A_i=A_{i+1}$ and $A_i'=A_{i+1}'$. This is because, otherwise we can combine $\{T_{B,C}^{A_i}\}_{(B,C)\in V_i}$ and $\{T_{B,C}^{A_{i+1}}\}_{(B,C)\in V_{i+1}}$,
and combine $\{T_{B,C}^{A_i'}\}_{(B,C)\in V_i}$ and $\{T_{B,C}^{A_{i+1}'}\}_{(B,C)\in V_{i+1}}$.
\end{itemize}
The first observation we make (as can be seen from Figure \ref{fig:thm2}) is that each set $V_i$ is actually a rectangle; namely, there is some $B_i^-, C_i^-\in\N$, such that $V_i=[B_i^-, B_i]\times [C_i^-, C_i] \cap\Z^2$.

Take $\oV_i=\{(B_i^-,C):C_i^-\le C \le C_i\} \cup \{(B,C_i^-):B_i^-\le B \le B_i\}$, the lower and left boundary of $V_i$.
To reduce notation, from now on we denote
$\cbT=\{\cbT_i\}_{i=1}^g=\{\{T^{A_i}_{B,C}\}_{(B,C)\in V_i\setminus\oV_i}\}_{i=1}^g$
and $\obT=\{\{T^{A_i}_{B,C}\}_{(B,C)\in \oV_i}\}_{i=1}^g$.
Likewise, we let $\cbT'=\{\cbT_i'\}_{i=1}^g$
and $\obT'$ denote the same random variables, with each $A_i$ replaced by $A_i'$.
By the orderings given by the second assumption above, we can apply Theorem \ref{thm:main-de} to conclude that
$\obT$ has the same distribution as  $\obT'$.
The goal of the rest of this proof is to upgrade this to the conclusion of this theorem, i.e., to show that $\cbT$ and $\obT$ have the same joint distribution as $\cbT'$ and $\obT'$.

Toward this goal, we next study the law of $\cbT$ conditional on $\obT$.

Recall the setup of the Poisson field $\Pi$ on $\Z\times [0,\infty)$, where for each $x\in\Z$ and any $0\le a < b$, $\Pi(\{x\}\times [a,b])$ is the number of times that the clock on the edge $(x-1, x)$ rings, in the time interval $[a, b]$.
Also recall that for any $A\in\Z$ and $B, C\in\N$, we denote $\cP[A, B, C]=(A+B+1-C, T^A_{B,C})$.

We consider $\Pi$ in different areas, as follows.
We take a family of positive real numbers \[\obt=\{\{t^i_{B,C}\}_{(B,C)\in\oV_i}\}_{i=1}^g,\] satisfying the following conditions:
for any $C^-_i\le C < C_i$ we have $t^i_{B^-_i,C}< t^i_{B^-_i,C+1}$; for any $B^-_i\le B < B_i$ we have $t^i_{B,C^-_i}< t^i_{B+1,C^-_i}$.
Let $U_i\subset \Z\times [0,\infty)$ be the following set
\[
\left(\bigcup_{C^-_i<C\le C_i}\{(A_i+B^-_i-C+1,t): t > t^i_{B,C^-_i}\}\right)\cup \left(\bigcup_{B^-_i<B\le B_i}\{(A_i+B-C^-_i+1,t): t > t^i_{B^-_i,C}\}\right),
\]
and let $U=\bigcup_{i=1}^g U_i$. See Figure \ref{fig:pce}.
Similarly, we take $U_i'=\{(x+A_i'-A_i,t):(x,t)\in U_i\}$ and $U'=\bigcup_{i=1}^g U_i'$.

We next show the following:
\begin{enumerate}
    \item[(i)] These $U_i$ are mutually disjoint.
    \item[(ii)] The event $\obT=\obt$ is determined by $\Pi$ on $\Z\times [0,\infty) \setminus U$.
    \item[(iii)] Given $\obT=\obt$, each $\cbT_i$ is determined by $\Pi$ on $U_i$.
\end{enumerate}
For (i), it is implied by the fact that $A_i+B_i^--C_i > A_{i+1}+B_{i+1}-C_{i+1}^-$ for any $1\le i < g$, and we now prove this fact.
Indeed, we have 
\[
A_i\le A_{i+1},\; A_i+B_i^- \ge A_{i+1}+B_{i+1},\; A_i-C_i\ge A_{i+1}-C_{i+1}^-,
\]
since $\cR^{A_i}_{B_i^-,C_i^-} \le \cR^{A_{i+1}}_{B_{i+1},C_{i+1}}$ and $\cR^{A_i}_{B_i,C_i} \le \cR^{A_{i+1}}_{B_{i+1}^-,C_{i+1}^-}$ (by the second assumption above). 
Similarly, we have
\[
A_i'+B_i^- \ge A_{i+1}'+B_{i+1},\; A_i'\le A_{i+1}'.
\]
If $A_i+B_i^--C_i > A_{i+1}+B_{i+1}-C_{i+1}^-$ does not hold, we must have that $A_i=A_{i+1}$ and $B_i^-=B_{i+1}$, then $A_i'=A_{i+1}'$. Therefore we get a contradiction with the third assumption above.

We then prove (ii).
The strategy is to show that the event $\obT=\obt$ implies that 
\begin{equation}  \label{eq:BCp}
\cP[A_i,B',C'] \in \Z\times [0,\infty) \setminus U,    
\end{equation}
for each $i$ and any $(B', C') \in [1, B_i]\times [1, C_i] \cap\Z^2 \setminus V_i$.
Then we can apply Lemma \ref{lem:I-deter-gen} and get (ii).

Take any $1\le B'<B_i^-$. 
From (i) and the fact proved there, we have that $\cP[A_i,B',C_i]$ is disjoint from $U_j$ for $j\le i$, since $A_i+B'-C_i+1<A_j+B_j^--C_j+1$.
We now consider $j>i$.
If $A_i+B'-C_i=A_j+B-C$ for some $(B, C)\in\oV_j$, we claim that there must be $T^{A_i}_{B',C_i} \le T^{A_j}_{B,C}$.
Indeed, in $\opi^{A_j}_{T^{A_j}_{B, C}}$, there are precisely $C$ particles on or to the right of $A_j+B-C+1$.
Since $\opi^{A_j}_{T^{A_j}_{B, C}}$ and $\opi^{A_i}_{T^{A_j}_{B, C}}$ differ by $A_j-A_i$ particles, in $\opi^{A_i}_{T^{A_j}_{B, C}}$ there are at least $C+A_i-A_j$ particles on or to the right of $A_j+B-C+1= A_i+B'-C_i+1$.
Since $C+A_i-A_j\ge C_j^- +A_i-A_j \ge C_i$ (where the second inequality is due to the second assumption above), we have that $T^{A_i}_{B',C_i} \le T^{A_j}_{B,C}$.
Therefore $\cP[A_i,B',C_i]$ is disjoint from $U_j$ for each $j$, thus disjoint from $U$.
Similarly, for any $1\le C'<C_i^-$, we have that $\cP[A_i,B_i,C']$ is disjoint from $U$.
From these we conclude that \eqref{eq:BCp} holds for any $(B', C') \in [1, B_i]\times [1, C_i] \cap\Z^2 \setminus V_i$, so (ii) holds by Lemma \ref{lem:I-deter-gen}.

For (iii), the passage time $T^{A_i}_{B,C}$ for each $(B,C)\in{V_i\setminus\oV_i}$ is the smallest $t > T^{A_i}_{B-1,C}\vee T^{A_i}_{B,C+1}$, such that there is a point at $(t, A_i+B-C+1)$ in $\Pi$.
Thus given $\obT=\obt$, we can determine $\cbT_i$ by $\Pi$ on $U_i$, recursively.

From (i), (ii), and (iii), we have that conditional on $\obT=\obt$, the distributions of $\cbT_i$ are independent for each $1\le i\le g$.
Similarly, conditional on $\obT'=\obt$, the distributions of $\cbT_i'$ are independent for each $1\le i\le g$.
From the arguments of (iii), we also see that for each $1\le i \le g$, the distribution of $\cbT_i$ conditional on $\obT=\obt$ is the same as the distribution of $\cbT_i'$ conditional on $\obT'=\obt$.
These together imply that $\obT$ and $\cbT$ have the same joint distribution as $\obT'$ and $\cbT'$.
\end{proof}
For completeness, we give the proofs of Corollary \ref{cor:main-geo} and Theorem \ref{thm:conv-colored}.
\begin{proof}[Proof of Corollary \ref{cor:main-geo}]
We assume that $T_{B,C}^A=0$ for any $A\in\Z$ and $B,C \in \Z_{\ge 0}$, $BC=0$.

For each $1\le i \le g$, and any $(B, C) \in \Gamma_{B_i,C_i}^{A_i}$, the following results are true:
we have $(B-1, C) \in \Gamma_{B_i,C_i}^{A_i}$, if $T_{B-1,C}^{A_i} > T_{B,C-1}^{A_i}$; 
and we have $(B, C-1) \in \Gamma_{B_i,C_i}^{A_i}$, if $T_{B-1,C}^{A_i} < T_{B,C-1}^{A_i}$.
The same results are true if we replace $A_i$ with $A_i'$.
Since $\{ \{T^{A_i}_{B,C}\}_{(B,C)\in V_i} \}_{i=1}^g$ has the same distribution as $\{ \{T^{A_i'}_{B,C}\}_{(B,C)\in V_i} \}_{i=1}^g$ by Theorem \ref{thm:main-sa}, we must have that $\don[T_{B-1,C}^{A_i} < T_{B,C-1}^{A_i}]$ has the same distribution as $\don[T_{B-1,C}^{A_i'} < T_{B,C-1}^{A_i'}]$, jointly for all $1\le i \le g$ and $B,C \in W_i$. Thus the conclusion follows.
\end{proof}

\begin{proof}[Proof of Theorem \ref{thm:conv-colored}]
Recall the i.i.d. $\Exp(1)$ random variables $\{\omega^A(v)\}_{v\in\N^2}$ coupled with $\mu^A$, for each $A\in \Z$, and $L^A_{u,v}$ the passage time from $u$ to $v$ under $\omega^A$, satisfying $L^A_{(1,1),(B,C)} = T^A_{B,C}$ for any $B,C\in\N$.
For each $x, y \in \R$ and $n\in\N$, we define
\begin{align*}
\cS_n'(x,y) =& 2^{-4/3}n^{-1/3} (L^0_{(1, 1), (n+\lfloor 2^{2/3}n^{2/3}y \rfloor-\lfloor 2^{2/3}n^{2/3}x \rfloor, n-\lfloor 2^{2/3}n^{2/3}y \rfloor+\lfloor 2^{2/3}n^{2/3}x \rfloor)}-4n)
\\
=& 2^{-4/3}n^{-1/3} (T^0_{(n+\lfloor 2^{2/3}n^{2/3}y \rfloor-\lfloor 2^{2/3}n^{2/3}x \rfloor, n-\lfloor 2^{2/3}n^{2/3}y \rfloor+\lfloor 2^{2/3}n^{2/3}x \rfloor)}-4n).
\end{align*}
Without loss of generality, we assume that $\Theta$ is bounded, and $n$ is large enough.
Note that then for $(x, y) \in \Theta$, the points $(\lfloor 2^{2/3}n^{2/3}x \rfloor, \lfloor 2^{2/3}n^{2/3}y\rfloor) \in \Z^2$ take only finitely many values, with all the rectangles
$\cR^0_{(n+\lfloor 2^{2/3}n^{2/3}y \rfloor-\lfloor 2^{2/3}n^{2/3}x \rfloor, n-\lfloor 2^{2/3}n^{2/3}y \rfloor+\lfloor 2^{2/3}n^{2/3}x \rfloor)}$ being mutually ordered, and all the rectangles $\cR^{\lfloor 2^{2/3}n^{2/3}x \rfloor}_{(n+\lfloor 2^{2/3}n^{2/3}y \rfloor-\lfloor 2^{2/3}n^{2/3}x \rfloor, n-\lfloor 2^{2/3}n^{2/3}y \rfloor+\lfloor 2^{2/3}n^{2/3}x \rfloor)}$ also mutually ordered.
So by Theorem \ref{thm:main-sa}, $\cS_n'$ and $\cS_n$ have the same distribution on $\Theta$.

On the other hand, we have that $\cS_n'$ and $\cS_n^*$ on $\Theta$ are equal in distribution, by \cite[Theorem 1.2]{dauvergne2020hidden}.
Thus we conclude that $\cS_n$ and $\cS_n^*$ have the same distribution on $\Theta$.
By the convergence of $\cS_n$ to $\cS$, the conclusion follows.
\end{proof}

\subsection{Some further questions}  \label{ssec:shift-cons}
We now discuss some other possible extensions to Theorem \ref{thm:main-de} and Theorem \ref{thm:main-sa}.

A major question is whether some shift constraints in Theorem \ref{thm:main-de} and Theorem \ref{thm:main-sa} can be lifted. 
Let's consider just two passage times.
Using arguments that slightly generalize those in Example \ref{ex:simple}, or some simplification of the proof of Theorem \ref{thm:main-de}, we can get the following conditional equality in distribution.
\begin{ex}   \label{ex:condi}
Take $A_1, A_2, A_2' \in \Z$, $B_1, B_2, C_1, C_2 \in \N$, satisfying the following conditions: $A_1 \le A_2, A_2'$, and $A_1+B_1-C_1 \ge A_2+B_2-C_2, A_2'+B_2-C_2$, and $A_1+B_1 \ge A_2+B_2, A_2'+B_2$. For any $0<t_1\le t_2$, we have
\[
\PP[T^{A_1}_{B_1,C_1}<t_1, T^{A_2}_{B_2,C_2}<t_2] = \PP[T^{A_1}_{B_1,C_1}<t_1, T^{A_2'}_{B_2,C_2}<t_2].
\]
\end{ex}
We note that this statement is non-trivial only if $C_1 < C_2$, since otherwise we have that $T^{A_1}_{B_1,C_1}\le T^{A_2}_{B_2,C_2}$ and $T^{A_1}_{B_1,C_1}\le T^{A_2'}_{B_2,C_2}$.
The main difference between this example and Theorem \ref{thm:main-sa} is that here the condition $A_1-C_1\ge A_2-C_2, A_2'-C_2$ is replaced by the weaker ones $A_1+B_1-C_1 \ge A_2+B_2-C_2, A_2'+B_2-C_2$; whereas the equality in distribution only holds in half of the space $[0,\infty)^2$.

To see that the condition $t_1 \le t_2$ is necessary, we can compute the following example.
\begin{ex}
For any $t_1, t_2 > 0$, we have
\[
\PP[T^0_{3,1}=t_1, T^0_{1,2}=t_2]
=
\PP[T^0_{3,1}=t_1, T^1_{1,2}=t_2]
=\begin{cases}
e^{-t_2} - e^{-t_1-t_2}(1+t_1) ,\quad & \text{ if } t_1\le t_2 \\
e^{-t_1}(t_1-t_2+1) - e^{-t_1-t_2}(1+t_1) ,\quad & \text{ if } t_1> t_2
\end{cases},    
\]
and
\[
\PP[T^0_{3,1}=t_1, T^2_{1,2}=t_2]
=\begin{cases}
e^{-t_2} - e^{-t_1-t_2}(1+t_1) ,\quad & \text{ if } t_1\le t_2 \\
2e^{-t_1} + e^{-t_1-t_2}(t_2(t_1-t_2)^2/2-2(t_2+1)) ,\quad & \text{ if } t_1> t_2
\end{cases}.
\]
\end{ex}
From this, it seems that the constraint of rectangle ordering is needed to get equalities in distribution.
However, we expect that some conditional equalities in distribution hold with weaker constraints, and it will be interesting to find such an identity generalizing Example \ref{ex:condi} to more than two passage times.

In another direction (of lifting some shift constraints), one can ask whether the ordering condition can be imposed only on pairs of passage times where relative shift happens.
To be more precise, one can consider the following question, which is in a similar spirit as \cite[Conjecture 1.5]{borodin2019shift}.
\begin{question}
Let $g\in\N$, and take $A_i\in \Z$, $B_i, C_i \in\N$, for all $1\le i \le g$.
Let $1\le \iota < g$, and for any $1\le i\le g$ we let $A_i^+ = A_i+\don[i> \iota]$. Suppose that for any $1\le i\le \iota < i' \le g$, we have $\cR^{A_i}_{B_i,C_i} \le \cR^{A_{i'}}_{B_{i'},C_{i'}}$and $\cR^{A_{i}^+}_{B_{i},C_{i}} \le \cR^{A_{i'}^+}_{B_{i'},C_{i'}}$.
Are the vectors
$\{ T_{B_{i},C_{i}}^{A_{i}} \}_{i=1}^g$ and $\{ T_{B_{i},C_{i}}^{A_{i}^+} \}_{i=1}^g$ equal in distribution?
\end{question}
This equality in distribution (if true) implies both Theorem \ref{thm:main-de} and Theorem \ref{thm:main-sa}.
The main difference of this with Theorem \ref{thm:main-de} and Theorem \ref{thm:main-sa} is that, the rectangles  $\cR^{A_i}_{B_i,C_i}$ for $1\le i \le \iota$ or for $\iota<i\le g$ are not ordered, and these $A_i$ can be different. We hope that an answer to this question will give a better understanding of the limitation of and the mechanism behind these distributional equalities in the colored TASEP.

It is also natural to ask if some equalities in distribution can be obtained for the height function of the colored stochastic six-vertex model.
For example, can we prove some shift-invariance of the height function for more general points (than those in Theorem \ref{thm:gal-6v}), thus extending the results of \cite{borodin2019shift} and \cite{galashin2020symmetries}?
We note that the colored stochastic six-vertex model can be viewed as a discrete-time version of the colored ASEP (Asymmetric Simple Exclusion Process), which is a generalization of the colored TASEP, such that when the Poisson clock on edge $(x,x+1)$ rings, and the color at site $x$ is larger than the color at site $x+1$, we swap the particles with some probability $p\in [0,1]$.
The colored TASEP then corresponds to the colored stochastic six-vertex model with the parameter $b_2=0$.
By adapting our arguments to the discrete-time setting, and using the invertibility of the transition matrices instead of the analyticity of the transition distribution functions, we can get extensions of the height function shift-invariance of the colored stochastic six-vertex model, when $b_2=0$.
However, for the general case where $b_2>0$, 
one might need to prove some shift-invariance for the multi-time distribution of the colored ASEP with $p>0$.
Our arguments for the colored TASEP rely on the following observation: for the trajectories of the $i$-th rightmost particle in some $\opi^A$, and of the $j$-th leftmost hole in some $\opi^{A'}$, they will be independent once they cross each other; but this is not true for the colored ASEP.
Some new ideas are needed to prove (or disprove) possible extensions to the ASEP setting.

\section{Joint equality in distribution between OSP and LPP}  \label{sec:osp}

In this section, we prove Theorem \ref{thm:udv} and its implications.

Recall that for the colored TASEP, we encoded it with a family of TASEPs $\opi^A$, for $A\in \Z$, with different initial conditions and coupled using the same Poisson clocks.
We similarly encode OSP on $\{1,\ldots, N\}$, with a family of of TASEPs on $\{1,\ldots, N\}$, denoted as $\nu^{N,A}$, for $A\in \{1,\ldots, N-1\}$, such that $\nu^{N,A}_0(x) = 0$ for $x\le A$, and $\nu^{N,A}_0(x) = \infty$ for $x>A$ (as before we use $0$ to denote particles and $\infty$ to denote holes).
Then we have
\[
U_N(A) = \inf\{t>0: \nu^{N,A}_t(x) = 0, \forall N-A+1 \le x \le N; \nu^{N,A}_t(x) = \infty, \forall 1 \le x \le N-A\},
\]
i.e., $U_N(A)$ is the time for the last jump of $\nu^{N,A}$.

We use the following input from \cite{angel2009oriented}.
To state it, we consider two kinds of operators on particle-hole configurations on $\Z$ introduced there.
For each $k\in \N$, let $R_k$ be the \emph{cut-off operator}, such that for any configuration, it keeps the $k$ rightmost particles, and changes all other particles to holes (unless there are less than $k$ particles, in which case the configuration keeps the same).
For each $n\in \Z$ we let $B_n$ be the \emph{push-back operator}, which pushes all particles onto $(-\infty, n]\cap \Z$: for the $j$-th rightmost particle, it is moved to site $x\wedge n+1-j$, if it is originally at site $x$.

We couple the colored TASEP with OSP on $\{1,\ldots, N\}$, such that the Poisson clocks on $(k,k+1)$ for each $1\le k \le N-1$ are the same.
\begin{lemma}[\protect{\cite[Lemma 3.3]{angel2009oriented}}]  \label{lem:cbopera}
Under the above coupling, we have that $B_NR_A\opi^A$ on $\{1,\cdots,N\}$ is the same as $\nu^{N,A}$, for each $1\le A \le N-1$.
\end{lemma}

We can now prove the joint equality in distribution between $\bU_N$ and passage times in LPP, from Theorem \ref{thm:main-de} and using this connection.
\begin{proof}[Proof of Theorem \ref{thm:udv}]
By Lemma \ref{lem:cbopera}, 
for each $1\le A \le N-1$, if and only if $t\le T^A_{N-A,A}$, there are at least $A$ particles on or to the right of site $N-A+1$ in $\opi^A_t$, thus in $\nu^{N,A}_t$; and this is equivalent to that
$\nu^{N,A}_t(x) = 0, \forall N-A+1 \le x \le N$ and $\nu^{N,A}_t(x) = \infty, \forall 1 \le x \le N-A$.
Thus we must have that $T^A_{N-A,A} = U_N(A)$.

Via the connection between LPP and TASEP, as stated in Section \ref{sec:intro}, and symmetry between the two coordinates in LPP, we have that $\{T^1_{N-A,A}\}_{A=1}^{N-1}$ has the same distribution as $\{L_{(1,1),(A,N-A)}\}_{A=1}^{N-1}$.
It remains to show that it also has the same distribution as $\{T^A_{N-A,A}\}_{A=1}^{N-1}$.
For this we apply Theorem \ref{thm:main-de}: for each $1\le i < N-1$, we have that $\{T^{A\wedge i}_{N-A,A}\}_{A=1}^{N-1}$ has the same distribution as $\{T^{A\wedge (i+1)}_{N-A,A}\}_{A=1}^{N-1}$.
Thus $\{T^{A\wedge i}_{N-A,A}\}_{A=1}^{N-1}$ has the same distribution for each $1\le i \le N-1$.
By taking $i=1$ and $i=N-1$ we get the conclusion.
\end{proof}

We now deduce the asymptotic results from Theorem \ref{thm:udv}, using existing results of LPP.

\begin{proof}[Proof of Theorem \ref{cor:j-airy} and \ref{cor:max-loc}]
By Theorem \ref{thm:udv}, it suffices to consider the function
\begin{multline*}
x\mapsto 
\frac{(y(1-y))^{1/6}}{(1+2\sqrt{y(1-y)})^{2/3}N^{1/3}}
\\
\times (L_{(1,1),(\lfloor yN+ xN^{2/3} \rfloor,N-\lfloor yN+ xN^{2/3} \rfloor)} - (1+2\sqrt{y(1-y)})N - \frac{1-2y}{\sqrt{y(1-y)}}xN^{2/3}).
\end{multline*}
We have that it converges to $\cA_2(x)-x^2$ weakly, in the topology of uniform convergence on compact sets.
The finite-dimensional convergence follows from e.g.\ \cite{BF08,BP08}, and it can be upgraded to uniform convergence on compact sets using a tightness result from \cite{Pim17}.
See \cite[Theorem 3.8]{BGZ} for a proof in the $y=1/2$ case, which also goes through essentially verbatim for any fixed $y\in(0,1)$.

For Theorem \ref{cor:max-loc}, we note that $k_*=\argmax_{1\le k \le N-1}L_{(1,1),(k,N-k)}$.
Then we just need the following two additional facts: the process $\cA_2(x)-x^2$ attains its maximum at a unique point, and that $\frac{k_*-N/2}{N^{2/3}}$ is tight as $N\to\infty$.
The first fact is proved in \cite{CH14}.
The tightness follows from a transversal estimate in LPP, e.g.\ \cite[Proposition 4.2]{BGZ}, which implies the following: there exist $c_1,c_2>0$, such that for any $\phi$ large enough and any $N\in\N$, 
\[
\PP\left[\max_{1\le k \le N-1, |k-N/2|>\phi N^{2/3}}L_{(1,1),(k,N-k)} > 2N-c_1\phi^2N^{1/3}\right] < e^{-c_2\phi^3}.
\]
We also have that $\PP[L_{(1,1),(\lfloor N/2\rfloor,\lceil N/2\rceil)}-2N<-\phi N^{1/3}] < c^{-1}e^{-c\phi^3}$ for some constant $c>0$.
This is due to that $L_{(1,1),(\lfloor N/2\rfloor,\lceil N/2\rceil)}$ has the same distribution as the largest eigenvalue of
$X^*X$, where $X$ is an $\lfloor N/2\rfloor\times\lceil N/2\rceil$ matrix with i.i.d. standard complex Gaussian entries (see \cite[Proposition 1.4]{Jo99}); and the estimate is from \cite[Theorem 2]{LR10}.
With these bounds on the passage times, we get tightness of $\frac{k_*-N/2}{N^{2/3}}$.
\end{proof}

\begin{proof}[Proof of Theorem \ref{cor:loc-srw}]
It is proved in \cite[Theorem 2.1]{balazs2021local} that for the function \[x\mapsto L_{(1,1),(\lfloor yN \rfloor + x, N-\lfloor yN \rfloor - x)} - L_{(1,1),(\lfloor yN \rfloor, N-\lfloor yN \rfloor)},\] its total variation distance to $g_N$ decays to zero (actually it is stated for the total variation distance between the difference of the passage times, and the so-called Busemann function, which has the distribution of a two-sided random walk).
Then the conclusion immediately follows from Theorem \ref{thm:udv}.
\end{proof}

\subsection*{Acknowledgement}
The author would like to thank Jiaoyang Huang for first bringing the conjectures in \cite{bisi2020oriented} into his attention; and thank the organizers of the BIRS workshop: Permutations and Probability, and a talk by Dan Romik, for introducing some related backgrounds.
The author thanks Vadim Gorin for some very valuable comments on earlier drafts of this paper, and Duncan Dauvergne for some discussions on the relation between colored TASEP and the Airy sheet.
The author would also like to thank an anonymous referee for carefully reading this paper, and providing many very helpful suggestions on improving the expository.

\bibliographystyle{halpha}
\bibliography{bibliography}

\end{document}